\documentclass[10pt]{article}
\usepackage{amsmath,amssymb,amsthm}
\usepackage{color}
\usepackage[nohead,margin=1.0in]{geometry}
\usepackage{booktabs}
\usepackage{graphicx,caption,subcaption}
\usepackage{algorithm,algorithmic}
\usepackage{verbatim}
\usepackage{url}

\theoremstyle{plain}
\newtheorem{theorem}{Theorem}[section]
\newtheorem{remark}{Remark}[section]

\numberwithin{equation}{section}

\newcommand{\N}{\mathbb{N}}
\newcommand{\R}{\mathbb{R}}

\renewcommand{\d}{\mathrm{d}}

\title{Numerical Reconstruction in Magnetic Particle Imaging}
\author{Tobias Kluth\thanks{Center for Industrial Mathematics, University of Bremen, Bibliothekstr. 5, 28357 Bremen, Germany (\texttt{tkluth@math.uni-bremen.de})} \and
Bangti Jin\thanks{Department of Computer Science, University College London, Gower Street, London WC1E 6BT, UK (\texttt{b.jin@ucl.ac.uk, bangti.jin@gmail.com})}}

\begin{document}
\maketitle

\begin{abstract}
Magnetic particle imaging (MPI) is a medical imaging modality of recent origin, and it exploits the nonlinear
magnetization phenomenon to recover the spatially dependent concentration of the nanoparticles. Currently, image
reconstruction in MPI is frequently carried out by standard Tikhonov regularization with nonnegativity constraint,
which is then minimized by a Kaczmarz type method. In this work, we revisit several issues in the numerical
reconstruction in MPI from the perspective of modern inverse theory, i.e., the choice of data fidelity, and
choosing a suitable regularization parameter and accelerating Kaczmarz iteration via randomized singular value
decomposition. These algorithmic tricks are straightforward to implement and easy to incorporate in existing
reconstruction algorithms. Their significant potentials are illustrated by extensive numerical experiments on a
publicly available dataset.
\end{abstract}

%
%
%

\section{Introduction}
Magnetic particle imaging (MPI) is a relatively new medical imaging modality \cite{Gleich2005}. It exploits the nonlinear magnetization
behavior of ferromagnetic nanoparticles in an applied magnetic field to reconstruct the spatially dependent concentration of nanoparticles.
The experimental setup is as follows. A static magnetic field (selection field), given by a gradient field, generates a field
free point or a field free line. Its superposition with a spatially homogeneous but time-dependent field (drive field) moves
the field free point / line along a predefined trajectory defining the field-of-view. The most common trajectory is the so-called
Lissajous curve. The change of the applied field induces a
change of the nanoparticle magnetization, which can be measured and used to recover the concentration of the nanoparticles.

MPI has a number of distinct features: high data acquisition speed, high sensitivity, potentially high spatial resolution and free from the need
of harmful radiation. This makes MPI especially attractive for {\it in-vivo} applications, and the list of potential medical
applications is long and growing. The potential for imaging blood flow was demonstrated  in {\it in-vivo}
experiments using a healthy mouse \cite{weizenecker2009three}. The feasibility of a circulating tracer for long-term
monitoring was recently investigated \cite{khandhar2017evaluation}. The high temporal resolution of MPI is shown to be suitable
for potential flow estimation \cite{franke2017}, tracking medical instruments \cite{haegele2012magnetic}
and tracking and guiding instruments for angioplasty \cite{Salamon:2016cz}.
Further promising applications of MPI include cancer detection \cite{Yu2017} and cancer treatment by hyperthermia
\cite{murase2015usefulness}.

Hence, the numerical reconstruction in MPI is of enormous practical importance, and has received much attention. In the literature, there are
mainly two different groups of approaches, i.e., data-based v.s. model based, dependent of the description of the forward map. The data-based
approach employs experimentally calibrated forward operators, whereas the model-based approach employs mathematical models to
describe the physical process. Currently, the former delivers the state of art numerical reconstructions. In either case, MPI
reconstruction techniques often boil down to solving a linear inverse problem using standard regularization techniques (see, e.g.,
the monographs \cite{EnglHankeNeubauer:1996,SchusterKlatenbacher:2012,ItoJin:2015}). The most popular idea is standard Tikhonov
regularization with nonnegativity constraint, which is then minimized by Kaczmarz iteration \cite{weizenecker2009three,Knopp2010e,
RahmerHalkola:2015}. The Kaczmarz method \cite{Kaczmarz:1937} is very attractive for large volume of data, due to its low
operational complexity per iteration.  This idea was also combined with a preconditioning (row normalization) and row exclusion
to improve the image reconstruction \cite{Knopp2010e}. Only very recently, more advanced variational regularization techniques,
e.g., nonnegative fused lasso penalty \cite{Storath:2017}, total least-squares approach \cite{KluthMaass:2017}, approximation
error modeling \cite{BrandtSeppanen:2018} and deep image prior \cite{DittmerKluth:2018}, have been proposed and evaluated. The
total variation penalty allows recovering piecewise constant concentrations accurately. The approaches in \cite{KluthMaass:2017,
BrandtSeppanen:2018} allow incorporating model errors into the reconstruction process for enhanced imaging quality. We refer to
the recent survey \cite[Section 6]{KnoppGdaniec:2017} for an overview of other reconstruction methods. It is worth noting that
all these reconstruction techniques can be very expensive for three-dimensional problems, where the available datasets are of
relatively large volume. Therefore, there is a significant demand in developing fast MPI image reconstruction algorithms (possibly
with improved resolution).

There have been several important efforts \cite{Lampe_Fast_2012,Schmiester2017,Knopp_Online_2016,knopp2015local}
in accelerating MPI reconstruction. One idea is to employ sparse approximations of the linear forward operator
in predefined basis sets, achieved by first applying discrete orthonormal transformations (e.g., Fourier transform,
cosine transform or Chebyshev transform) and then thresholding small elements. The sparse approximation enables
reducing the computing times of iterative solvers (e.g., CGNE and LSQR) \cite{Lampe_Fast_2012}
or potential direct inversion techniques \cite{Schmiester2017}. This idea simultaneously provides a memory-efficient
sparse and approximate representation \cite{Lampe_Fast_2012,knopp2015local,Schmiester2017}. Alternatively, one can
reduce the dimension of the forward map using a row selection technique, based on an SNR type quality measure (see
Section \ref{ssec:freq_selection}) \cite{Knopp_Online_2016}. The speedup is achieved by dimension 
reduction in the data space.

In this work, we revisit several issues in the numerical reconstruction in MPI in the lens of modern inverse theory
(see, e.g., \cite{EnglHankeNeubauer:1996,SchusterKlatenbacher:2012,ItoJin:2015}) and contribute to the development
of robust, accurate and fast reconstruction techniques. First, we highlight the importance of noise covariance in the
reconstruction algorithm, and propose a simple whitening procedure from the perspective of maximal likelihood estimation,
leading to the standard least-squares type fidelity for the whitened problem. Second, we propose a dimension reduction
procedure in the data space to accelerate the benchmark MPI reconstruction algorithm using randomized singular value
decomposition (SVD). It exploits the inherent ill-posed nature of the MPI imaging problem, that is, the system matrix
admits a low-rank approximation, in order to reduce the effective number of equations. This step can be easily incorporated
into any existing algorithms. Third, we discuss the choice of the crucial regularization parameter and describe two popular
rules from the inverse problem community, i.e., discrepancy principle and quasi-optimality criterion. Last, we present
extensive numerical experiments on a publicly available dataset, i.e., the ``shape'' phantom from \texttt{Open MPI dataset}
(available at \url{https://www.tuhh.de/ibi/research/open-mpi-data.html}), to demonstrate the performance of the proposed
algorithmic improvements. These represents the main contributions of the work. Our findings include that the whitening step can
improve the reconstruction accuracy, the randomized SVD can accelerate the benchmark algorithm by tens of times, and
the quasi-optimality criterion is able to determine a suitable regularization parameter in a purely data-driven manner.
Thus, these techniques together may enable automated fast and accurate MPI reconstruction.

The rest of the paper is organized as follows. In Section \ref{sec:forward}, we discuss the proper formulation of the MPI
imaging problem, including system matrix calibration, frequency selection and whitening. In Section 3, we describe the classical
reconstruction method based on Kaczmarz iteration, its acceleration via randomized SVD and parameter choice rules. Then in
Section 4, we present extensive numerical results to illustrate the proposed approaches. In Section \ref{sec:conc},
we present concluding remarks and additional discussions. In an appendix, we provide an error estimate of the approximate
minimizer with the low rank approximation, so as to justify the acceleration procedure.

\section{The MPI forward map}\label{sec:forward}

The accurate mathematical modeling of MPI is still in its infancy. Several mathematical models have
been proposed; see the recent survey \cite{Kluth:2018} for an overview. Nonetheless, state of art numerical
reconstructions are achieved by experimentally calibrated forward operators, which we describe 
in Section \ref{ssec:system_matrix} below.

Mathematically, the physical process can be modeled as follows. Let $\Omega\subset\mathbb{R}^3$ be the spatial
domain occupied by the object of interest, and $c:\Omega\to \mathbb{R}_+$ be the concentration of the magnetic nanoparticles. Then
the measured voltage signal $v_\ell : I:=[0,T] \rightarrow \R$, for $\ell=1,\hdots,L$, obtained at $L\in \N$ receive coils, $0<T<\infty$, is given by
\begin{align}
 v_\ell (t) =&  \int_\Omega c(x) \underset{=s_\ell(x,t)}{\underbrace{\int_I -{a}_\ell(t-t') \mu_0 p_\ell(x)^t\dot{ \bar{m}}(x,t)  \d t' }}\d x 
  +\underset{=v_{\mathrm{E},\ell}(t)}{\underbrace{\int_I \int_{\R^3} -{a}_\ell(t-t')\mu_0 p_\ell(x)^t\dot{ H}(x,t)  \d x \d t'}}, \label{eq:complete-problem}
\end{align}
where the superscript $t$ denotes the transpose of a vector (or a matrix), and the notation $\cdot{}$ denotes taking derivative
with respect to the time $t$. The relevant parameters in the model \eqref{eq:complete-problem} are defined below
\begin{itemize}
  \item $s_\ell:\Omega \times I \rightarrow \R$: the system functions characterizing the magnetic behavior of the nanoparticles
  \item $\bar{m}:\Omega \times I \rightarrow \R^3$: mean magnetic moment of the nanoparticles
  \item $\mu_0>0$: the magnetic permeability in vacuum
  \item  $a_\ell: \bar{I}:=[-T:T]
\rightarrow \R$: the analog filters in the signal acquisition chain
\item  $p_\ell: \R^3 \rightarrow \R^3$: the sensitivity profiles of the receive coil units
\item $H:\R^3 \times I\rightarrow \R^3$: the applied magnetic field, which also induces a voltage in the receive coil
\item $v_{\mathrm{E},\ell} : I \rightarrow \R$: direct feedthrough
\end{itemize}
The analog filters $a_\ell$ are employed to filter out the direct feedthrough $v_{\rm E,\ell}$, and in practice, they are
commonly band stop filters adapted to excitation frequencies of the drive field. However, the direct feedthrough $v_{\rm E,\ell}$ is
usually not perfectly removed by the analog filter $a_\ell$. One big challenge in the modeling is that the analytic forms
of the filters $a_\ell$ are rarely available. A second challenge is the modeling of the mean magnetic moment $\bar m$.
One often assumes that the moment $\bar m$ is independent of the concentration $c$, and thus ignore possible particle-particle 
interactions (which is however present for high concentrations \cite{loewa2016}).
Then one popular way to relate the moment $\bar m$ to the applied magnetic field $H$ is Langevin theory for paramagnetism,
leading to the so-called equilibrium model \cite{Kluth:2018}. These considerations lead to a simplified affine linear forward map $F:X\to Y^L$:
\begin{align*}
  c \mapsto \left( \int_\Omega s_\ell(x,t) c(x) \ {\rm d}x +  v_{\mathrm{E},\ell} (t) \right)_{\ell=1}^L,
\end{align*}
for suitable function spaces $X$ and $Y$, e.g., $X=L^2(\Omega)$, $Y=L^2(I)$, and $\{s_\ell\}_{\ell=1}^L \subset L^2(\Omega
\times I)$. The task in MPI is to recover the concentration $c$ from the measured voltages $(v_\ell)_{\ell=1}^L\in Y^L$. However,
due to the aforementioned practical complications, the precise kernels $s_\ell$ are usually unavailable, and instead they are
calibrated experimentally for MPI image reconstruction.

\subsection{System matrix calibration} \label{ssec:system_matrix}
First we describe the calibration process for obtaining the system matrix. Let $\Gamma \subset \R^3$ be a reference
volume placed at the origin, which is often taken to be a small cube. Then one selects a set of calibration positions
$\{x^{(i)}\}_{i=1}^m\subset \Omega$, which are often chosen such that the sets $\{x^{(i)}+\Gamma\}_{i=1}^m$ form a
partition of the domain $\Omega$, i.e., they are pairwise disjoint and $\Omega=\cup_{i=1}^m\{x^{(i)} +\Gamma\}$. Let $\chi_S$ denote
the characteristic function of a set $S$. Then the set of piecewise constant functions $\{\chi_{x^{(i)}+\Gamma}\}_{i=1}^m$
forms an orthonormal basis (ONB) for a finite dimensional space, which can be used for approximating the concentration
$c$ in the domain $\Omega$. In the experiment, a small sample is placed at these predefined grid points $\{x^{(i)}\}_{i=1}^m$,
which is described as $c^{(i)} = c_0 \chi_{x^{(i)}+\Gamma}$ for some $c_0>0$ and represents one sample volume for calibration.

The measurements $\{v_\ell^{(i)}=\frac{1}{c_0} F_\ell c^{(i)}\}_{i=1}^m$, $\ell=1,\ldots,L$, are then used to
characterize the discrete data-based forward operator via a discrete system matrix. Mathematically, this can be
formulated using the following map
\begin{align}
 Q_n: L^2(I)^L &\rightarrow \mathbb{R}^{ n:=\sum_{i=1}^L 2|J_i|} \notag \\
(v_\ell)_{\ell=1}^L & \mapsto \left[ \begin{array}{c}
     \left(\begin{array}{c}\mathrm{Re}(\langle v_1 , \psi_j \rangle) \\ \mathrm{Im}(\langle v_1 , \psi_j \rangle) \end{array}\right)_{j\in J_1 } \\ \hline
\vdots \\ \hline
\left(\begin{array}{c}\mathrm{Re}(\langle v_L , \psi_j \rangle) \\ \mathrm{Im}(\langle v_L , \psi_j \rangle) \end{array}\right)_{j\in J_L }
    \end{array}\right],\label{eqn:map}
\end{align}
where $\{\psi_j\}_{j\in\mathbb{N}} \subset L^2(I)$ is an ONB of $L^2(I)$, which is commonly
taken to be the Fourier basis of time-periodic signals in $L^2(I)$, i.e. $\psi_j(t)=T^{-\frac12} (-1)^j
e^{i2\pi j t/T}$. The finite index sets $\{J_\ell\}_{\ell=1}^L\subset \mathbb{Z}$ serve as
a preprocessing step prior to image reconstruction, to be described  below.

The map $Q_n$ in \eqref{eqn:map} consists of concatenating multiple receive coil signals, splitting real and imaginary
parts (if necessary), index / frequency selection, and discretization via projection onto a finite subset of the
ONB $\{ \psi_j \}$ (indexed by $J_\ell$). The system matrix $S$ is then given by
\begin{equation}
 S= \left[ \begin{array}{c|c|c}
Q_n((v_\ell^{(1)})_\ell) & \hdots & Q_n((v_\ell^{(m)})_\ell)
    \end{array}\right] \in \R^{n\times m}.
\end{equation}
For the measured signals $\{v_\ell\}_{\ell=1}^L$, we build the measurement vector $v=Q_n((v_\ell)_{\ell=1}^L)$ analogously.

The background measurement $v^{(0)}=F\mathbf{0}$ (or more precisely, the mean over multiple measurements) is
used to remove the influence of the direct feedthrough $v_{\rm E,\ell}$. Then by subtracting the vector $v_0=
Q_n((v^{(0)}_\ell)_{\ell=1}^L)$ and rank-one matrix $S_0 = v_0 \mathbf{1}_m^t$ (with $\mathbf{1}_m\in \mathbb{R}^m$
with all entries equal to unit), we obtain the following linear MPI reconstruction problem
\begin{equation}
 Ax=y,
\end{equation}
where $A\in \mathbb{R}^{n\times m}$, $y\in\mathbb{R}^n$ and $x\in\mathbb{R}^m$ are defined by
\begin{equation*}
  A=S-S_0, \quad y=v-v_0,\quad \mbox{and}\quad c=\sum_{i=1}^m x_i \chi_{x^{(i)}+\Gamma}.
\end{equation*}
That is, we have used a piecewise constant representation of the concentration $c$,
with $x_i$ being the concentration $c$ on the cell $x^{(i)}+\Gamma$.

It is worth noting that the calibration procedure is laborious, time consuming and highly problem dependent, and has
limited spatial resolution. For example, it requires a recalibration whenever the experimental setting changes. Therefore,
there is a huge demand in developing accurate model-based approaches or hybrid
approaches for MPI image reconstruction. We refer interested readers to \cite{Kluth:2018} for relevant mathematical models and
\cite{MarzWeinmann:2016,ErbWeinmann:2018,KluthJinLi:2017} for preliminary mathematical analysis.

\begin{remark}
Note that different strategies have been proposed to perform the background subtraction for the system matrix in
the literature \cite{weizenecker2009three,Them2016}, which may require additional effort during the system
matrix calibration, in view of costly robot movements.
\end{remark}

\subsection{Frequency selection}\label{ssec:freq_selection}
In MPI there are two standard preprocessing approaches, i.e., band pass approach and SNR-type
thresholding, and they are often combined via the index sets $\{J_\ell\}_{\ell=1}^L$. Let
$I_\mathrm{BP}=\{j\in \mathbb{Z}|\ b_1 \leq |j|/T \leq b_2 \}$ be the band pass indices for
frequency band limits $0\leq b_1 < b_2 \leq \infty$. The main purpose of band pass is to
filter out the direct feedthrough $v_{\rm E,\ell}$ (although not perfectly), and outside the frequency band
$I_{\rm BP}$, the signal is deemed to be too noisy and simply discarded. For the SNR-type thresholding, one
standard quality measure is determined by computing a ratio of mean absolute values from
individual measurements $v_\ell^{(i)}$ (cf. Section \ref{ssec:system_matrix}) and a set of
empty scanner measurements $\{ v_{\ell,0}^{(k)} \}_{k=1}^K$ \cite{Franke:2016} obtained
during the calibration process:
\begin{equation}
 d_{\ell,j}= \frac{\frac{1}{N}\sum_{i=1}^N |\langle v_\ell^{(i)} -\mu_\ell^{(i)}, \psi_j \rangle
 |}{\frac{1}{K} \sum_{k=1}^K |\langle v_{\ell,0}^{(k)} - \mu_\ell, \psi_j \rangle |},
\end{equation}
where $\mu_\ell=\frac{1}{K} \sum_{k=1}^K v_{\ell,0}^{(k)}$ is the mean measurement, and
$\mu_\ell^{(i)}=\kappa_i v_{\ell,0}^{(k_i)} + (1-\kappa_i) v_{\ell,0}^{(k_i+1)}$ is a convex
combination of the previous ($k_i$-th) and following ($k_i+1$-th) empty scanner measurement
with respect to the $i$-th calibration scan. The parameters $\kappa_i \in [0,1]$ are chosen
equidistant for all calibration scans between two subsequent empty scanner measurements.
Then for a given threshold $\tau \geq 0$, we define
\begin{equation}
 J_\ell=\{ j \in I_\mathrm{BP} |  d_{\ell,j}\geq \tau \},\quad \ \ell =1,\hdots,L.
\end{equation}

\begin{remark}
The SNR-type thresholding was also used to obtain a dimensionality reduction in the system of linear
equations in \cite{Knopp_Online_2016} to enable online reconstruction.
\end{remark}

\subsection{Whitening}\label{ssec:whitening}
The calibration process leads to a linear inverse problem
\begin{equation*}
  Ax=y^\delta  \quad \mbox{with } y^\delta = y^\dag + \eta,
\end{equation*}
where $\eta$ denotes the noise in the data, due to the imperfect data acquisition process.
In practice, it is often assumed to follow a Gaussian distribution $N(\mu,C)$ with mean $\mu\in\mathbb{R}^n$ and covariance
$C\in \mathbb{R}^{n\times n}$ (real symmetric positive semidefinite), invoking the central limit theorem (for multiple measurements). These statistical parameters
are then estimated from repetitive measurements. The mean is often approximately zero after
background subtraction. The full covariance matrix $C$ has a large number of parameters, and requires a
large volume of data for a reliable estimate, which is not necessarily available in practice.
Then one often imposes suitable structures on the covariance $C$, e.g., diagonal covariance, or
uses more advanced options, e.g., sparse inverse covariance \cite{FriedmanHastie:2008}. In MPI
experiments, the covariance $C$ is often not a scalar multiple of the
identity matrix (i.e., the noise components are not necessarily independent and identically distributed).
Then it is important to exploit the structure
of the covariance $C$ in image reconstruction, in the spirit of statistical inference. This can be
achieved using a whitening matrix $W$ such that $W(\eta - \mu)$ follows a zero mean Gaussian
distribution with identity covariance. The whitening matrix $W$ can be determined from the
eigendecomposition $(Q,\Lambda)$ of the covariance $C$ (i.e.,
$C= Q\Lambda Q^t$) by $W=\Lambda^{-\frac12}Q^t$. Alternatively one may employ the Cholesky
decomposition to whiten the noise. Then we arrive at the following linear problem
\begin{equation}\label{eqn:whiten}
WAx = W(y^\delta - \mu).
\end{equation}
The whitening step enables the use of the standard least-squares formulation in MPI reconstruction,
in the spirit of the classical maximum likelihood approach, i.e.,
\begin{equation}\label{eqn:whiten-ml}
  \|WAx-W(y^\delta -\mu)\|^2.
\end{equation}
Conceptually, a large variance indicates that the corresponding measurement may be not so reliable,
and thus may behave like an outlier within the dataset, for which an inadvertent use of the standard
least-squares formulation may significantly sacrifice the reconstruction accuracy. Instead, it should be weighed down in the
reconstruction step, which is precisely the role played by the whitening step. Clearly, the whitening in
\eqref{eqn:whiten-ml} is equivalent to the weighted least-squares $(Ax-(y^\delta-\mu))^tC^{-1}(Ax-(y^\delta-\mu))$,
which corresponds to the maximum likelihood estimate for the data $y^\delta$. This formulation also properly accounts
for the noise statistics. However, the explicit whitening construction is advantageous for accelerating
reconstruction via the randomized SVD described in Section \ref{ssec:rsvd} below.

\begin{remark}
The weighting was also used in \cite{Knopp2010e}, with the weight $w_k$ given by the energy of the $k$th row $a_k$
of the system matrix $A$, i.e., $w_k=\|a_k\|$. Thus the weighting in \cite{Knopp2010e} represents a 
form of preconditioning, which differs from the covariance interpretation in \eqref{eqn:whiten}, despite the formal similarity.
\end{remark}

For a calibrated system matrix $A$ as in Section \ref{ssec:system_matrix}, the whitening process has to be adapted
properly. Specifically, for an ONB $\{b_i\}_i\subset X$, we have
\begin{equation*}
w_i=A^\dagger b_i + \eta_i,
\end{equation*}
where $A^\dagger:X\rightarrow Y$ denotes the (unknown) true forward map and $\eta_i$ follow the same
distribution as the noise $\eta$, i.e., $(\eta_i-\mu)\sim N(0,C)$. Then the (mean) corrected and
noisy forward map $A$ is given by
\begin{equation*}
Ax=\sum_i \langle x, b_i \rangle (w_i-\mu) = A^\dagger x + \sum_i \langle x, b_i \rangle (\eta_i-\mu).
\end{equation*}
Thus, the noise term due to modeling error (in the forward map $A$) (relative to the exact one $A^\dag$)
is given by $\sum_i \langle x, b_i \rangle (\eta_i-\mu)$. It is important to observe that the statistics of this term is actually
dependent of the unknown concentration $x$: the mean is still zero, but the covariance is changed via a linear map depending
on $x$. In practical inversion, this error term is often lumped into the data error, and combined with the measurement
error in the data $y^\delta$, whose noise statistics are then $x$-dependent. This short discussion highlights the distinct role 
of modeling error in the data-based approach. In our
discussions below, we shall ignore the modeling error (for the acceleration step) and employ the whitening
procedure described above. Clearly, more suitable approaches should employ alternatives, e.g., total
least-squares approach \cite{KluthMaass:2017} or approximation error modeling \cite{BrandtSeppanen:2018},
which are, however, beyond the scope of this work.

\section{Enhanced image reconstruction}\label{sec:alg}
Now we describe the common MPI reconstruction method, its acceleration via
randomized SVD and the proper choice of the regularization parameter.
\subsection{The common approach}\label{ssec:Kaczmarz}
Currently, the most popular and successful idea in MPI reconstruction is based on the following constrained Tikhonov regularization
with a quadratic penalty:
\begin{equation}\label{eqn:std-Tikhonov}
  x^\dag = \arg\min_{x\geq 0}\|Ax-y^\delta\|^2 + \alpha\|x\|^2,
\end{equation}
where $\alpha>0$ is the regularization parameter, controlling the tradeoff between the two terms \cite{ItoJin:2015};
see Section \ref{ssec:rule} below for two parameter choice rules, i.e., discrepancy principle and quasi-optimality
criterion. The nonnegativity constraint $x\geq 0$ is understood componentwise, and reflects the fact that the
concentration $x$ is nonnegative. The constraint is essential for obtaining physically meaningful reconstructions.
The whitening approach in Section \ref{ssec:whitening} may be implemented in the form \eqref{eqn:std-Tikhonov}
straightforwardly by penalizing the fidelity functional in equation \eqref{eqn:whiten-ml}, and thus all the discussions
below adapt accordingly.

In practice, a variant of the popular Kaczmarz method \cite{Kaczmarz:1937}, developed in \cite{Dax:1993}, is often employed
for solving the constrained optimization problem \eqref{eqn:std-Tikhonov} in the MPI reconstruction. It has demonstrated
excellent empirical performance \cite{weizenecker2009three,Knopp2010e,RahmerHalkola:2015},
and has been implemented in commercial MPI scanners (included in ParaVision\textsuperscript{\textregistered} (Bruker BioSpin MRI
GmbH, Germany) as reported in \cite{franke2017}). One distinct feature of the Kaczmarz method is that at each iteration,
it operates only on one equation, instead of the whole linear system, and thus its computational complexity per iteration
is independent of the amount of data. This feature makes the algorithm especially attractive for problems with large datasets
e.g., 3D MPI, and traditionally it has been very successful within the computed tomography community \cite{HermanLentLutz:1978,
Natterer:1986,JiaoJinLu:2017}. The complete
procedure of the variant in \cite{Dax:1993} is  given in Algorithm \ref{alg:Kaczmarz}. The algorithm often reaches the desired
convergence within tens of sweeps through the equations, and thus its complexity is roughly proportional to the number $n$ of rows
in the matrix $A$. Next we shall employ it as the benchmark algorithm, and propose a preprocessing step to accelerate the computation.

\begin{algorithm}[hbt!]
  \centering
  \caption{Kaczmarz method for problem \eqref{eqn:std-Tikhonov}.\label{alg:Kaczmarz}}
  \begin{algorithmic}[1]
  \STATE Input matrix $A\in\mathbb{R}^{n\times m}$, $y^\delta\in\mathbb{R}^n$, and $\alpha>0$ \\
    Optional: initial value $x_0 \in \R^m$ ($0$ default), relaxation parameter $\omega \in (0,2)$ ($1$ default);
  \STATE Initialize $x=x_0$, $z=0 \in \R^n$, $\bar{z}=0 \in \R^m$;
    \FOR{$k=1,\ldots,K$}
     \STATE $i=( k \ \mathrm{mod} \ n )+1$; \quad \textbackslash\textbackslash row index
     \STATE $\eta = - \omega \frac{\langle a_i, x \rangle + \sqrt{\alpha} z_i - y^\delta_i }{\|a_i\|^2 + \alpha}$;
     \quad  \textbackslash\textbackslash $a_i$ is $i$-th row of $A$
     \STATE $z_i \leftarrow z_i + \eta \sqrt{\alpha}$;
     \STATE $x \leftarrow x + \eta a_i^t$;
    \IF{$i=n$ or $k=K$}
      \STATE $\bar{\eta}=-(\mathrm{min}(\bar{z}_j,\omega x_j))_{j=1,\hdots,m}$;
      \STATE $\bar{z} \leftarrow \bar{z} + \bar{\eta}$;
      \STATE $x \leftarrow x + \bar{\eta}$; \quad \textbackslash\textbackslash positivity constraint
    \ENDIF
    \ENDFOR
   \STATE Return the approximation $x_K\leftarrow x$.
  \end{algorithmic}
\end{algorithm}

\subsection{Acceleration by randomized SVD}\label{ssec:rsvd}

Now we describe a simple acceleration method for Algorithm \ref{alg:Kaczmarz} based on randomized singular value
decomposition (SVD). Recall that SVD of a matrix $A\in \mathbb{R}^{n\times m}$ is given by
\begin{equation*}
  A = U\Sigma V^t,
\end{equation*}
where $U=[u_1\ u_2\ \ldots\ u_n]\in\mathbb{R}^{n\times n}$ and $V=[v_1\ v_2 \ \ldots\ v_m]\in\mathbb{R}^{m\times m}$
are column orthonormal matrices, $\Sigma\in\mathbb{R}^{n\times m}$ is a diagonal matrix, with the diagonal entries ordered
in a nonincreasing manner: $\sigma_1\geq \sigma_2\geq \ldots\geq \sigma_r>\sigma_{r+1}=\ldots=\sigma_{\min(m,n)}$,
where $r$ is the rank of the matrix $A$. Traditional methods for computing SVD, e.g., Lanczos bidiagonalization, are
not attractive for general dense matrices as arising in MPI. For example,  the complexity of
Golub-Reinsch algorithm for computing SVD is $4n^2m + 8m^2n +9m^3$ (for $n\geq m$) \cite[p. 254]{GolubVanLoan:1996}.
Thus, it can be prohibitively expensive for large-scale matrices. The randomized SVD (rSVD) provides an
efficient way to construct a low-rank approximation by randomly mixing the columns of $A$ \cite{HalkoMartinssonTropp:2011}.
The overall procedure is given in Algorithm \ref{alg:rsvd} for the case $n\geq m$, and the case $n<m$ can be obtained
by applying Algorithm \ref{alg:rsvd} to the transposed matrix $A^t$.

\begin{algorithm}[hbt!]
  \centering
  \caption{rSVD for $A\in\mathbb{R}^{n\times m}$, $n\geq m$.\label{alg:rsvd}}
  \begin{algorithmic}[1]
    \STATE Input matrix $A\in\mathbb{R}^{n\times m}$, $n\geq m$, and target rank $k$;
    \STATE Set parameters $p$ (default $p=5$), and $q$ (default $q=0$);
    \STATE Sample a random matrix $\Omega=(\omega_{ij})\in\mathbb{R}^{m\times (k+p)}$, with $\omega_{ij}\sim N(0,1)$;
    \STATE Compute the randomized matrix $Y=(AA^*)^qA\Omega$;
    \STATE Find an orthonormal basis $Q$ of $\mathrm{range}(Y)$;
    \STATE Form the matrix $B=Q^*A$;
    \STATE Compute the SVD of $B=WSV^*$;
    \STATE Return the rank $k$ approximation $(\tilde U_k,\tilde \Sigma_k,\tilde V_k)$, cf. \eqref{eqn:rsvd}.
  \end{algorithmic}
\end{algorithm}

In Algorithm \ref{alg:rsvd}, Step 4 is to extract the column space $\mathcal{R}(A)$ of $A$, i.e., $\mathcal{R}(Y)\subset
\mathcal{R}(A)$, and Step 5 is to find an orthonormal basis for $\mathcal{R}(Y)$, e.g., via QR decompositon or skinny SVD. The
remaining steps can be regarded as one subspace iteration for computing SVD of the matrix $QQ^tA$. Since the involved
matrices are of much smaller size, these SVDs can be carried out efficiently. The accuracy of $\mathcal{R}(Y)$ to
$\mathcal{R}(A)$ is crucial to the success of the algorithm. A positive exponent $q$ can improve the accuracy when the
singular values of $A$ decay slowly, and the oversampling parameter $p$ is to improve the accuracy of the range probing.
The low-rank approximation $\tilde A_k$ by rSVD is given by
\begin{equation}\label{eqn:rsvd}
  \tilde A_k = \tilde U_k\tilde \Sigma_k \tilde V_k^t, \mbox{ with }\tilde U_k=(QW)_{:,1:k}, \ \ \tilde\Sigma_k = S_{1:k,1:k}, \ \ \tilde V_k=V_{:,1:k},
\end{equation}
where the notation $1:k$ denotes taking the first $k$ columns/rows of the matrix. The complexity of Algorithm
\ref{alg:rsvd} is around $4(q+1)kmn$, which is much lower than computing SVD of $A$ directly. Clearly, the
efficiency of the approach relies crucially on the low-rank structure of $A$. In the context of MPI, it was rigorous
justified for the equilibrium model in \cite{KluthJinLi:2017}: by means of singular value decay estimates, the MPI
inverse problem is shown to be severely ill-posed for the equilibrium model with common experimental setups.

With the rSVD $(\tilde U_k,\tilde \Sigma_k,\tilde V_k)$ at hand, we approximate constrained Tikhonov
regularization \eqref{eqn:std-Tikhonov} by
\begin{align}
  &\quad \arg\min_{x\geq 0}\|Ax-y^\delta\|^2 + \alpha \|x\|^2 \nonumber\\
  & \approx \arg\min_{x\geq0}\|\tilde U_k\tilde \Sigma_k\tilde V_k^tx - y^\delta\|^2 + \alpha \|x\|^2\nonumber\\
    & = \arg\min_{x\geq0} \|\tilde\Sigma_k\tilde V_k^tx - \tilde U_k^ty^\delta\|^2 + \alpha\|x\|^2. \label{eqn:std-Tikhonov-apprx}
\end{align}
The number $k$ of rows in the approximate optimization problem \eqref{eqn:std-Tikhonov-apprx} is much smaller than $n$, enabling a
significant speedup of Algorithm \ref{alg:Kaczmarz}. In essence, rSVD is a preprocessing step to extract essential information
content in $A$, and can also be viewed as a dimensionality reduction strategy in the data space; see Appendix \ref{app:error}
for error estimates on the minimizer due to the low-rank approximation. Note that this step does not alter the whole reconstruction procedure.

So far we have described the acceleration procedure for problem \eqref{eqn:std-Tikhonov}. It applies equally well
to the whitened problem \eqref{eqn:whiten} derived from Section \ref{ssec:whitening}, when the covariance $C$ of the noise $\eta$
is not a scalar multiple of the identity matrix. This can be achieved simply by applying rSVD to the whitened
matrix $WA$ to construct an accurate low-rank approximation. Clearly, the whitening matrix $W$ may influence
the spectral behavior of $WA$, which is generally different from that of $A$. In passing, we note that the reduced
problem \eqref{eqn:std-Tikhonov-apprx} may be solved by another iterative solvers, e.g., CGNE or LSQR \cite{Saad:2003}, and
acceleration is also expected, which however will not be further pursued below.

\subsection{Parameter choice}\label{ssec:rule}
One important issue of any imaging algorithm is the proper choice of the regularization parameter $\alpha$ in the regularized
problem \eqref{eqn:std-Tikhonov}. Too small a value for $\alpha$ leads to overfitting, whereas too large a value for $\alpha$
leads to oversmoothing and smearing in the reconstruction. Thus it is very important to choose a proper $\alpha$ value. This
choice generally has been a notoriously challenging issue and is still not satisfactorily resolved. Nonetheless, a large
number of choice rules, e.g., discrepancy principle, quasi-optimality criterion, balancing principle, generalized cross
validation and L-curve criterion, have been proposed \cite{ItoJin:2015}. However, these rules have not been extensively studied within the
MPI community, where only the L-curve criterion has been experimentally evaluated \cite{Knopp_etal2008fc}. One challenge with MPI is
the presence of significant model errors, besides the usual data error, and these rules have to be adapted properly
(see the work \cite{HamarikKindermann:2018} for some recent insights). We describe two popular choice rules, i.e.,
discrepancy principle \cite{Morozov:1966} and quasi-optimality criterion \cite{TikhonovGlaskoKriksin:1979}.

The discrepancy principle due to Morozov \cite{Morozov:1966} chooses a parameter $\alpha$ such that the residual is comparable with the noise level  $\delta$
of the data, i.e., $\delta:=\|y-y^\delta\|$ and the error $\epsilon:=\|A-\tilde A\|$ of the inexact operator $\tilde A$. This may
be carried out as follows. Let $\{\alpha_i\equiv\alpha_0q^i\}_{i\geq0}$ be a geometrical
sequence, with $\alpha_0>0$ and $q\in(0,1)$ being the largest regularization parameter and the decreasing
factor, respectively. Often one sets $\alpha_0$ to $\alpha_0=\|A\|^2$. Then the discrepancy principle
choose the optimal $\alpha_{i^*}$ from the sequence such that
\begin{equation}\label{eqn:discrepancy}
 i^\ast = \arg \min_{i\geq0} \{ \|\tilde A\tilde x^\delta_{\alpha_i} - y^\delta \| \leq \tau \delta + \sigma \epsilon \}.
\end{equation}
Here, the parameters $\tau,\sigma$ are to be specified. A rule of thumb of their choice is as follows: $\tau>1$, say $\tau =1.1$,
and $\sigma>\|x^\dag\|$, e.g., $\sigma = 9/4\|x^\dag\|$ \cite{MaassRieder:1997}. The bound $\|x^\dag\|$ may have to be estimated.
With the choice, then the true solution $x^\dag$ is admissible in the following sense:
\begin{equation*}
  \|\tilde Ax^\dag-y^\delta\|\leq \epsilon\|x^\dag\| + \delta.
\end{equation*}
In order to apply the discrepancy principle, one needs an accurate bound on $\delta$ and $\epsilon$, which is
not always easy to obtain and thus may limit its applicability in practice.

The quasi-optimality criterion \cite{TikhonovGlaskoKriksin:1979} is one popular heuristic rule, which is purely data driven.
In a Hilbert space setting, the rule amounts to minimize the function $\alpha
\|\frac{d}{d\alpha}\tilde x_\alpha^\delta\|$, which can serve as a rough bound on the reconstruction
error \cite{JinLorenz:2010}. On a geometric sequence $\{\alpha_i\equiv \alpha_0q^i\}_{i\geq0}$, the
discrete version reads
\begin{equation}\label{eqn:quasi_optimality}
  i^* = \arg\min_{i\geq 0}\{\|\tilde x_{\alpha_{i+1}}^\delta - \tilde x_{\alpha_i}^\delta\|\},
\end{equation}
and the chosen value is then given by $\alpha_{i^*}$. Clearly, the criterion is straightforward to implement.
Under various structural conditions on the noise, the convergence and rates of the rule were analyzed in
\cite{JinLorenz:2010}.

\section{Numerical results and discussions}\label{sec:numer}

In this section, we present numerical results to illustrate the performance of the proposed algorithmic
tricks, i.e., whitening, acceleration and parameter choices. The experimental setup is as follows.
We employ a measured system matrix, where a band pass filter is applied (with $b_1=80$~kHz and $b_2=625$~kHz),
which yields a system matrix $A\in \R^{n\times m}$ for the $L=3$ receive channels, cf. Section
\ref{ssec:system_matrix}. Background measurements (with the same band filter) are used to obtain a diagonal
whitening operator $W\in \R^{n\times n}$ (cf. Section \ref{ssec:whitening}) and thus also the whitened matrix
$A_W=WA\in\R^{n\times m}$ (for the whitening approach). Let $(\tilde U_k, \tilde \Sigma_k,\tilde V_k)$ be the
rSVD of $A$ given by Algorithm \ref{alg:rsvd}, and analogously $(\tilde U_{W;k},\tilde \Sigma_{W;k},
\tilde V_{W;k})$ for the rSVD of $A_W$. All forward maps are scaled to a unit operator norm.

Below we compare the reconstructions by the proposed method with that by the standard Kaczmarz method
(i.e., Algorithm \ref{alg:Kaczmarz}) and the dimensionality reduction method proposed in
\cite{Knopp_Online_2016}. Specifically, we consider the following reconstruction methods:
\begin{itemize}
  \item \textbf{[STD]}: The reconstructions $x_{\rm STD}$ and $x_{W;\rm STD}$ are respectively obtained by
    \begin{align*}
      x_\mathrm{STD}&=\arg\min_{x\geq 0}\|Ax-y^\delta\|^2 + \alpha\|x\|^2 \quad \mbox{and}\quad x_{W;\mathrm{STD}}
   =\arg\min_{x\geq 0}\|A_Wx-Wy^\delta\|^2 + \alpha\|x\|^2,
   \end{align*}
   with Algorithm \ref{alg:Kaczmarz}; cf. problem \eqref{eqn:std-Tikhonov}.
  \item \textbf{[SNR]}: For a given $k\in\N$, there exists a $\tau_k$ such that a reduced system with $k$ rows is
    obtained via the SNR-type frequency selection for $\tau=\tau_k$ (cf. Section \ref{ssec:freq_selection}) using $Q_k$
    to build the system matrix $A_k$ and the measurement vector $y^\delta_k$ as proposed for online reconstruction in
    \cite{Knopp_Online_2016}. Then the diagonal whitening operator $W_k$ is determined for the reduced system. The reconstructions
    $x_{\rm SNR}$ and $x_{W_k,\rm SNR}$ are respectively obtained by
    \begin{align*}
       x_\mathrm{SNR}=\arg\min_{x\geq 0}\|A_kx-y^\delta_k\|^2 + \alpha\|x\|^2\quad\mbox{and}\quad x_{W_k;\mathrm{SNR}}=\arg\min_{x\geq 0}
    \|A_{k;W_k}x-W_ky^\delta_k\|^2 + \alpha\|x\|^2,
    \end{align*}
    with Algorithm \ref{alg:Kaczmarz}; cf. problem \eqref{eqn:std-Tikhonov}.
   \item \textbf{[rSVD1]}: The (rSVD) reconstructions $x_{\rm rSVD1}$ and $x_{W;\rm rSVD1}$ are respectively obtained by
   \begin{align*}
     x_\mathrm{rSVD1} &=\arg\min_{x\geq0} \|\tilde\Sigma_k\tilde V_k^tx - \tilde U_k^ty^\delta\|^2
   + \alpha\|x\|^2,\\
    x_{W;\mathrm{rSVD1}}&=\arg\min_{x\geq0} \|\tilde\Sigma_{W;k}\tilde V_{W;k}^tx - \tilde U_{W;k}^t
     Wy^\delta\|^2 + \alpha\|x\|^2,
   \end{align*}
   with Algorithm \ref{alg:Kaczmarz} for given $k \in \N$; cf. problem \eqref{eqn:std-Tikhonov-apprx}.
   \item \textbf{[rSVD2]}: The reconstruction $x_{\rm rSVD2}$ is computed via
   \begin{equation*}
     x_\mathrm{rSVD2}=P_{\R_+^m}\tilde V_k \tilde\Sigma_k^{-1;\alpha} \tilde U_k^ty^\delta,
  \end{equation*}
  with $\tilde\Sigma_k^{-1;\alpha}=\mathrm{diag}(\tilde  \Sigma_{k;ii}/(\tilde\Sigma_{k;ii}^2
  + \alpha^2))\in\mathbb{R}^{k\times k}$. The reconstruction $x_{W;\mathrm{rSVD2}}$ is obtained
  similarly. These methods treat the nonnegativity constraint in an ad hoc manner, and can be used
  as rough approximations to $x_{\rm rSVD1}$ and $x_{\rm W;rSVD1}$.
\end{itemize}

These methods are evaluated on a publicly available 3D dataset, i.e., \texttt{open MPI dataset} (downloaded
from \url{https://www.tuhh.de/ibi/research/open-mpi-data.html}, last accessed on January 19, 2019) provided
in the MPI Data Format (MDF) \cite{knopp2018mdf}. The (measured) system matrix data $\{v_\ell^{(i)}\}_{i=1}^m$,
$\ell=1,2,3$, is obtained using a cuboid sample of size 2 mm $\times$ 2 mm $\times$ 1 mm. The calibration is
carried out with Perimag\textsuperscript{\textregistered} tracer with a concentration 100 mmol/l.
The field-of-view has a size of 38 mm $\times$ 38 mm $\times$ 19 mm and the sample positions have a distance
of 2 mm in x- and y-direction and 1 mm in z-direction, resulting in $19\times19\times19=6859$ voxels, which
gives the number $m$ of columns in the system matrix $A$. The entries of $A$ are averaged over 1000 repetitions
and empty scanner measurements are performed and averaged every 19 calibration scans. The phantom measurements
are averaged over 1000 repetitions of the excitation sequence, and with each phantom, an empty measurement
with 1000 repetitions is provided, which are used for the background correction of the measurement and the
system matrix $A$ (cf. Section \ref{ssec:system_matrix}) and also for the approximation of the covariance $C$
respectively the whitening matrix $W$ (see Section \ref{ssec:whitening}).

We validate the proposed methods on the ``shape'' phantom in the dataset. It is a cone defined by a 1 mm radius
tip, an apex angle of 10 degree, and a height of 22 mm. The total volume is 683.9 $\mu$l. Perimag\textsuperscript{\textregistered}
tracer with a concentration of 50 mmol/l is used. See Fig. \ref{fig:phantom_shape} for a schematic illustration
of the phantom and the visualization structure of the 3D reconstructions below.

\begin{figure}[hbt!]
\centering
\begin{minipage}{0.4\textwidth}
\includegraphics[height=5cm]{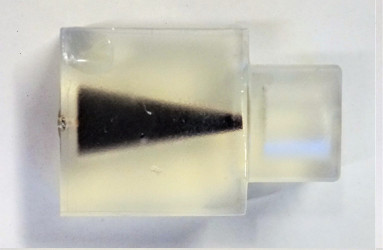}
\end{minipage}
\hspace{2cm}
\begin{minipage}{0.2\textwidth}
\includegraphics[height=5cm]{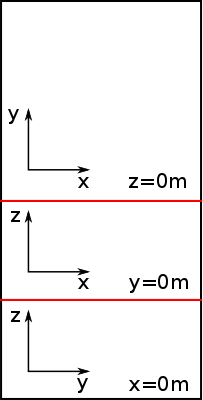}
\end{minipage}
 \caption{``Shape'' phantom from the open MPI dataset (left) and visualization structure
 for the 3D reconstructions (right).}\label{fig:phantom_shape}
\end{figure}

\subsection{The benefit of whitening}
First, we illustrate the benefit of whitening in the standard reconstruction technique (i.e., STD). Due to the
small number of repetitions of the empty measurements (1000 repetitions compared to 23482 indices
in the band limits for each receive coil, when assuming the receive coils are independent), the
covariance $C$ is approximated by a diagonal one to ensure a reliable estimation. The estimated
(diagonal) covariance $C$ is shown in Fig. \ref{fig:cov_diag}. The magnitude of the noise variance
is observed to vary dramatically with the frequency over the frequency band for both  real and
imaginary parts, and the behavior is similar for all three receive coils. The heteroscedastic nature
of the noise necessitates the use of the whitening / weighting in the reconstruction algorithm as
discussed in Section \ref{ssec:whitening} in order to properly account for the noise statistics.

\begin{figure}[hbt!]
\centering
\begin{minipage}{0.3\textwidth}
\centering
$x$-coil \\
\includegraphics[width=\textwidth]{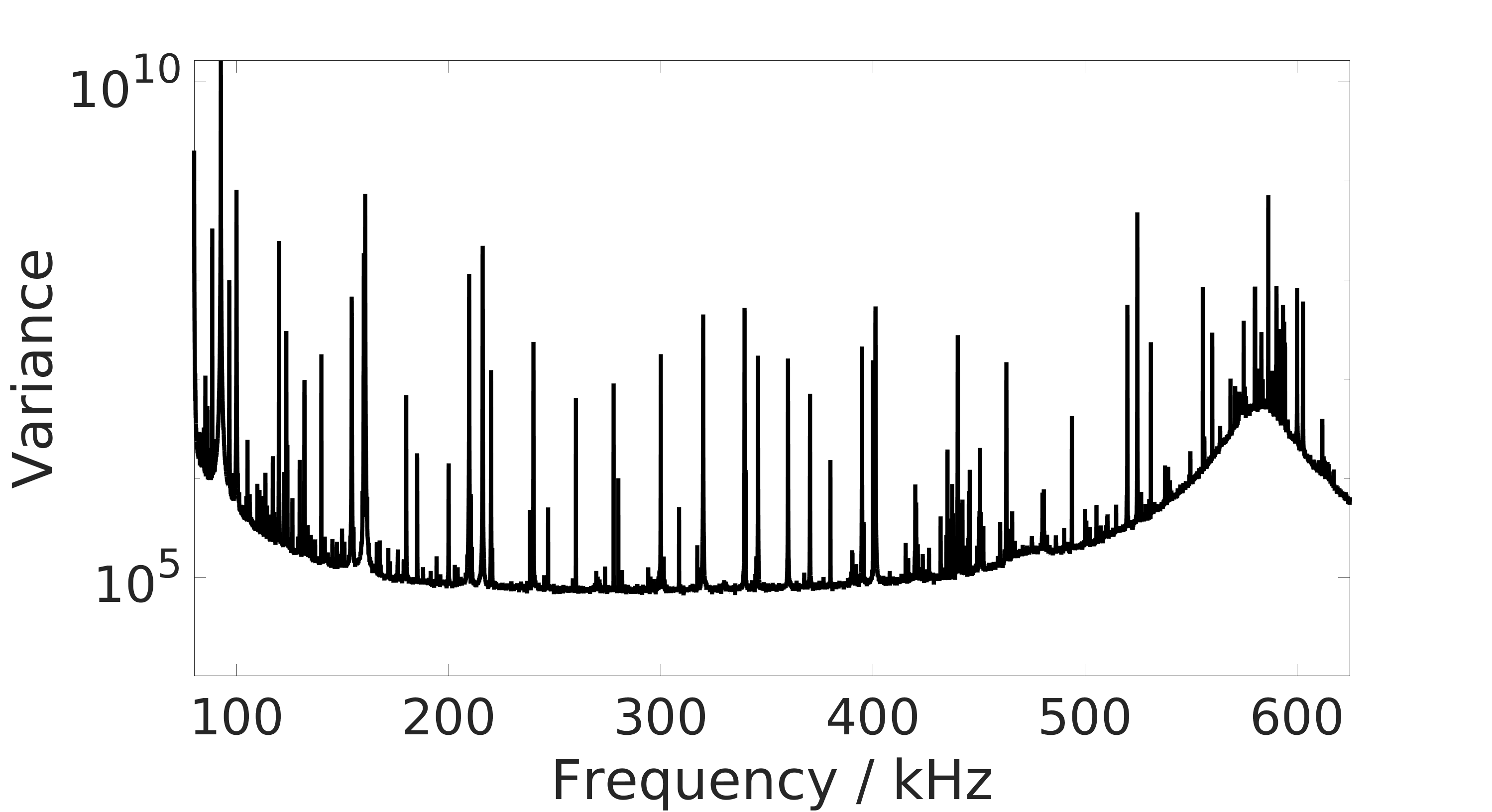} \\
\includegraphics[width=\textwidth]{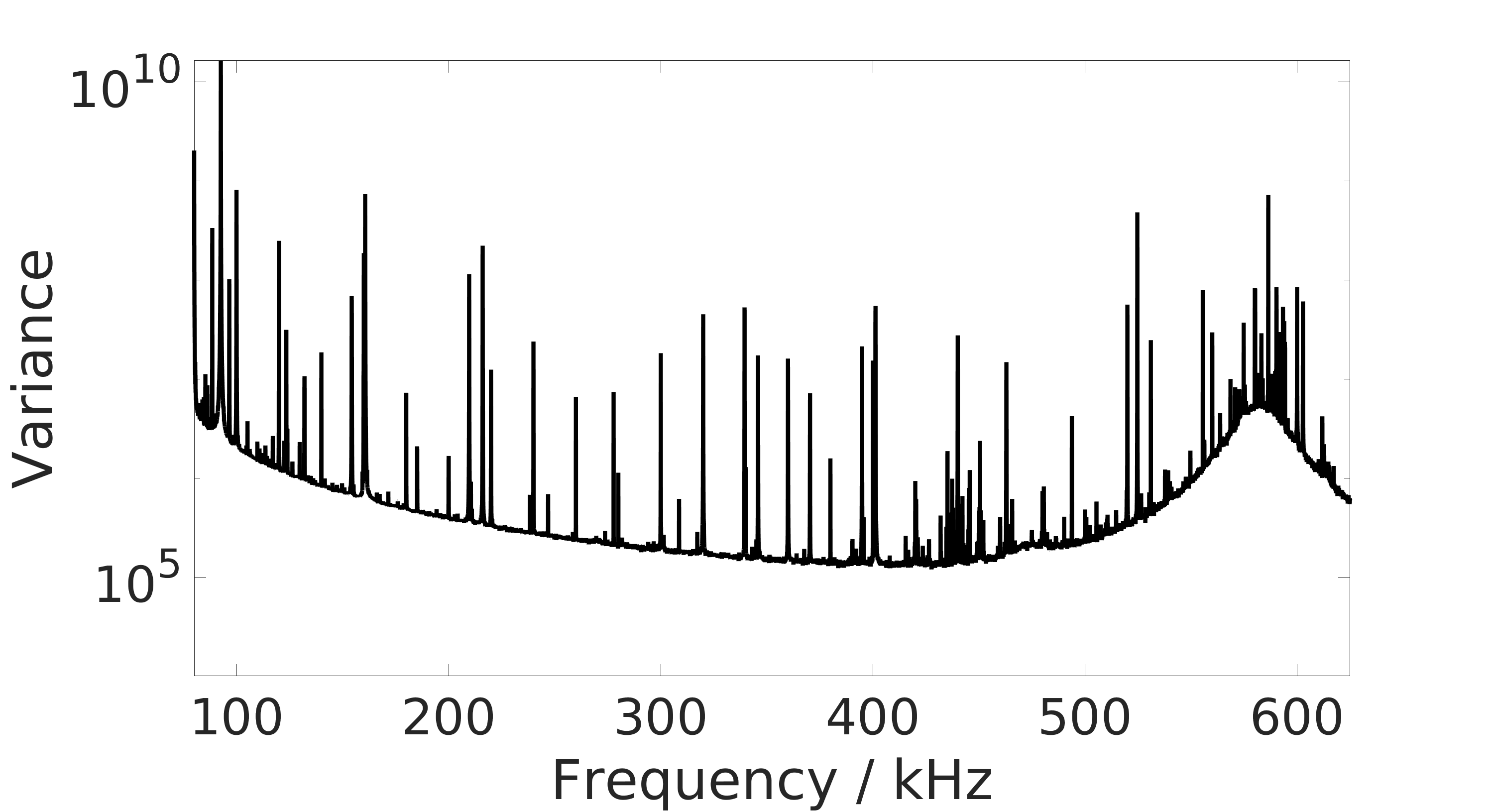} \\
 \end{minipage}
\begin{minipage}{0.3\textwidth}
\centering
$y$-coil \\
\includegraphics[width=\textwidth]{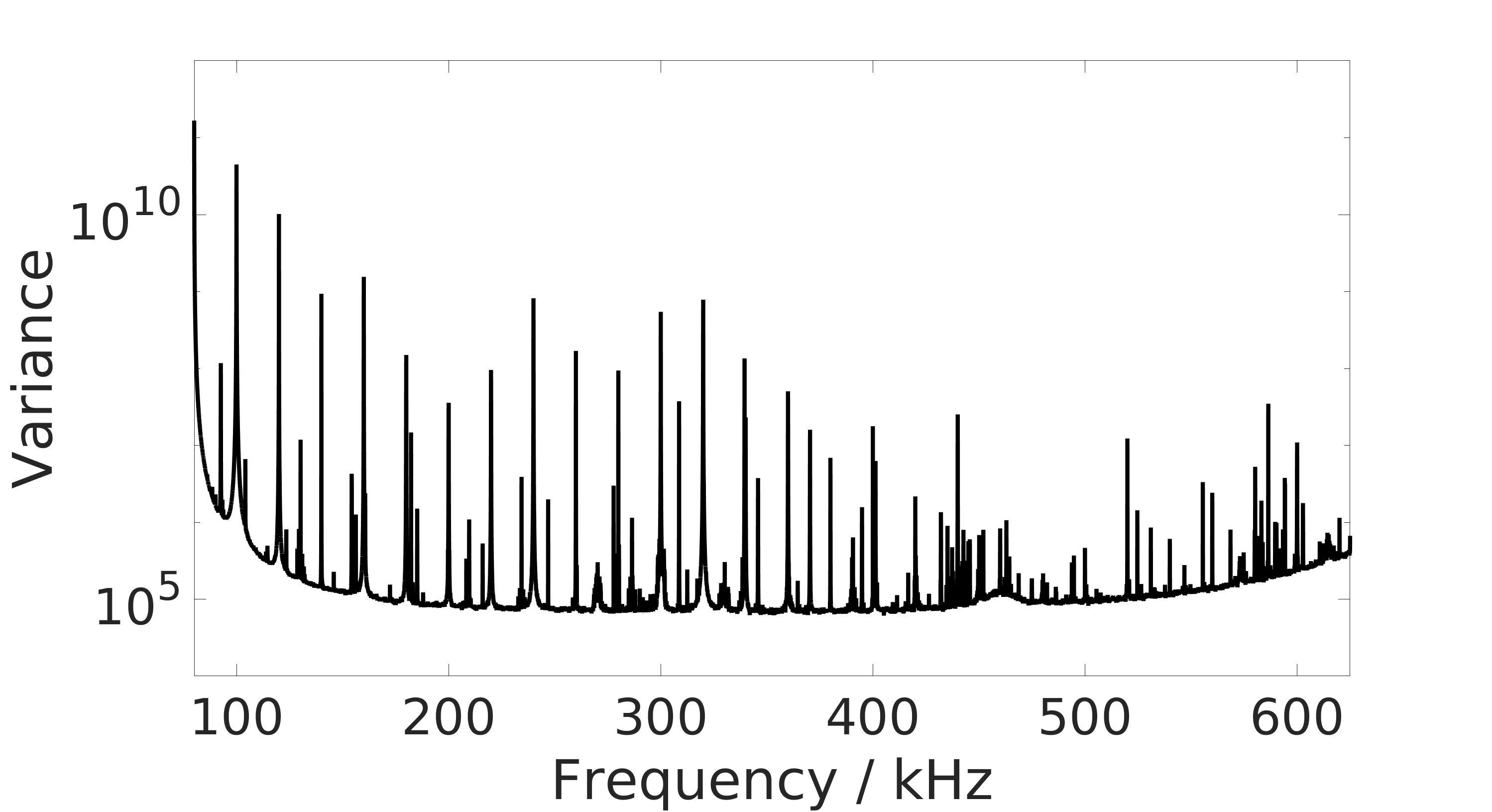} \\
\includegraphics[width=\textwidth]{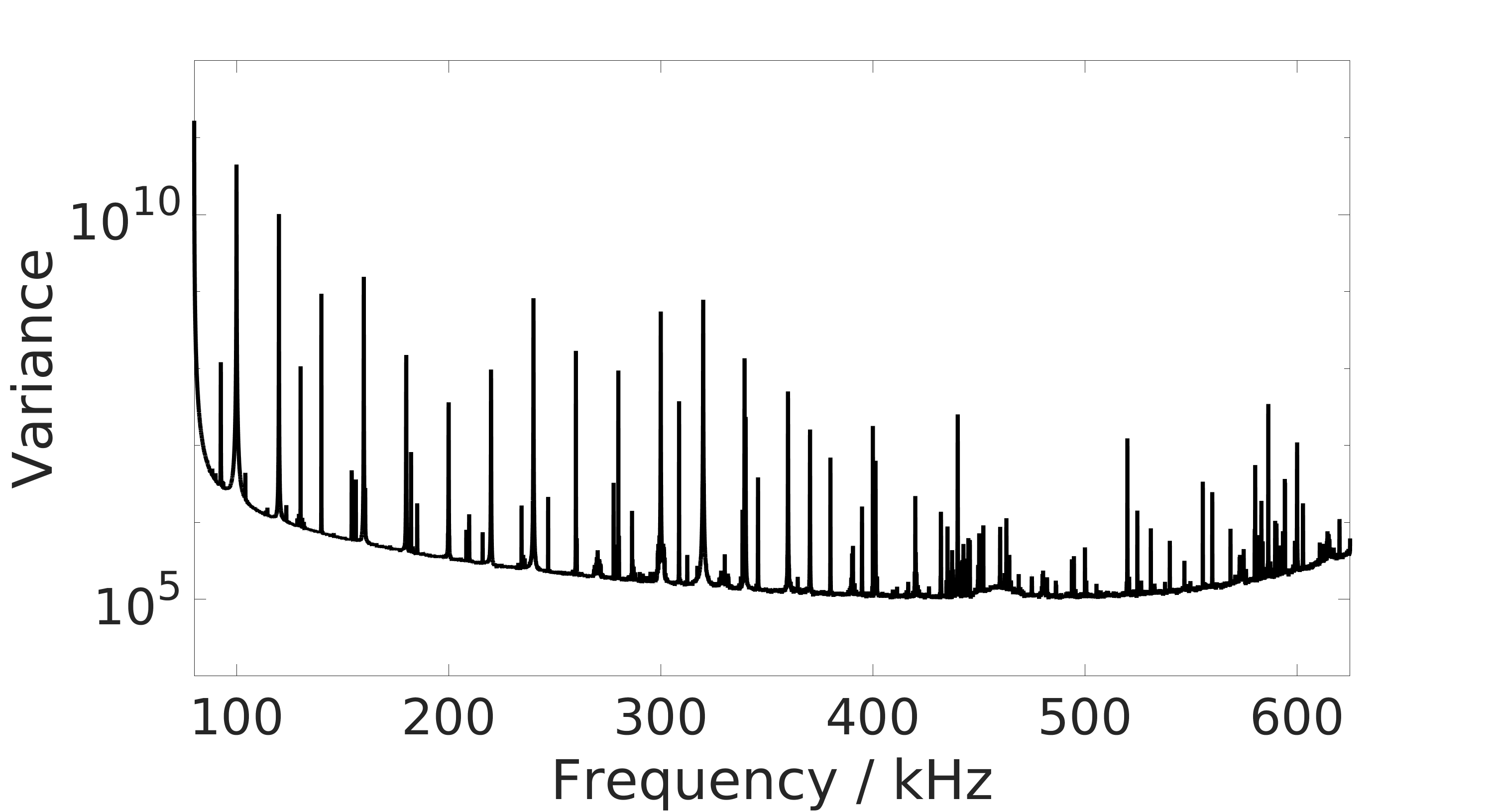} \\
 \end{minipage}
\begin{minipage}{0.3\textwidth}
\centering
$z$-coil \\
\includegraphics[width=\textwidth]{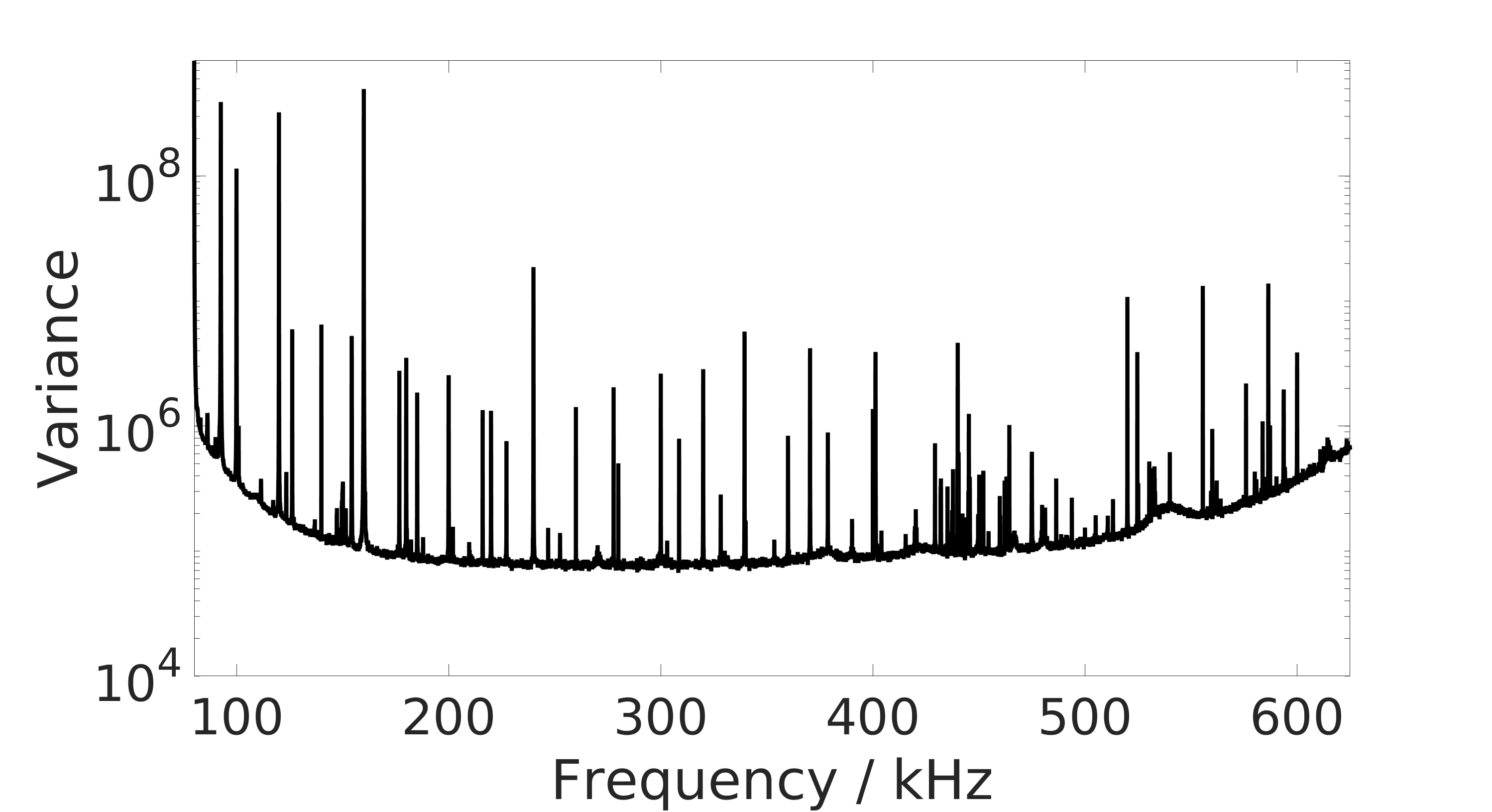} \\
\includegraphics[width=\textwidth]{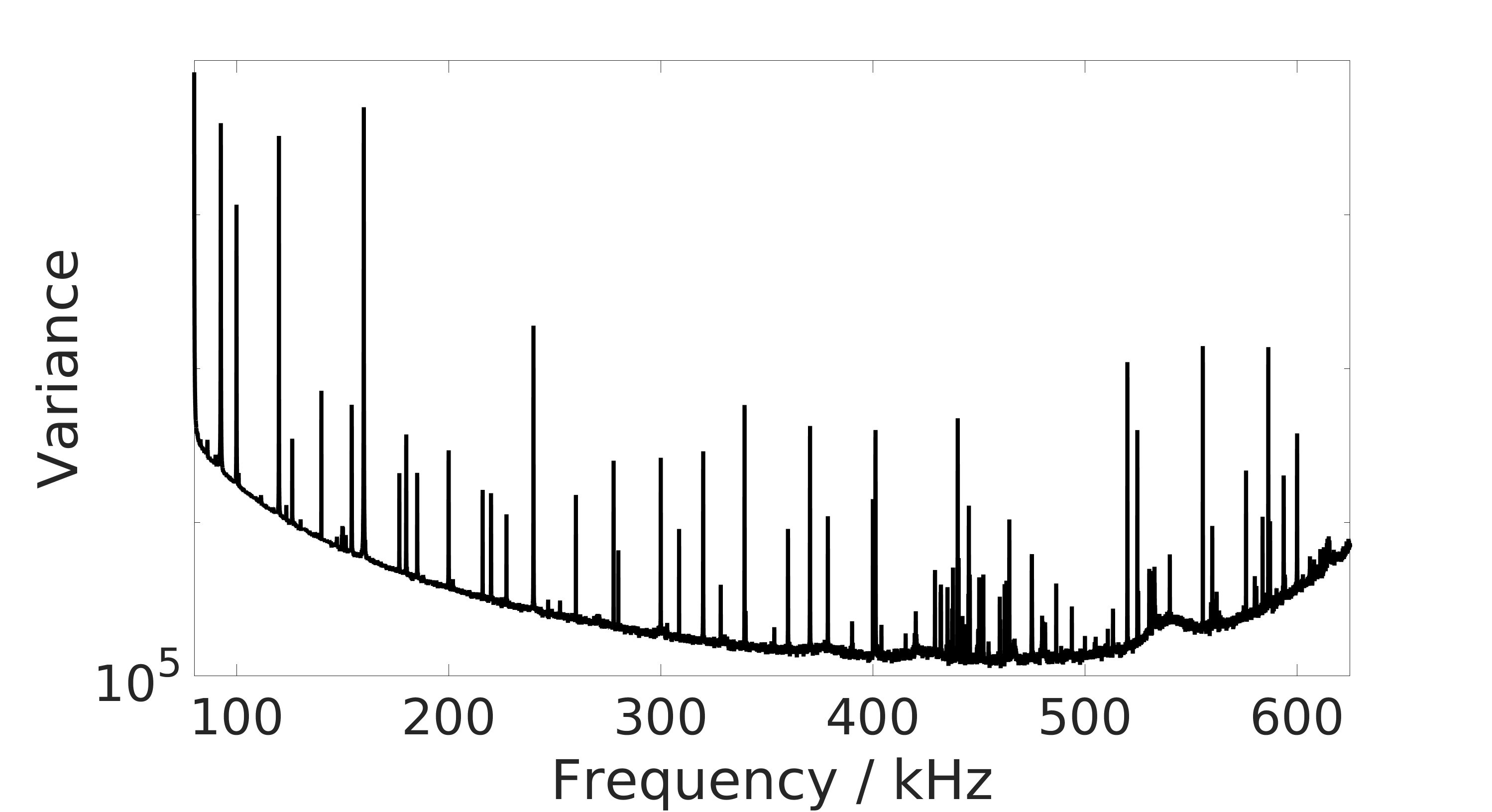} \\
 \end{minipage}
\caption{Variance structure of the diagonal covariance matrix $C$. Visualized
individually for each receive coil with respect to the frequency; real part (top),
imaginary part (bottom).}
\label{fig:cov_diag}
\end{figure}

The STD reconstructions for the non-whitened and whitened cases are shown in Fig. \ref{fig:std_reco}, for three
different $\alpha$ values, including the cases of over, medium and under regularization, respectively. For the medium and small
$\alpha$ values, the reconstructed phantoms for both non-whitened and whitened cases are of similar quality; see
the middle and right columns of Fig. \ref{fig:std_reco}. However, a closer inspection shows that the reconstruction
in the non-whitened case suffers from pronounced background artifacts, whereas, in the whitened case, the artifacts
can be reduced even for much smaller $\alpha$ values. Meanwhile, for a large $\alpha$ value ($\alpha=9.77\times
10^{-2}$) (the left column), the background artifacts disappear from the reconstructions in the non-whitened case
but also the reconstructed cone is overly smoothed, due to over-regularization introduced by the penalty; and these
observations hold also for the whitened case. Thus, the whitening step makes the reconstruction algorithm more robust to the
choice of the $\alpha$ value, which is highly desirable in practice, since its optimal choice is generally very challenging.

In summary, the STD reconstructions in Fig. \ref{fig:std_reco} have similar quality in the whitened
and non-whited cases, except some smaller background artifacts for the non-whitened approach.

\newcommand{\PLH}{{\mkern-2mu\times\mkern-2mu}}

\begin{figure}[hbt!]
 \begin{tabular}{ccc|ccc}
\multicolumn{3}{c|}{Non-whitened} & \multicolumn{3}{c}{Whitened}\\
\hline
$\alpha=9.77 \PLH 10^{-2}$ & $\alpha=3.05 \PLH 10^{-3}$ &$\alpha=9.54 \PLH 10^{-5}$ & $\alpha=9.77 \PLH 10^{-2}$& $\alpha=3.05 \PLH 10^{-3}$ & $\alpha=9.54 \PLH 10^{-5}$ \\
\hline
 \includegraphics[width=0.14\textwidth]{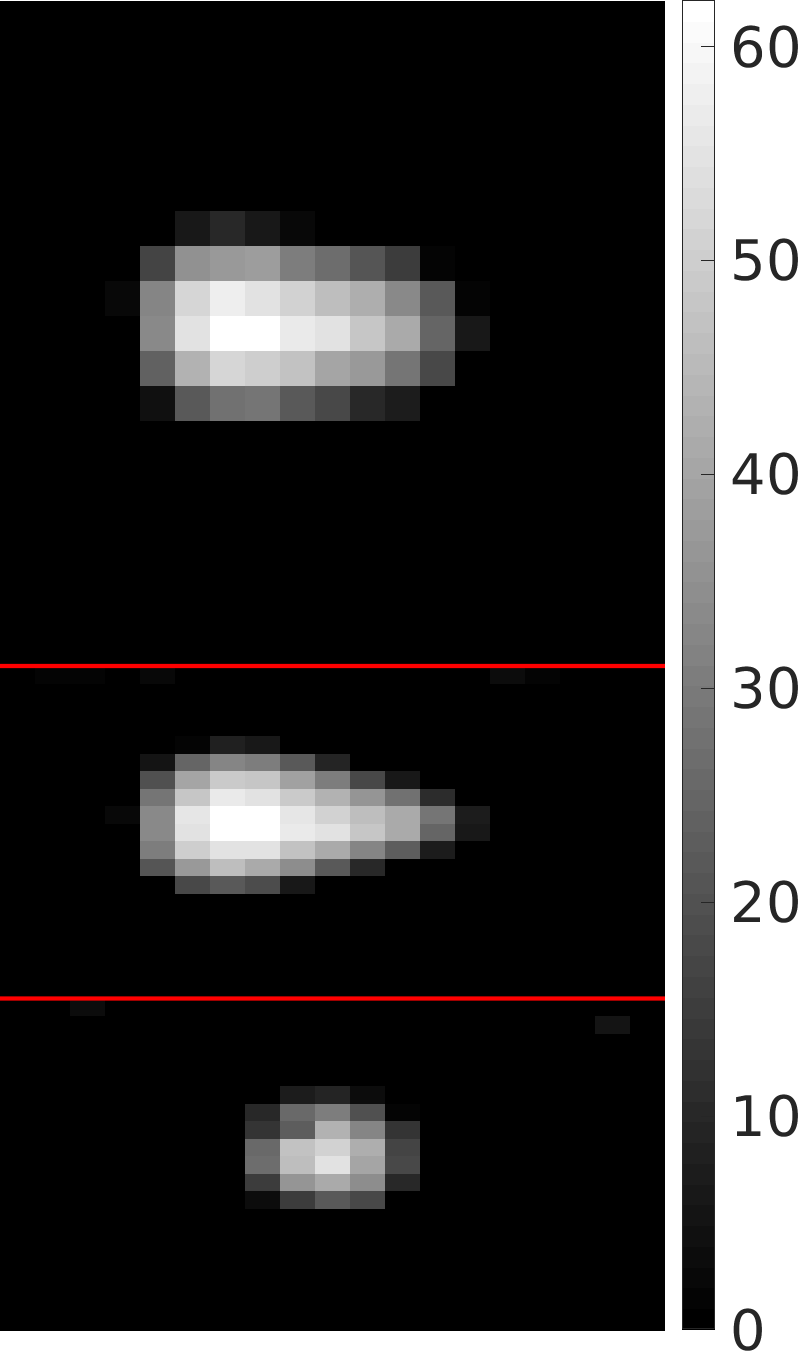}&
 \includegraphics[width=0.14\textwidth]{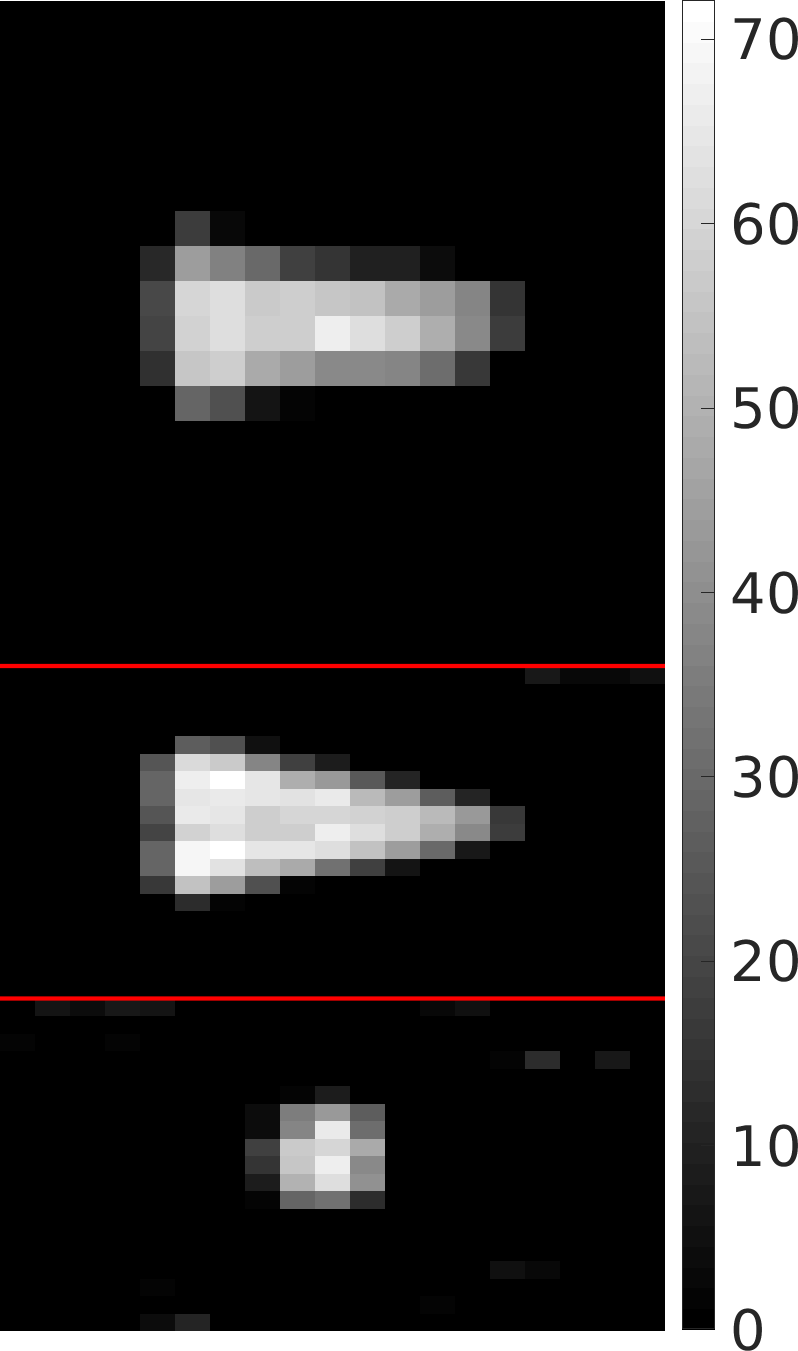}&
 \includegraphics[width=0.14\textwidth]{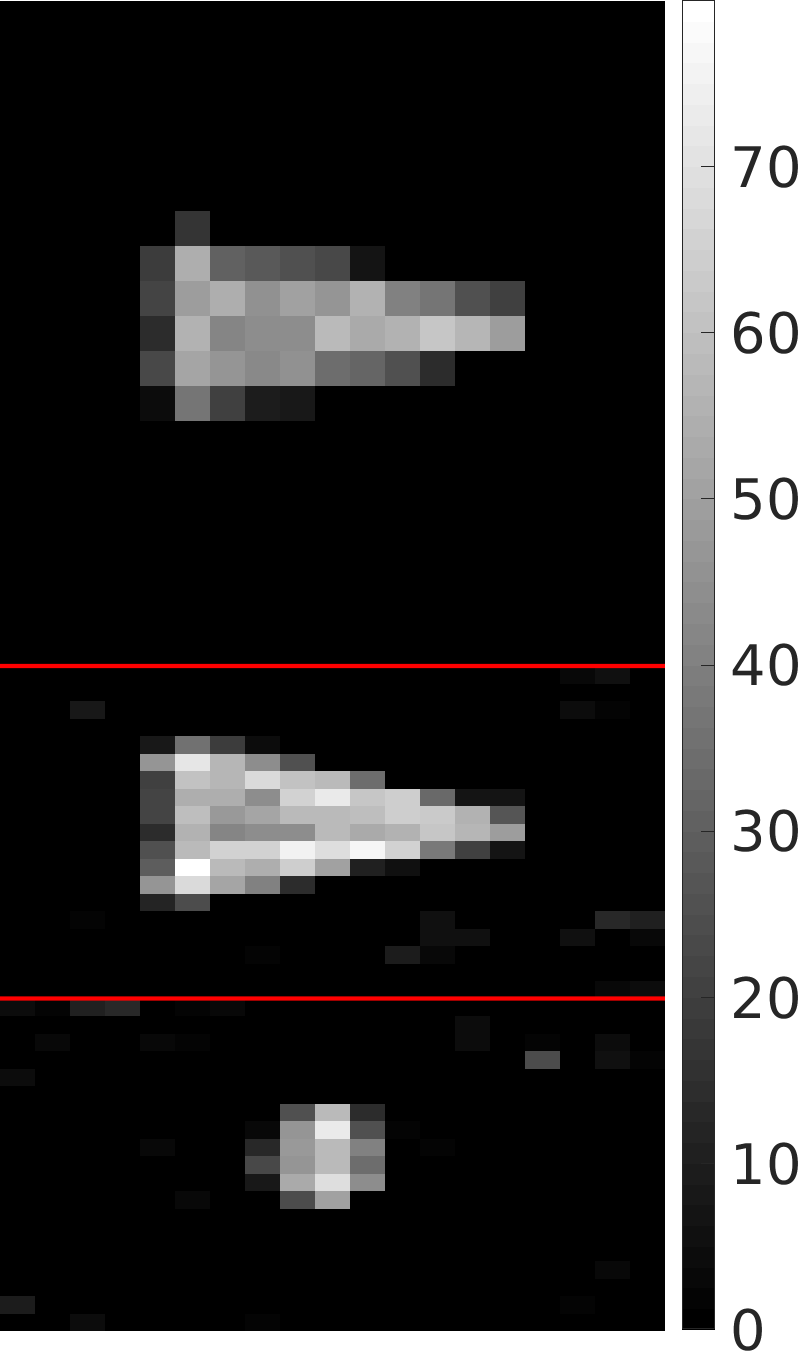}&
 \includegraphics[width=0.14\textwidth]{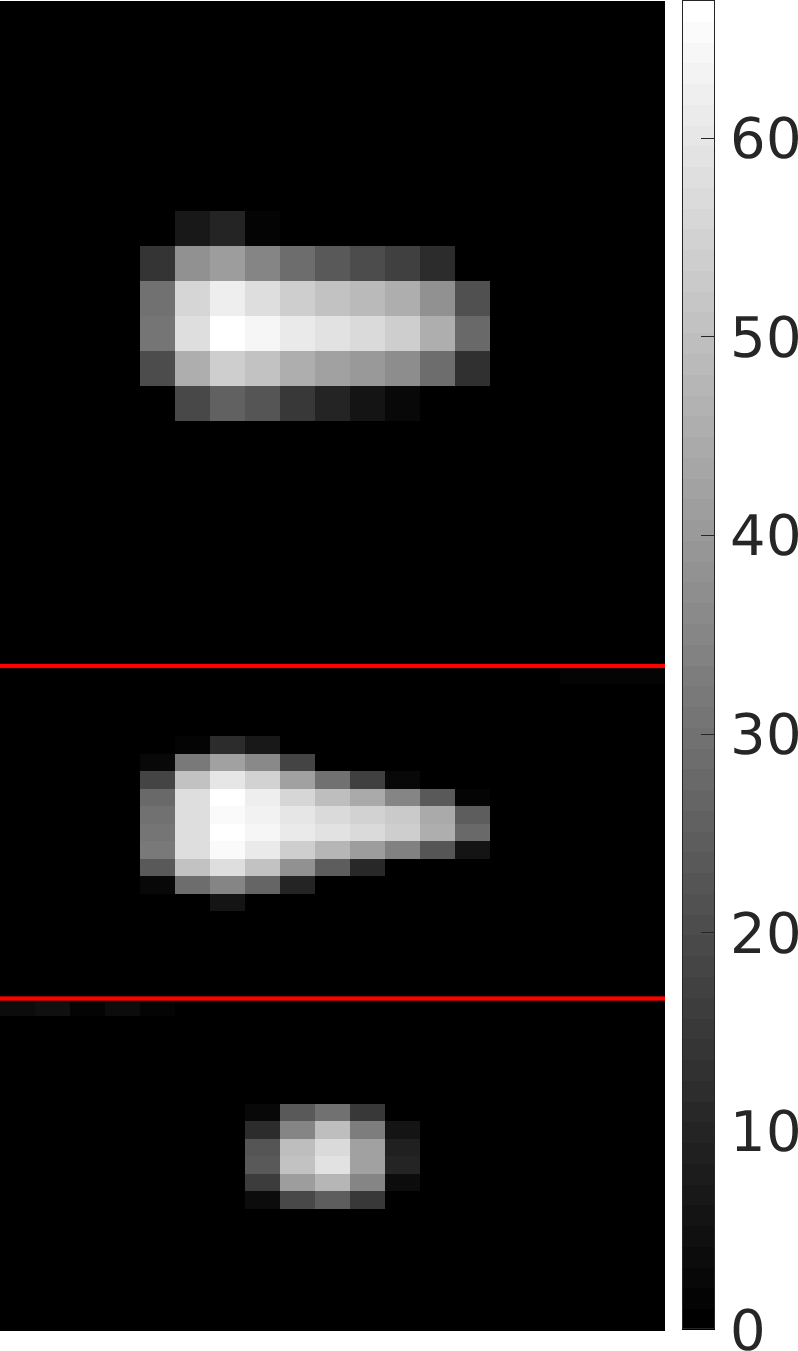}&
 \includegraphics[width=0.14\textwidth]{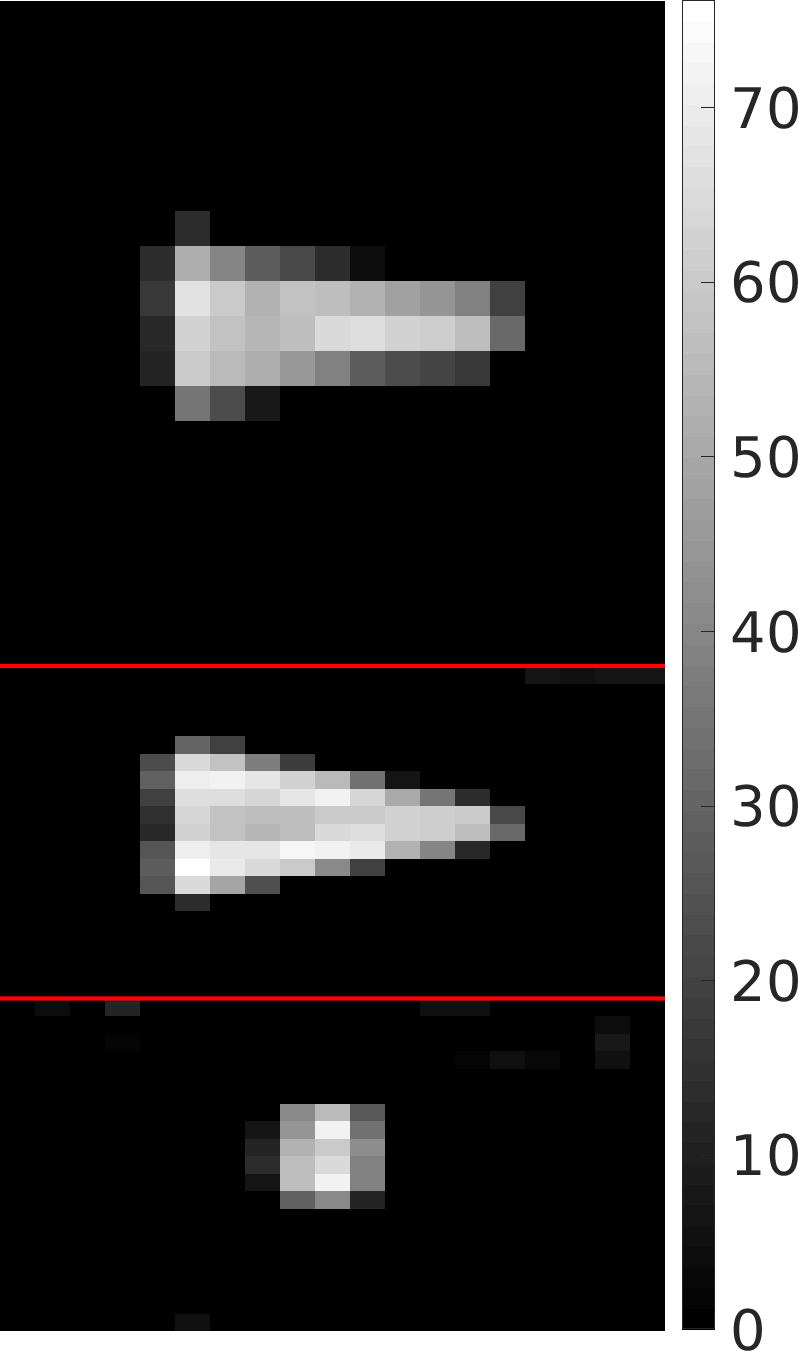}&
 \includegraphics[width=0.14\textwidth]{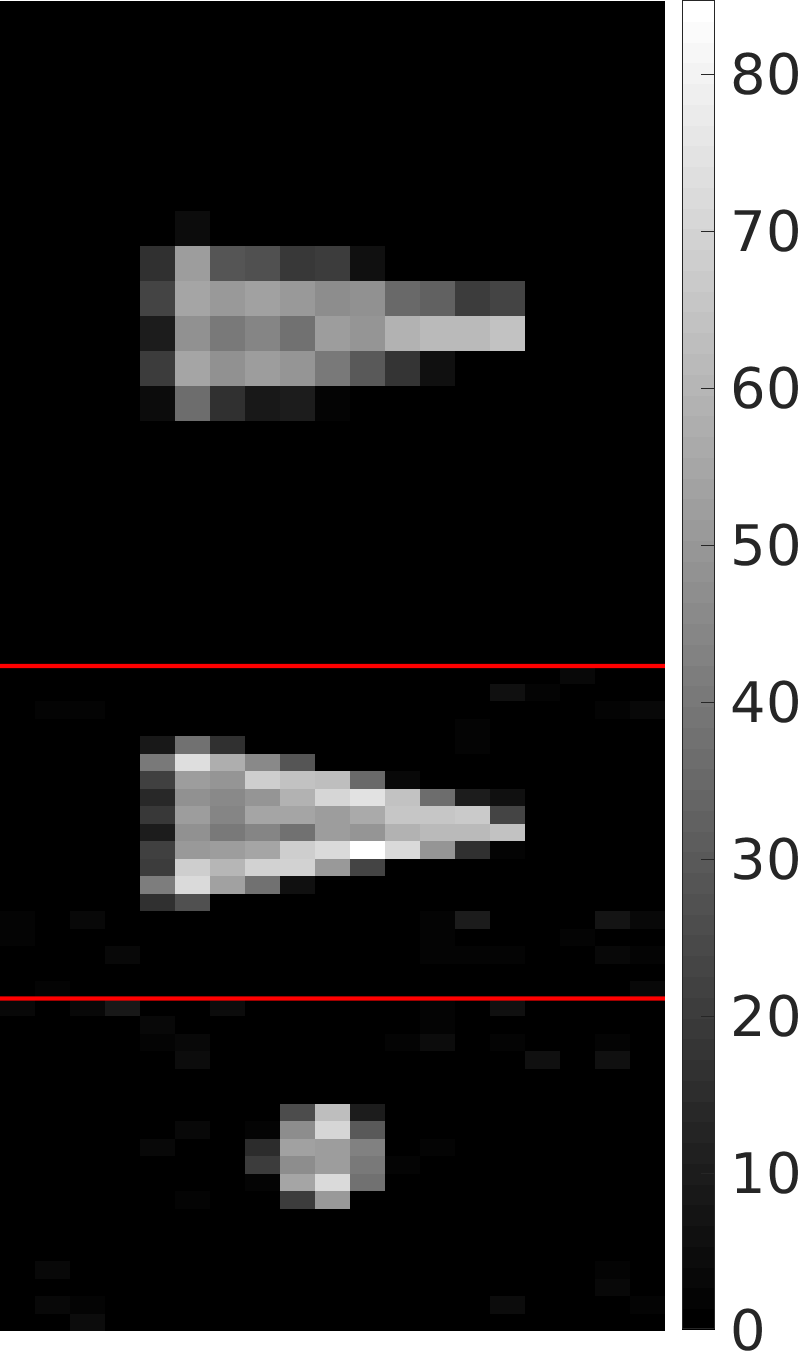}\\
\end{tabular}
\caption{STD reconstruction non-whitened/whitened of the ``shape'' phantom. Illustration structure as in Fig.
\ref{fig:phantom_shape}(right). Concentration in mmol/l.    }
\label{fig:std_reco}
\end{figure}

\subsection{Acceleration via randomized SVD}

\begin{figure}[hbt!]
\centering
\begin{minipage}{8cm}
\includegraphics[width=8cm]{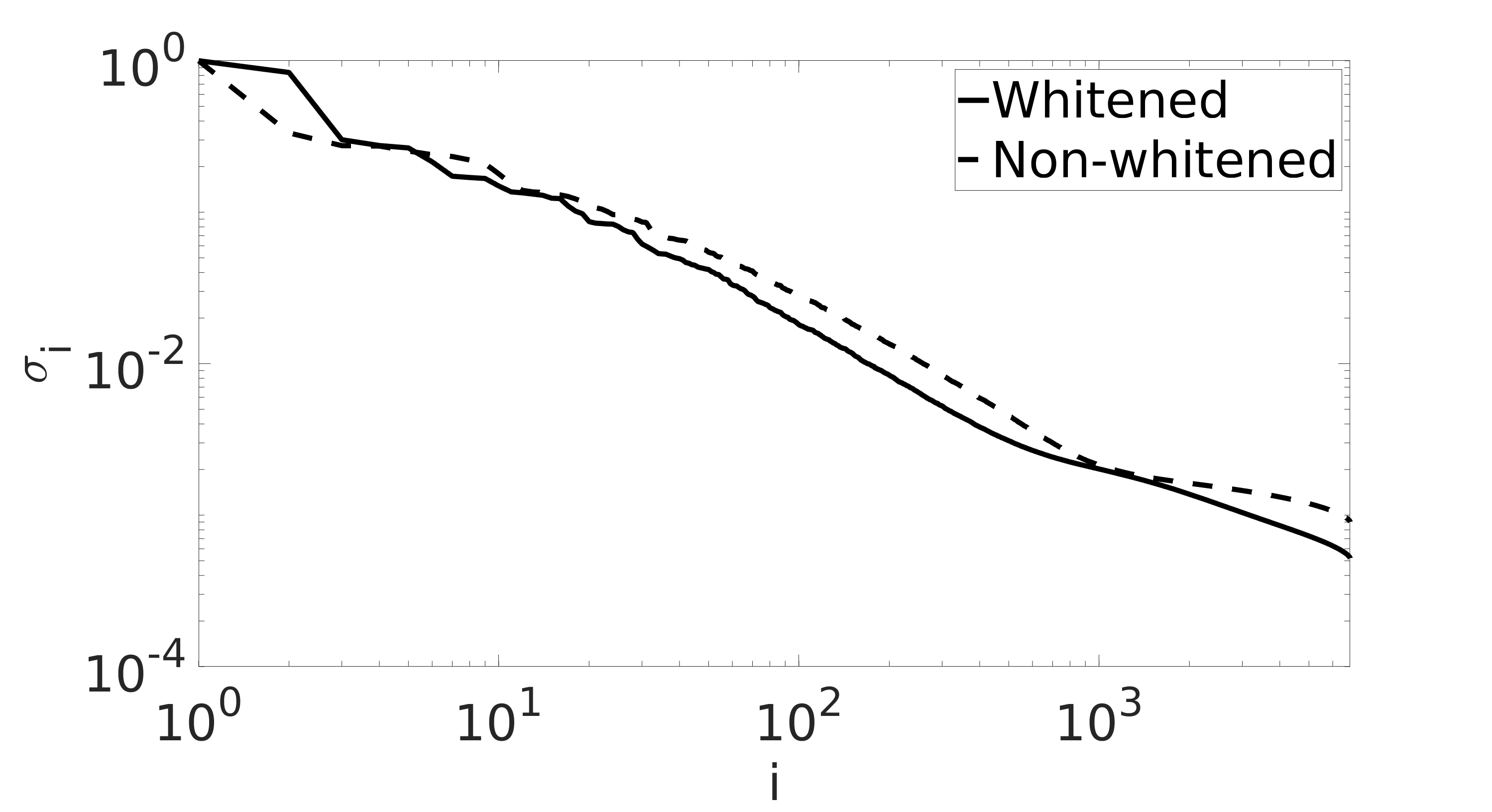}
\end{minipage}
\begin{minipage}{8cm}
 \begin{tabular}{l|cc}
 &\multicolumn{2}{c}{$(\sum_{i=1}^k \sigma_i^2)/ (\sum_{i=1}^m \sigma_i^2)$} \\
$k$ & Whitened & Non-whitened \\
\hline
500 & 99.23 & 99.61 \\
1000 & 99.46 & 99.73 \\
1500 & 99.55 & 99.80 \\
2000 &  99.61 &  99.85 \\
\hline
 \end{tabular}
\end{minipage}
\caption{Illustration of singular value decay (left) of the system $A\in \mathbb{R}^{n\times m}$.
Table (right) including energy percentage $(\sum_{i=1}^k \sigma_i^2)/ (\sum_{i=1}^m \sigma_i^2)$
for low rank approximations of $A\in \mathbb{R}^{n\times m}$; relevant for constructing randomized
SVD approximation. }
\label{tab:sv_energy}
\end{figure}

Now we illustrate randomized SVD for accelerating the Kaczmarz algorithm, and discuss its interplay with whitening.
In Fig. \ref{tab:sv_energy}, we plot the singular values (SVs) of the non-whitened and whitened system matrices. The
SVs decay algebraically with comparable decay rates for both cases, indicating that the MPI inverse 
problem is mildly ill-posed. A useful quantitative measure of the low-rank approximation is the percentage 
$(\sum_{i=1}^k \sigma_i^2)/ (\sum_{i=1}^m \sigma_i^2)$, which roughly corresponds to the optimal error bound 
on the rank-$k$ approximation in the Frobenius norm. According to the table in Fig. \ref{tab:sv_energy}, five 
hundred SVs capture nearly all the energies for both whitened and non-whitened cases and thus can give an accurate 
low rank approximation. 
Interestingly,
in the whitened case, the same number of SVs can capture more energy percentage than that for the non-whitened case,
and thus whitening may yield slightly more accurate low-rank approximations and hold more potential for speedup.
The decay behavior justifies the use of the rSVD approach for accelerating the algorithm in Section \ref{ssec:rsvd}.
In MPI, the SV decay was rigorously proved for simplified models in \cite{ErbWeinmann:2018} and \cite{KluthJinLi:2017}
for the one-dimensional and multi-dimensional cases, respectively.

In view of the SV decay in Fig. \ref{tab:sv_energy} and the energy percentage shown in the table therein, two truncation numbers,
i.e., $k=500$ and $k=1000$, are employed below for the accelerated reconstruction. The numerical results are
presented in Figs. \ref{fig:methods_alpha1}--\ref{fig:methods_alpha3}
for three different $\alpha$ values (as were used in Fig. \ref{fig:std_reco}), which represent over, medium
and under- regularization, respectively. In these different regimes, the behavior of the reconstruction algorithms
differs slightly.

With the $\alpha$ value properly chosen (i.e., $\alpha=9.77\times10^{-2}$), rSVD1 can provide reasonable
reconstructions in the whitened case for both $k$ values (cf. Fig. \ref{fig:methods_alpha2}), and the reconstructions
are comparable with that by STD in Fig. \ref{fig:std_reco}. However, in the non-whitenend case, slight
blurring appears in the reconstructed cone.  With the choice $k=1000$, SNR gives comparable reconstructions,
but with $k=500$, either significant distortions or smoothing appear in the reconstructions for all three $\alpha$ values,
and thus the choice $k=500$ seems insufficient to capture the essential information of the data. Thus, 
rSVD is more effective in compressing the data than SNR. Somewhat surprisingly, rSVD2, the simplest and fastest
approach, can also provide reasonable reconstructions of comparable quality, even for the over-regularized case, but
in the non-whitenend case, it gives reconstructions containing pronounced background artifacts, which, however, can be
greatly reduced in the presence whitening; see Figs. \ref{fig:methods_alpha2} and \ref{fig:methods_alpha3}. This clearly
shows the significant potential of the strategy whitening + SVD2 for MPI reconstruction.

\begin{figure}
\begin{tabular}{ccc|ccc}
\multicolumn{3}{c|}{Non-whitened} & \multicolumn{3}{c}{Whitened}\\
\hline
SNR & rSVD1 & rSVD2 & SNR & rSVD1 & rSVD2 \\
\hline
\multicolumn{3}{l|}{$k=1000$} & \multicolumn{3}{l}{}\\
 \includegraphics[width=0.14\textwidth]{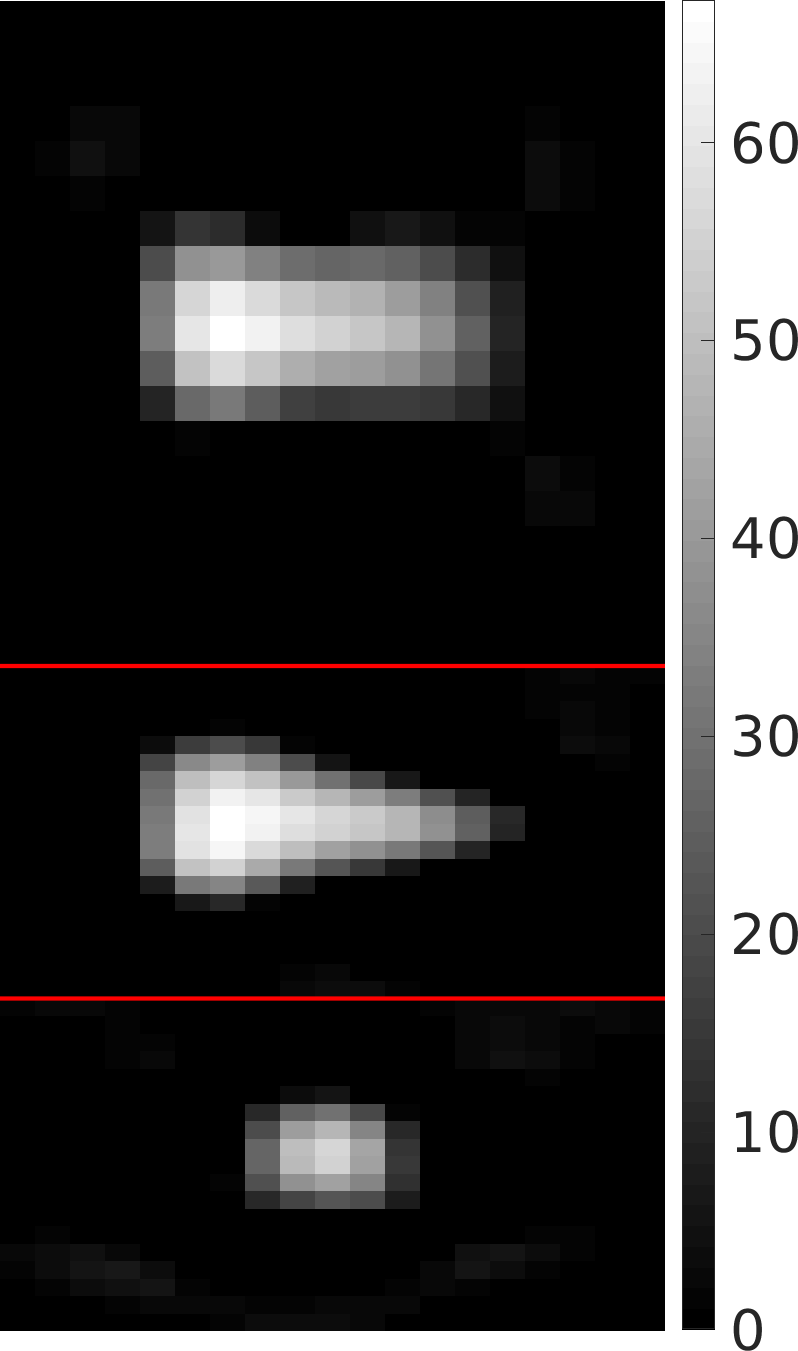}&
 \includegraphics[width=0.14\textwidth]{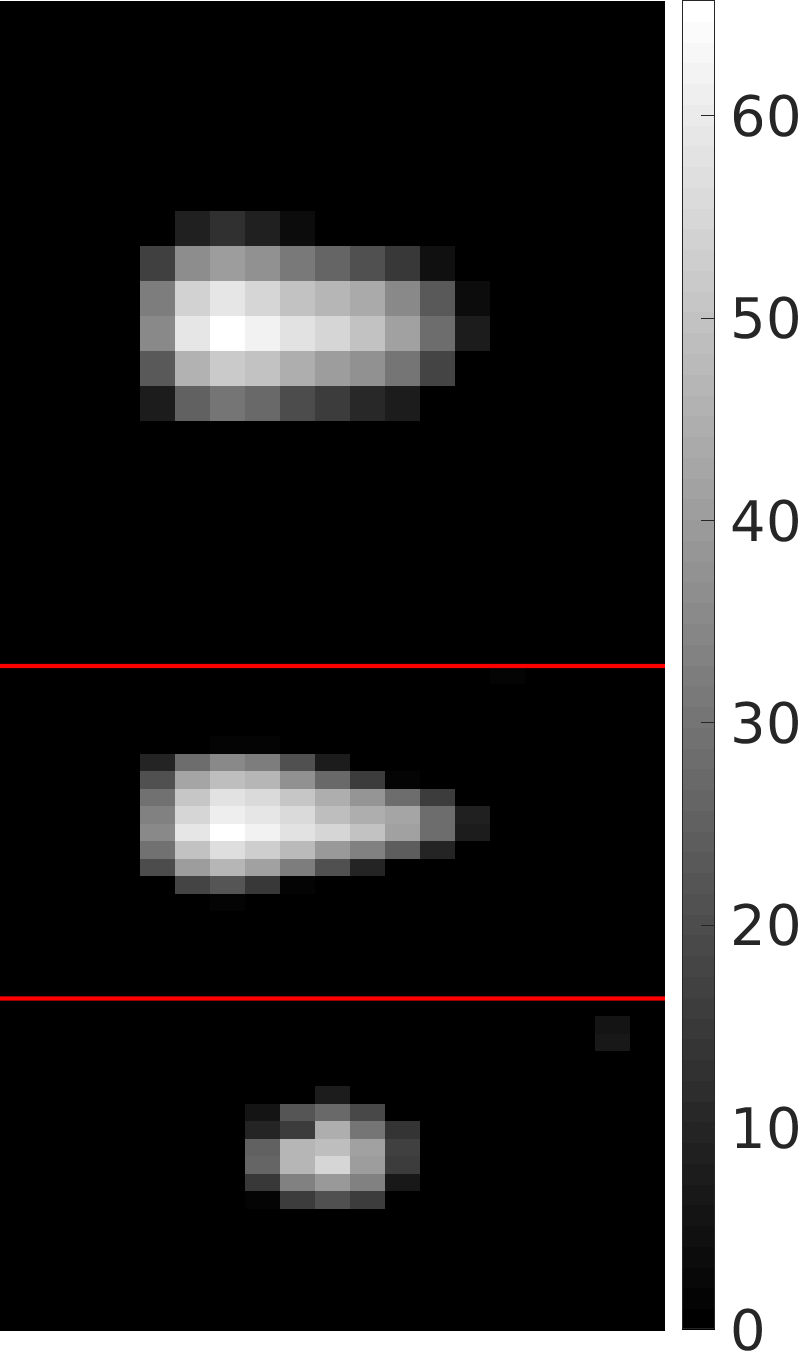}&
 \includegraphics[width=0.14\textwidth]{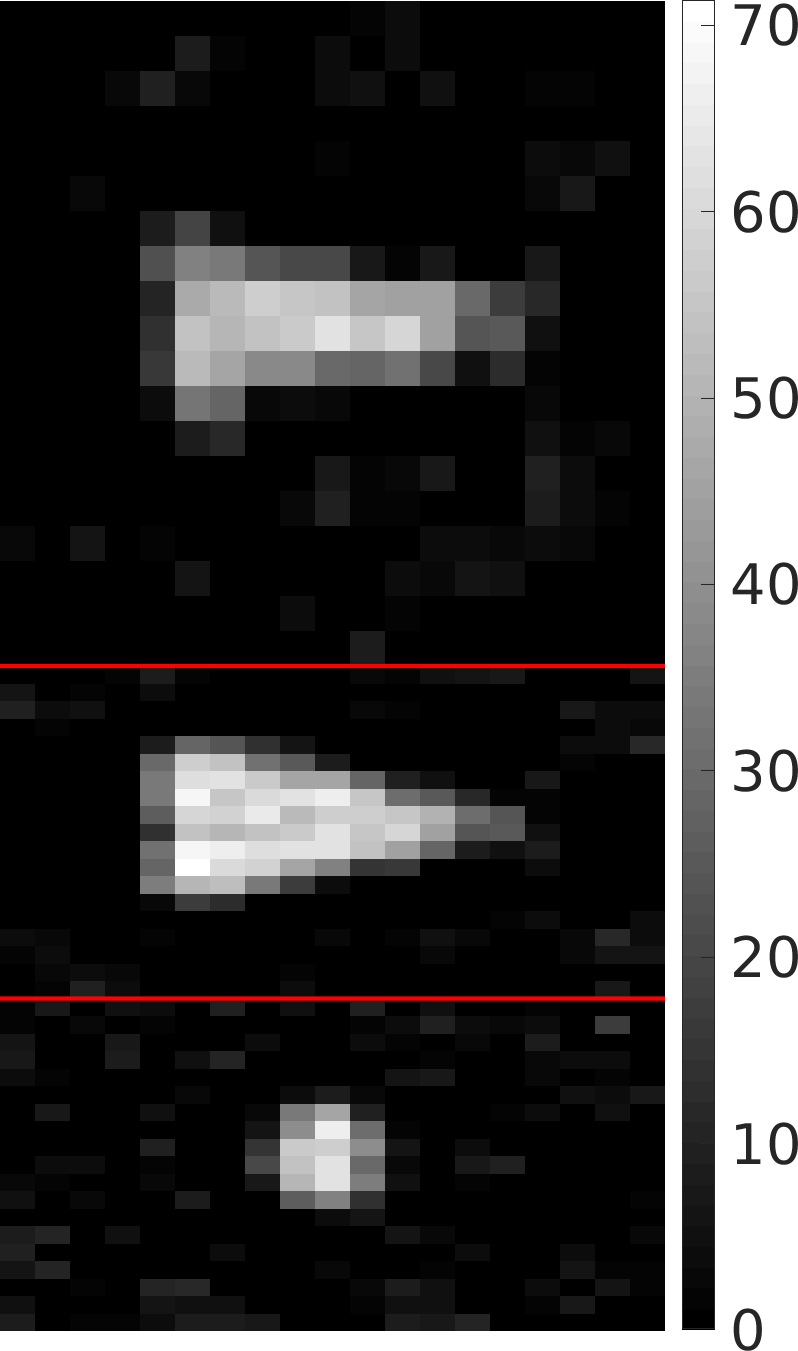}&
 \includegraphics[width=0.14\textwidth]{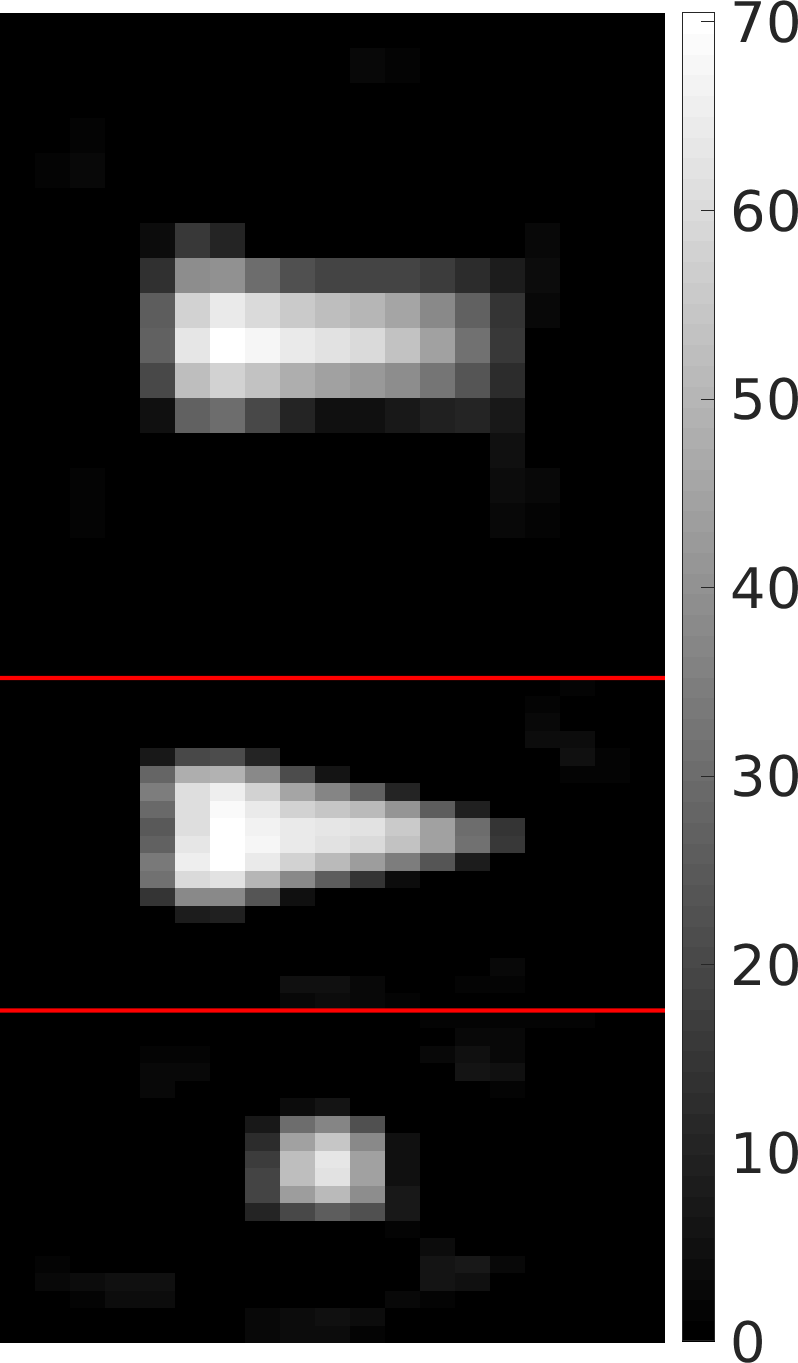}&
 \includegraphics[width=0.14\textwidth]{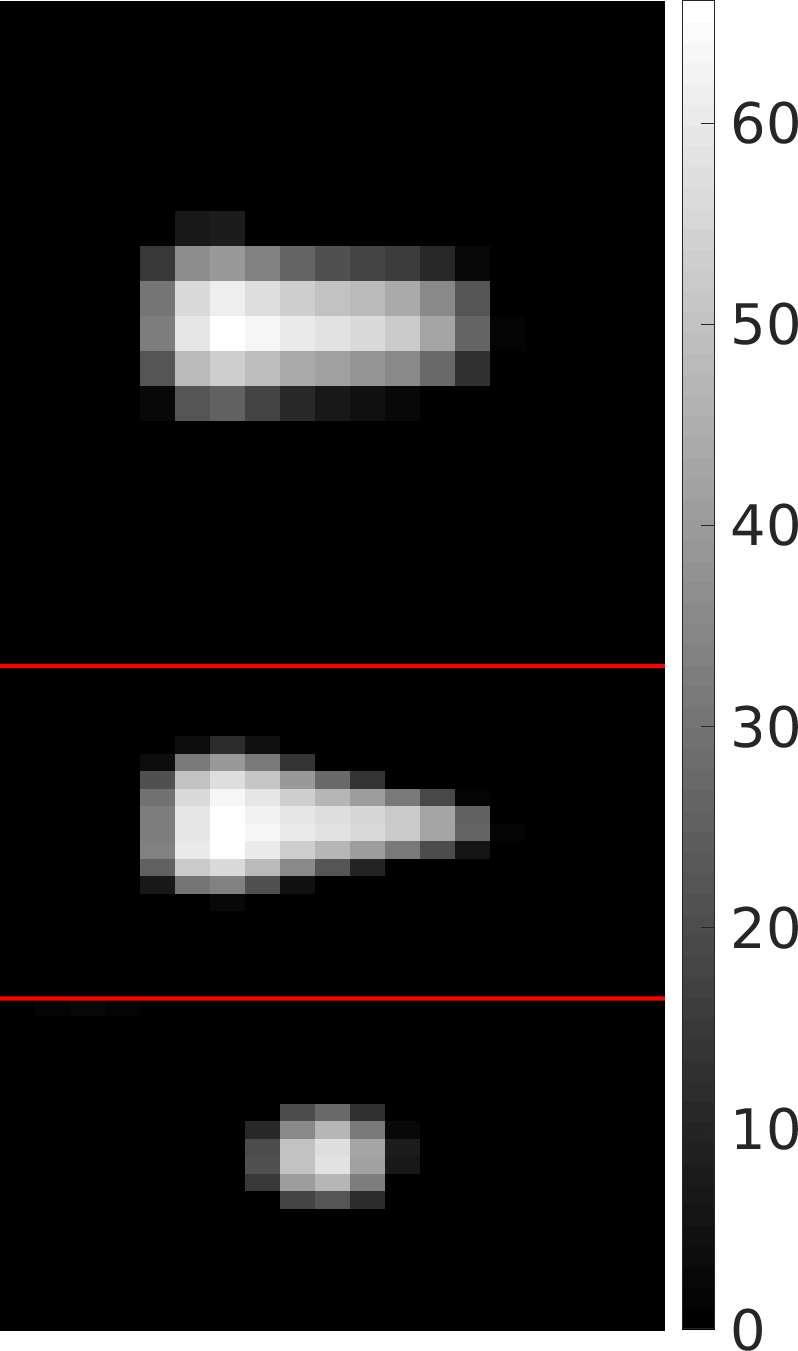}&
 \includegraphics[width=0.14\textwidth]{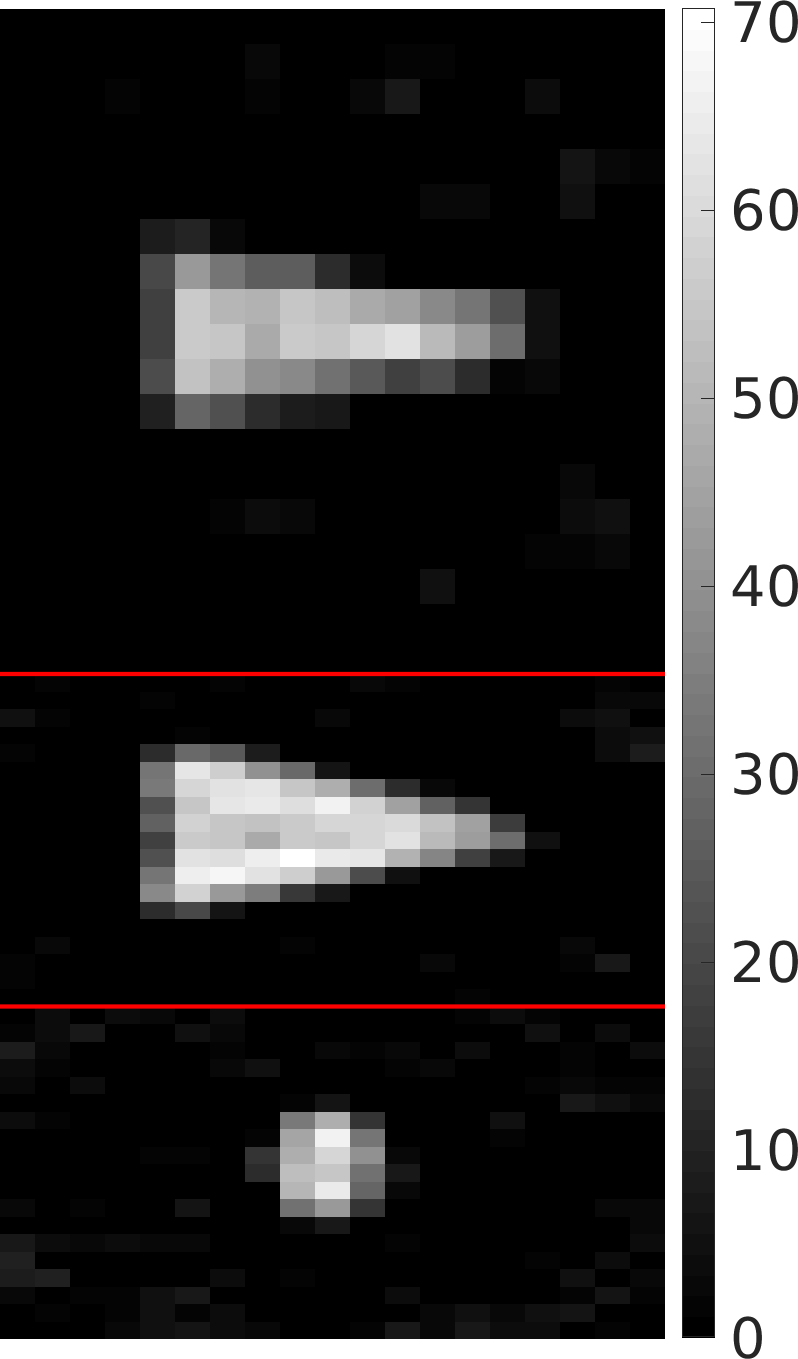}\\
\hline
\multicolumn{3}{l|}{$k=500$} & \multicolumn{3}{l}{}\\
 \includegraphics[width=0.14\textwidth]{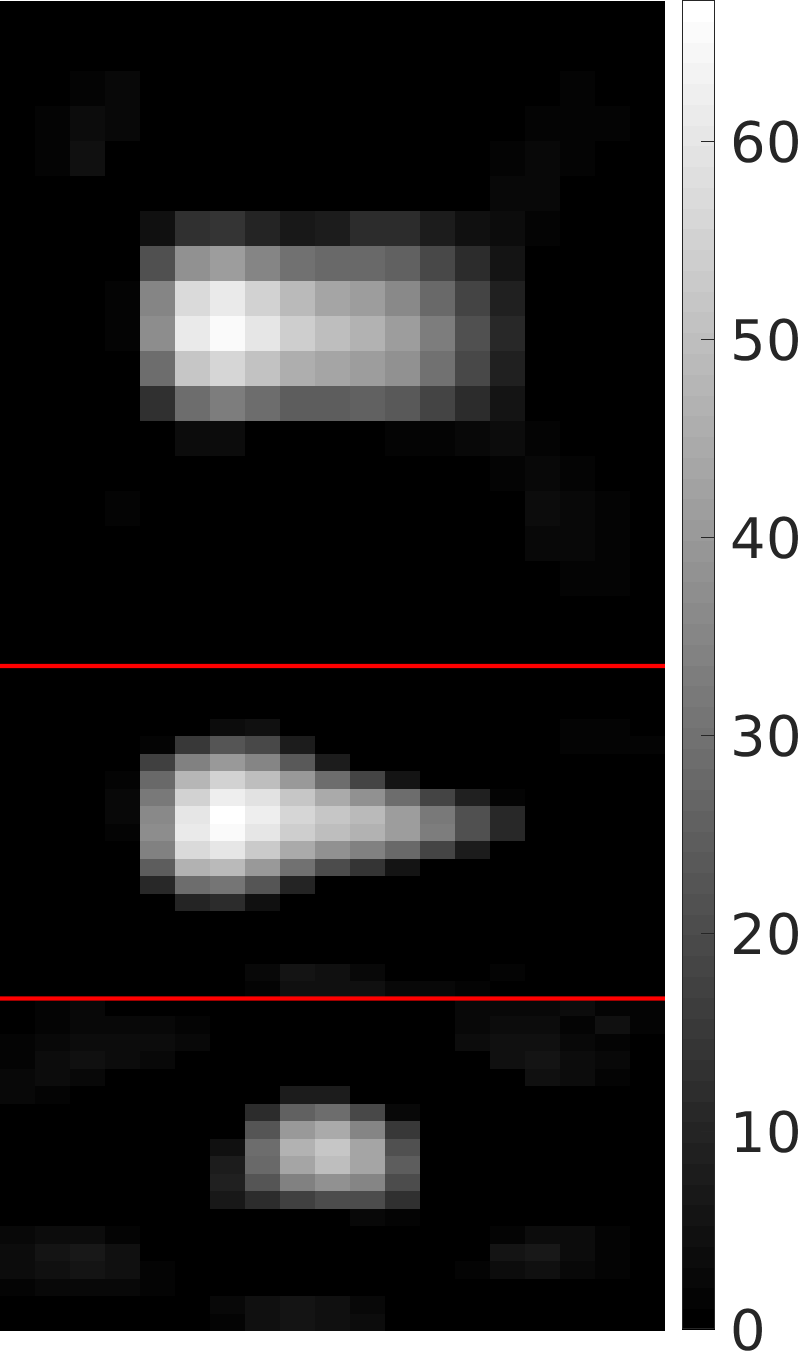}&
 \includegraphics[width=0.14\textwidth]{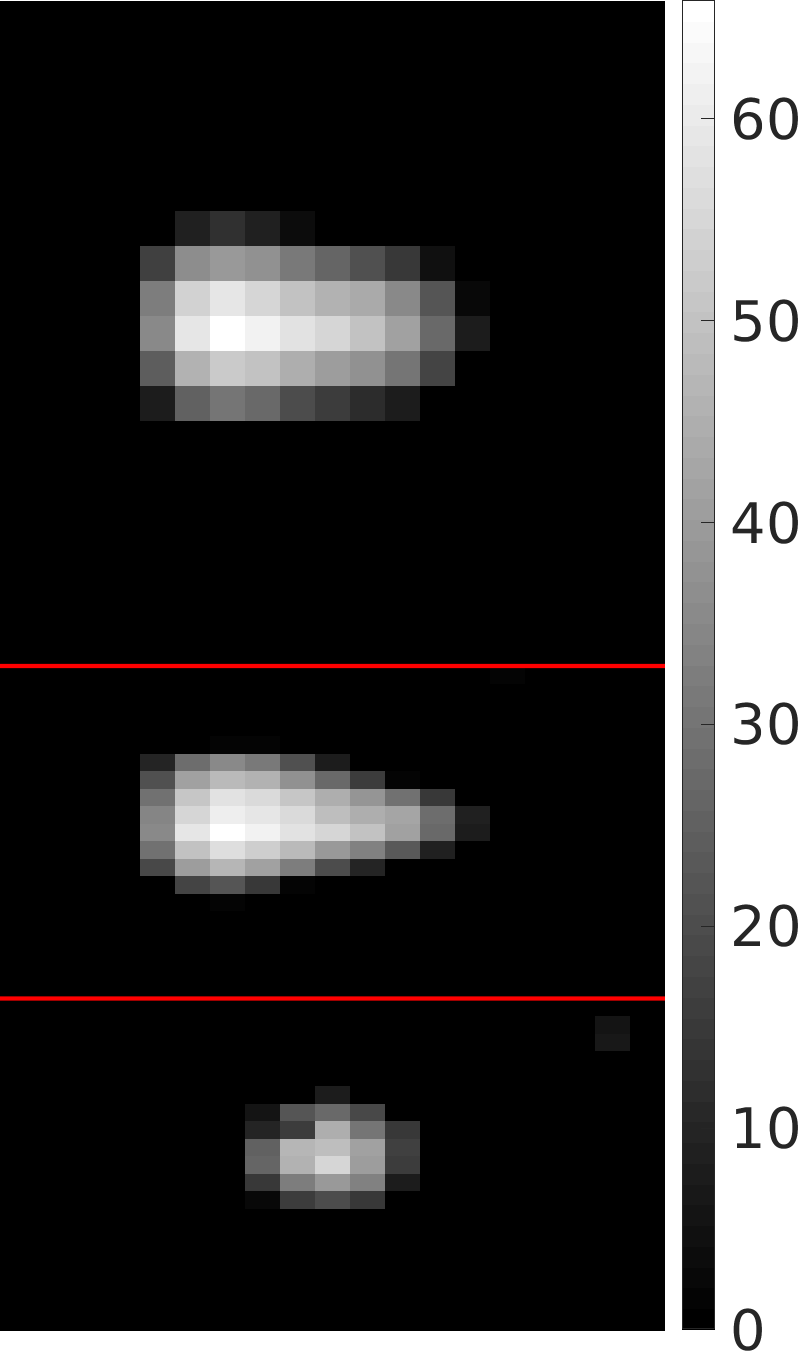}&
 \includegraphics[width=0.14\textwidth]{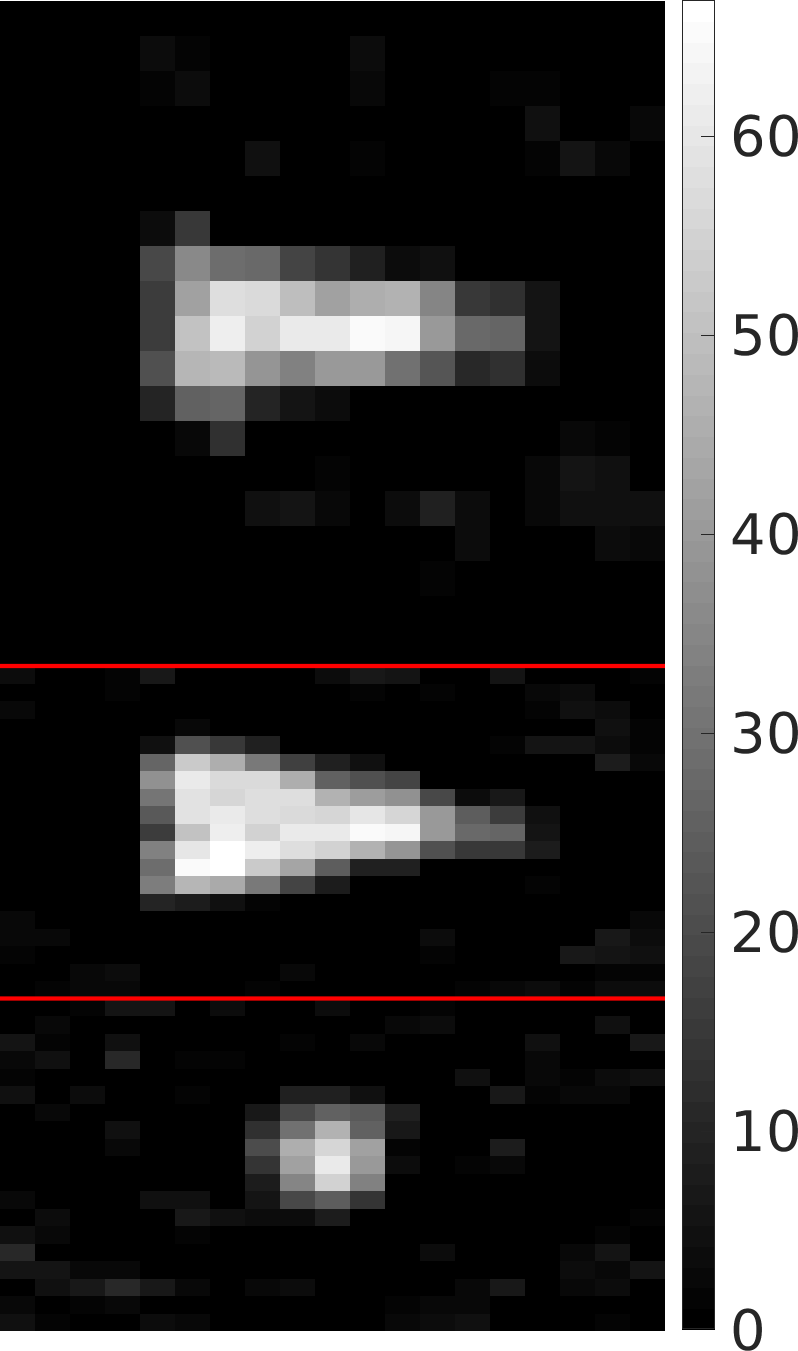}&
 \includegraphics[width=0.14\textwidth]{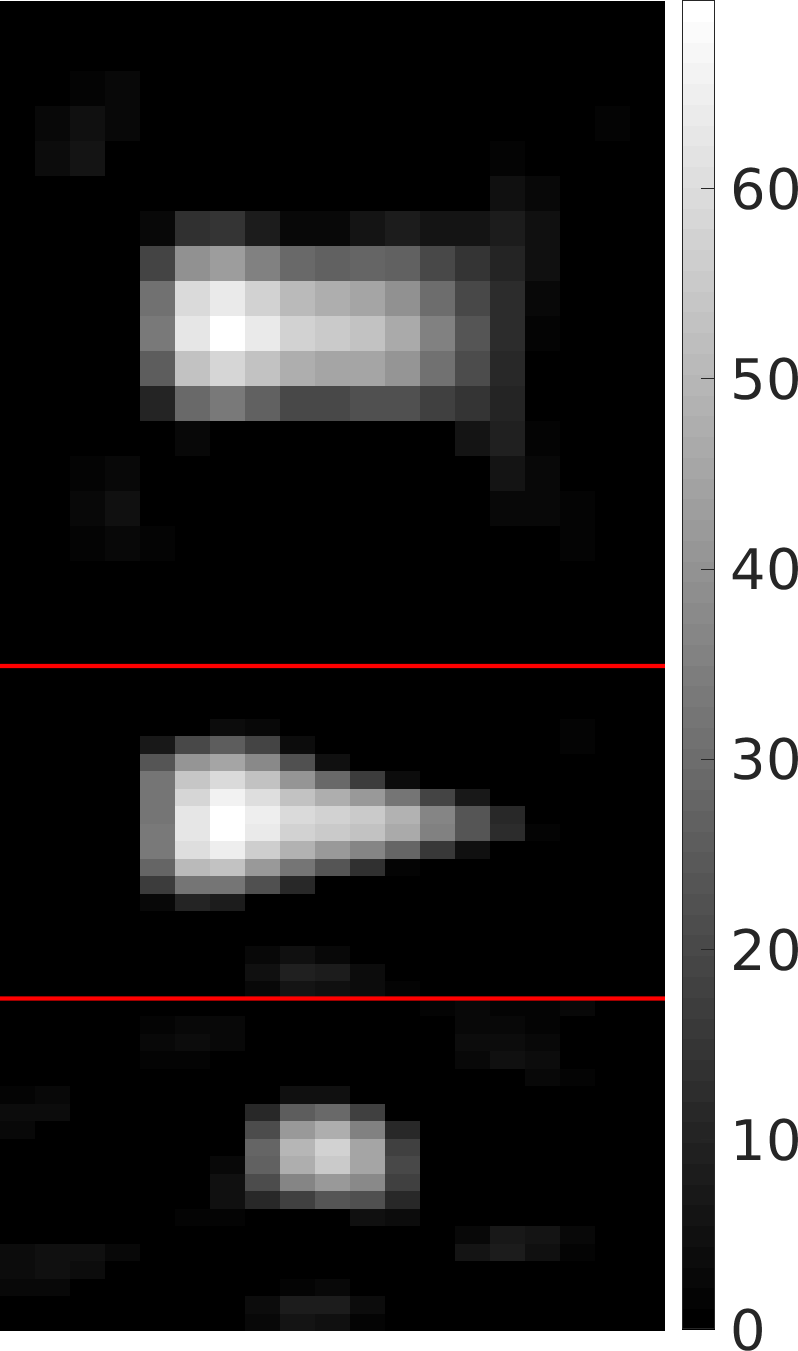}&
 \includegraphics[width=0.14\textwidth]{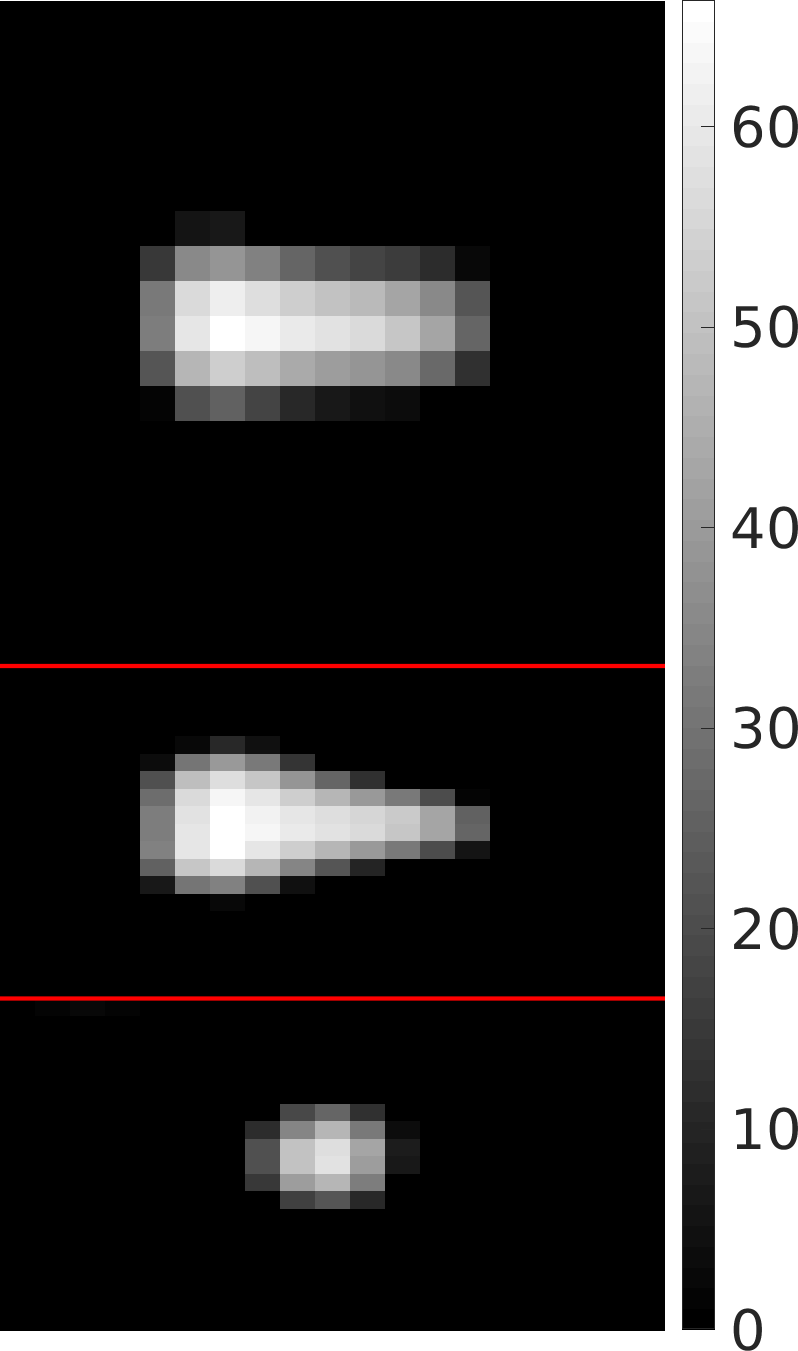}&
 \includegraphics[width=0.14\textwidth]{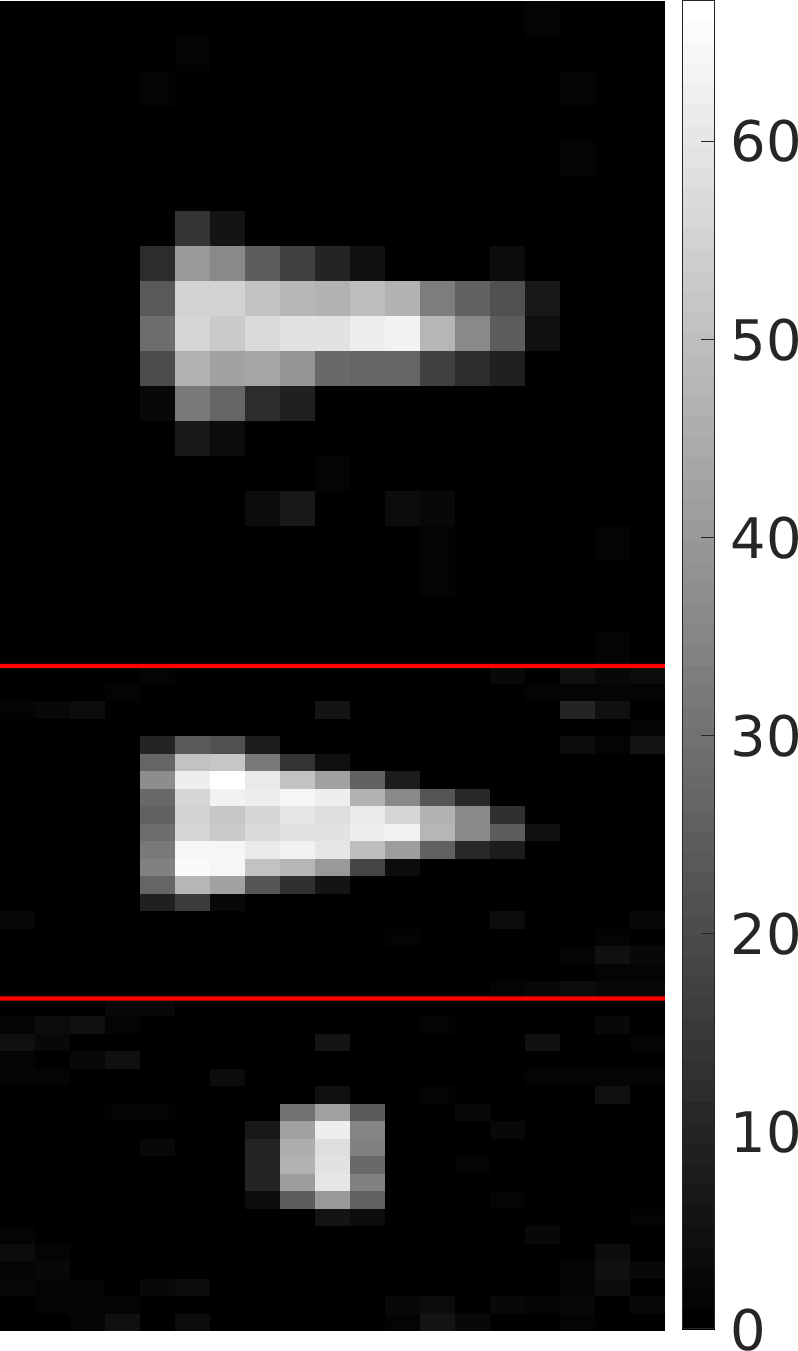}
\end{tabular}
\caption{``Shape'' phantom reconstructions (SNR, rSVD1, rSVD2) for $\alpha=9.77 \PLH 10^{-2}$. Concentration in mmol/l.} 
\label{fig:methods_alpha1}
\end{figure}

\begin{figure}
\begin{tabular}{ccc|ccc}
\multicolumn{3}{c|}{Non-whitened} & \multicolumn{3}{c}{Whitened}\\
\hline
SNR & rSVD1 & rSVD2 & SNR & rSVD1 & rSVD2 \\
\hline
\multicolumn{3}{l|}{$k=1000$} & \multicolumn{3}{l}{}\\
 \includegraphics[width=0.14\textwidth]{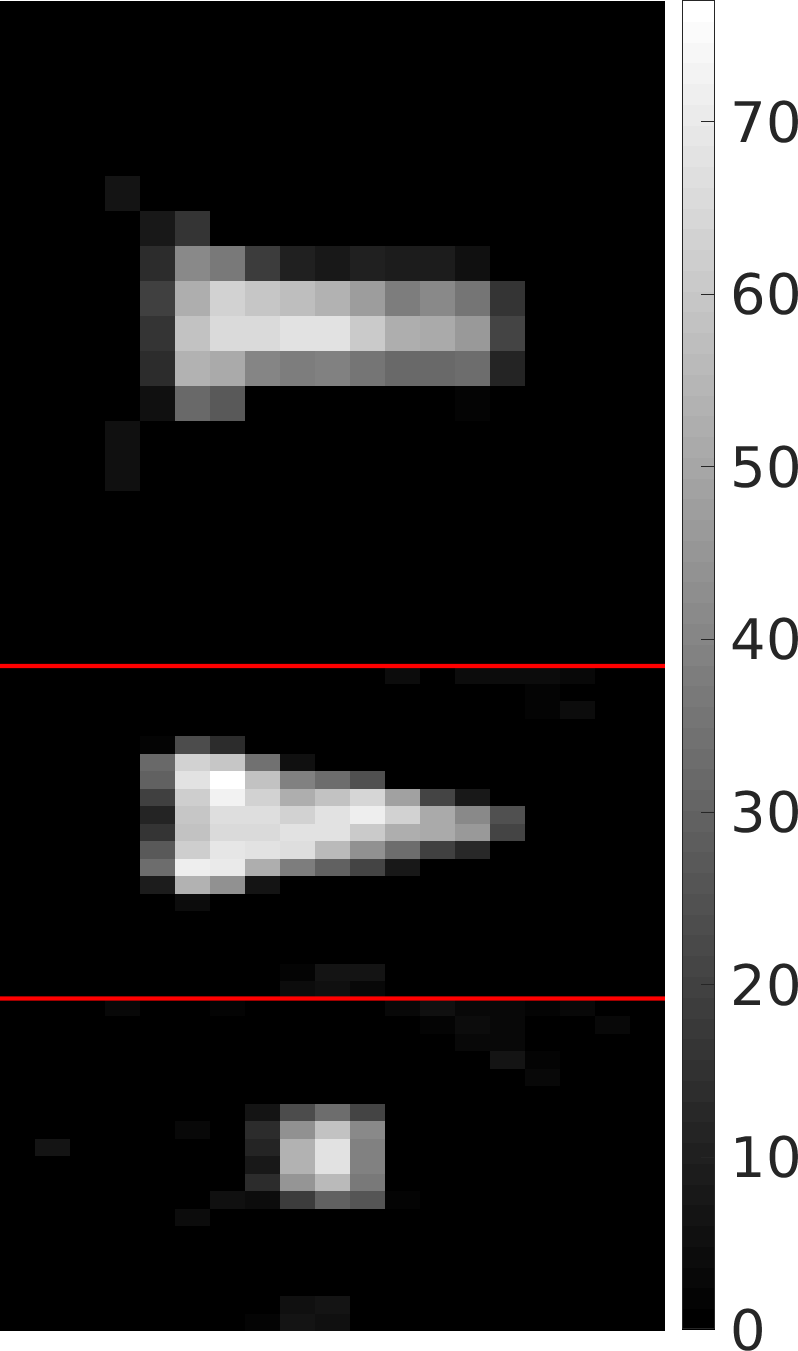}&
 \includegraphics[width=0.14\textwidth]{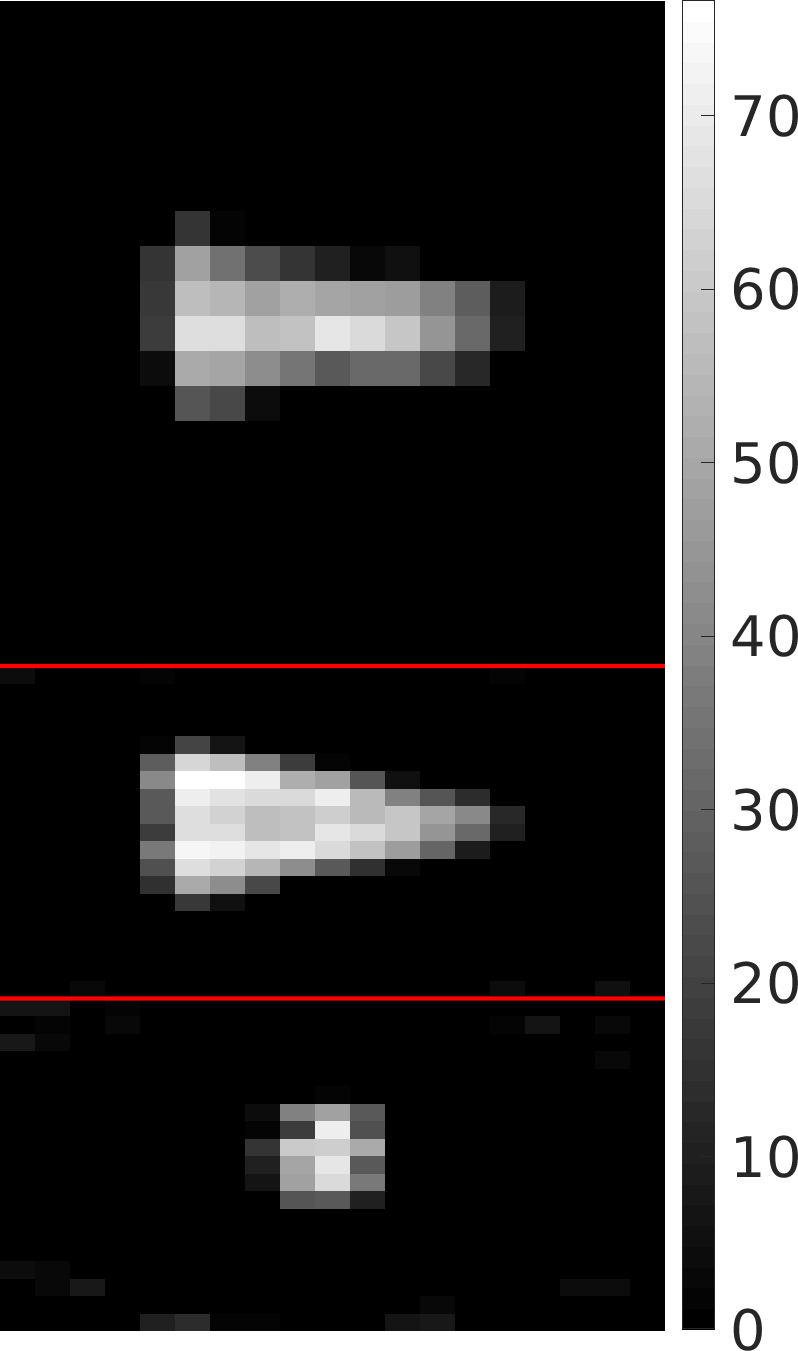}&
 \includegraphics[width=0.14\textwidth]{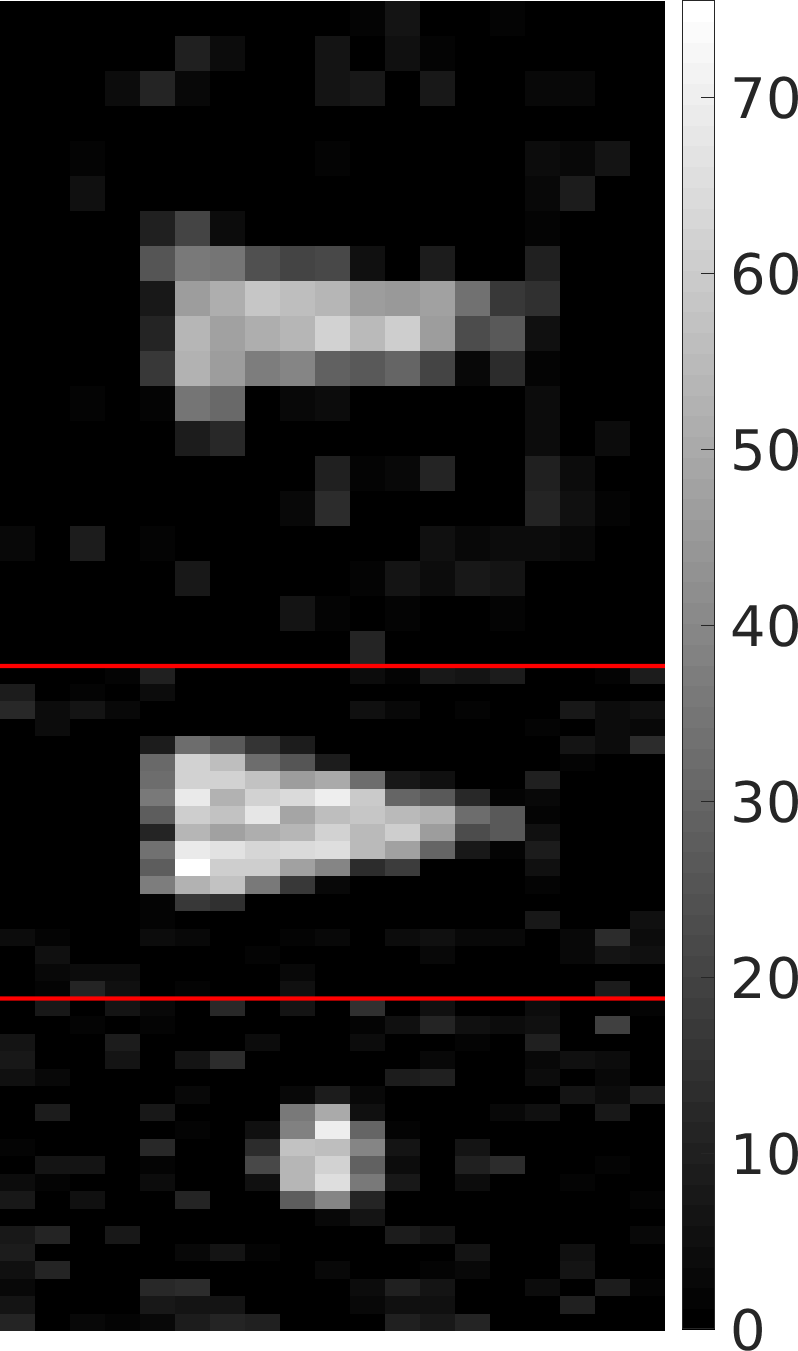}&
 \includegraphics[width=0.14\textwidth]{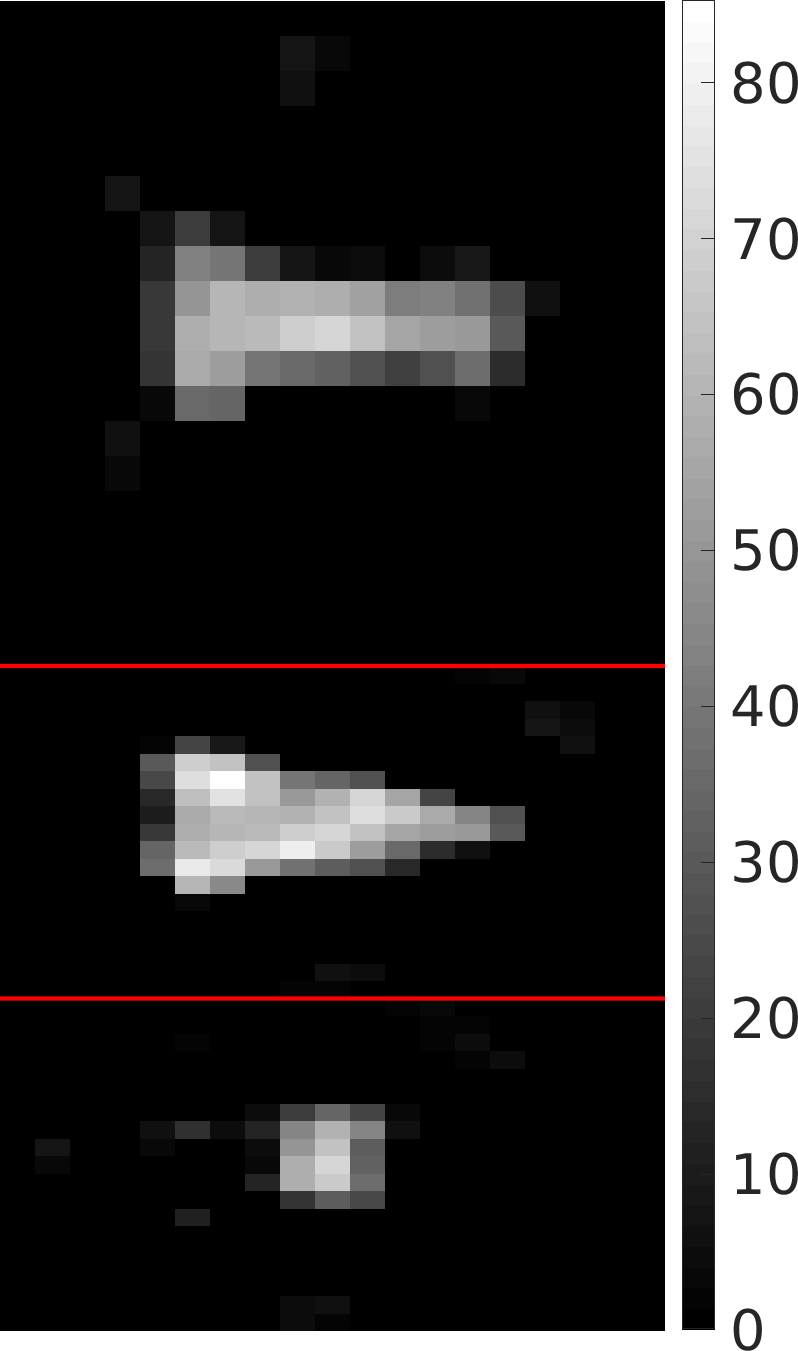}&
 \includegraphics[width=0.14\textwidth]{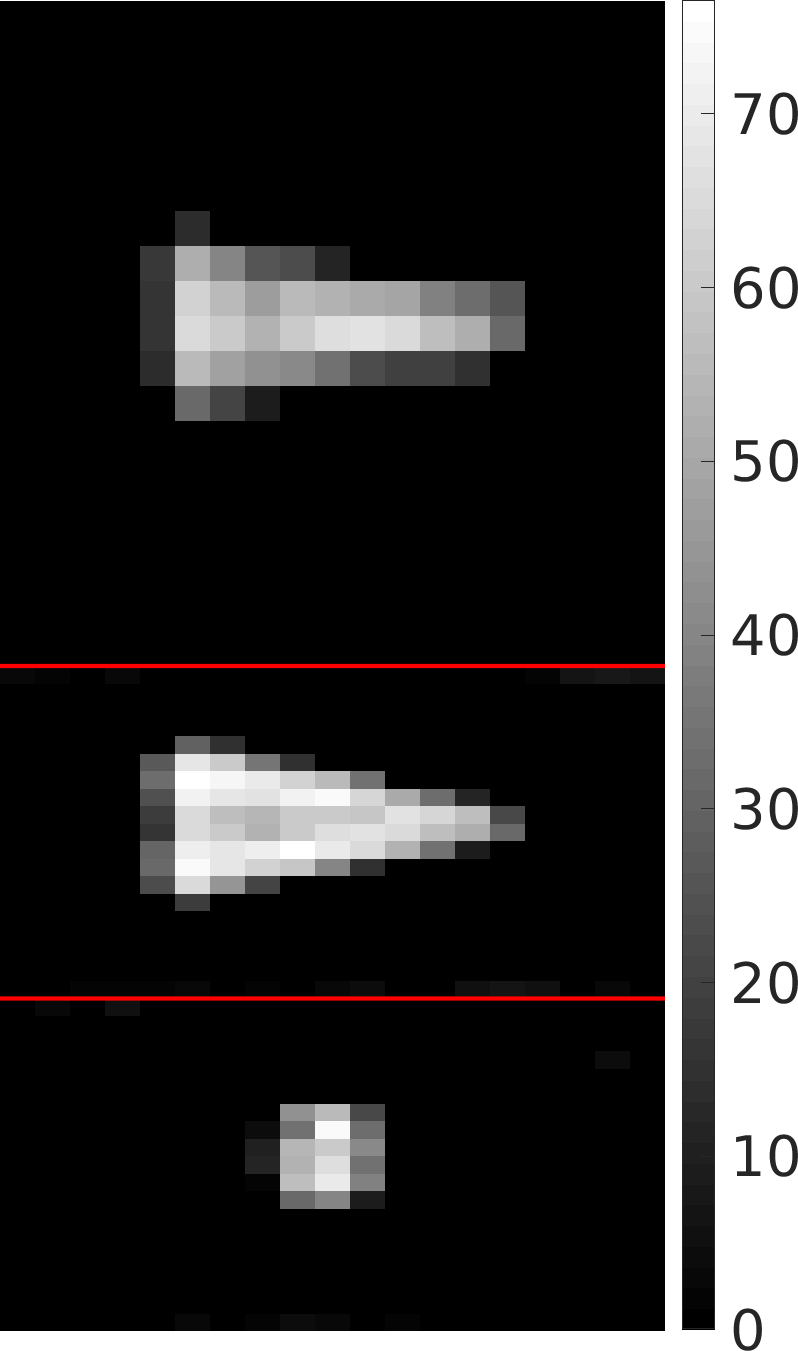}&
 \includegraphics[width=0.14\textwidth]{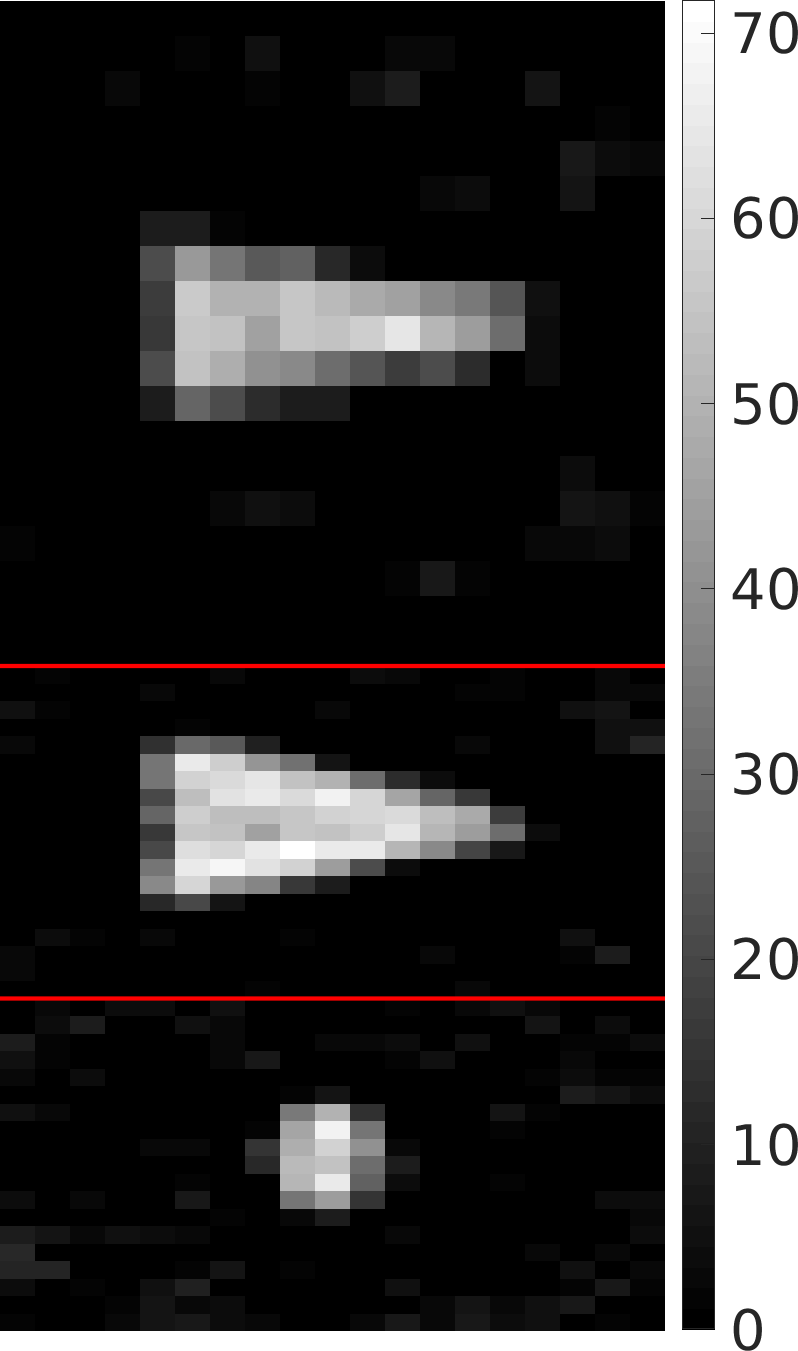}\\
\hline
\multicolumn{3}{l|}{$k=500$} & \multicolumn{3}{l}{}\\
 \includegraphics[width=0.14\textwidth]{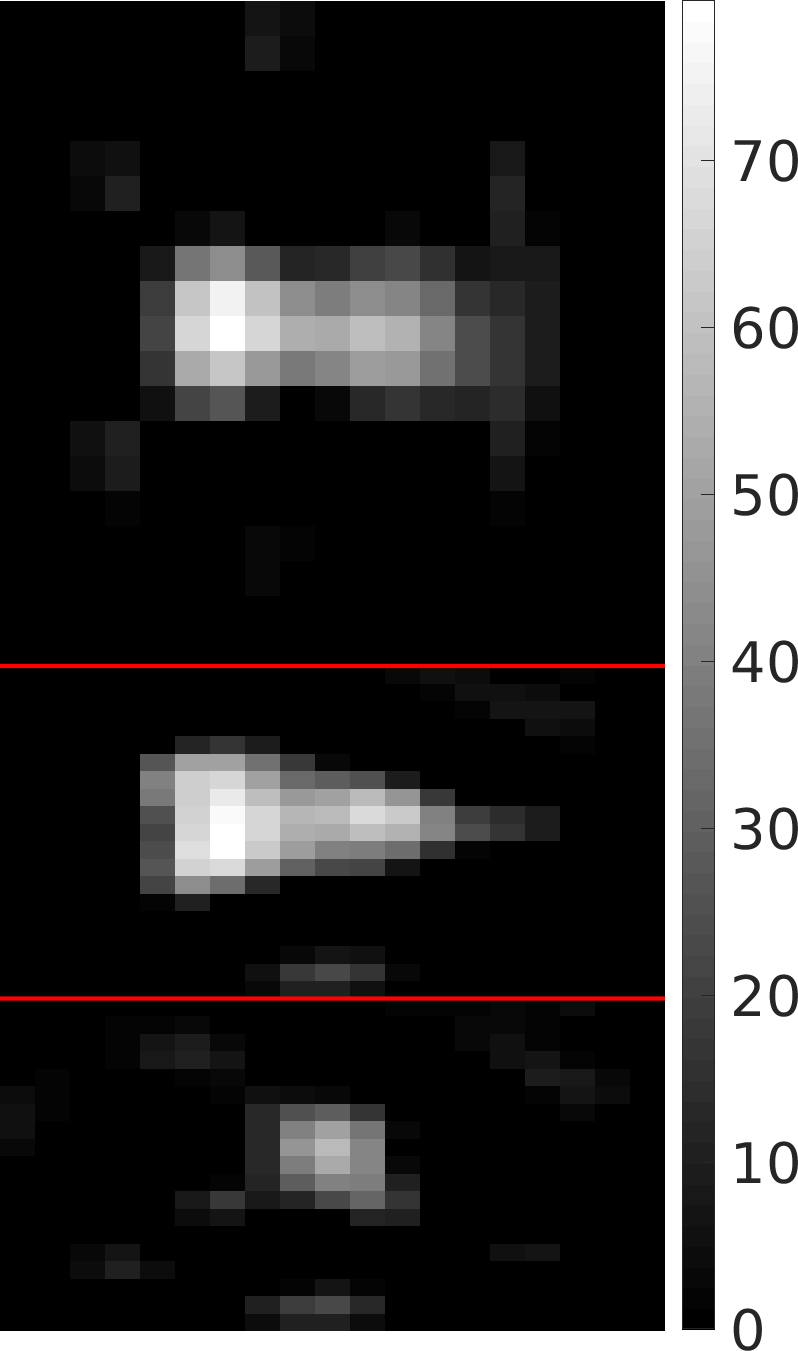}&
 \includegraphics[width=0.14\textwidth]{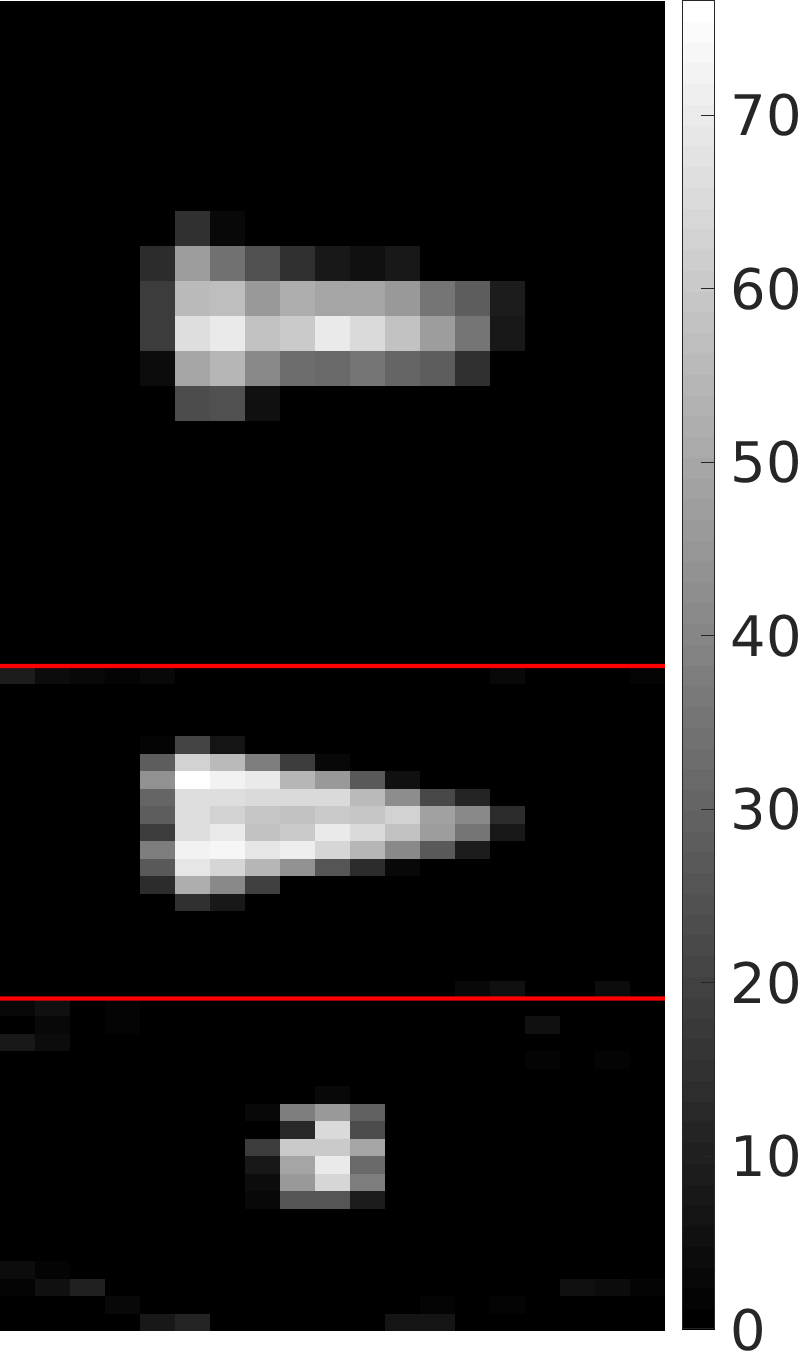}&
 \includegraphics[width=0.14\textwidth]{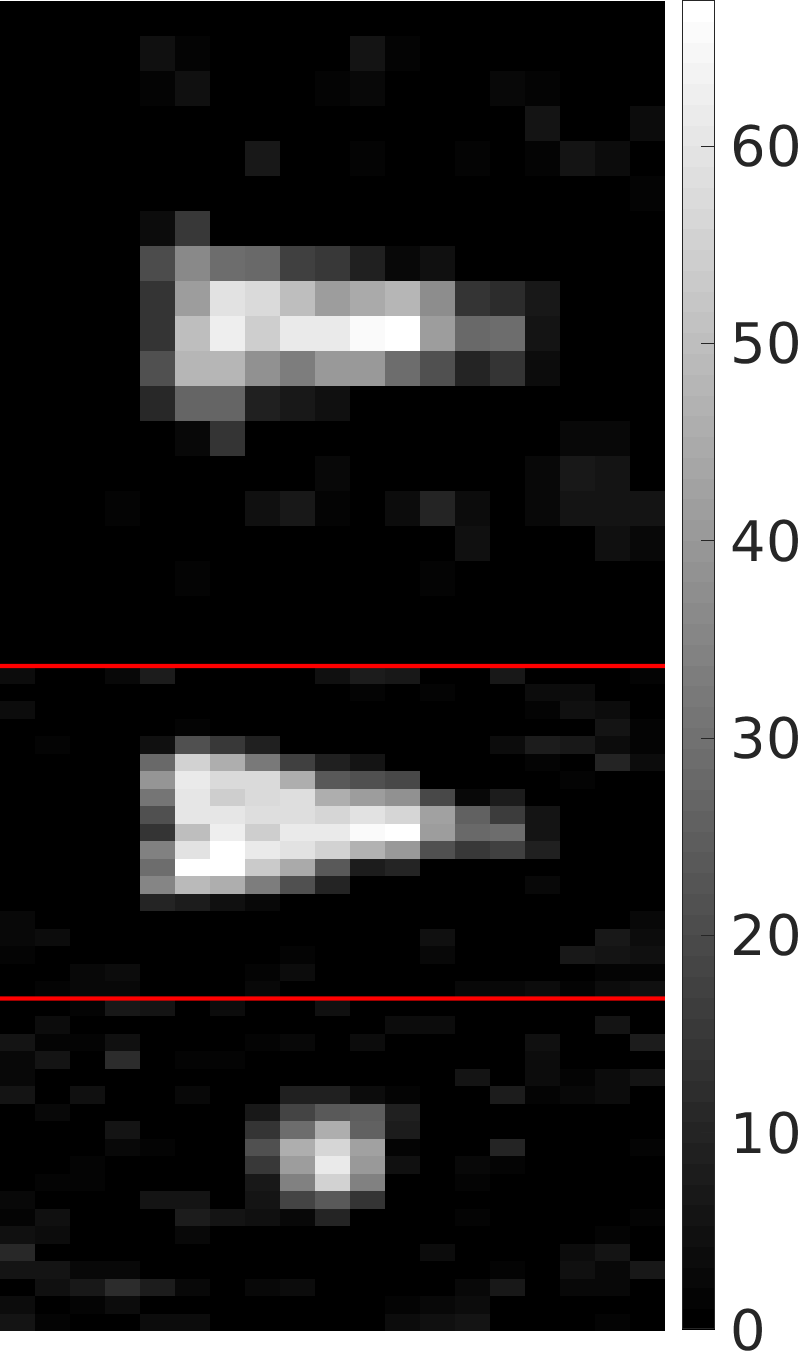}&
 \includegraphics[width=0.14\textwidth]{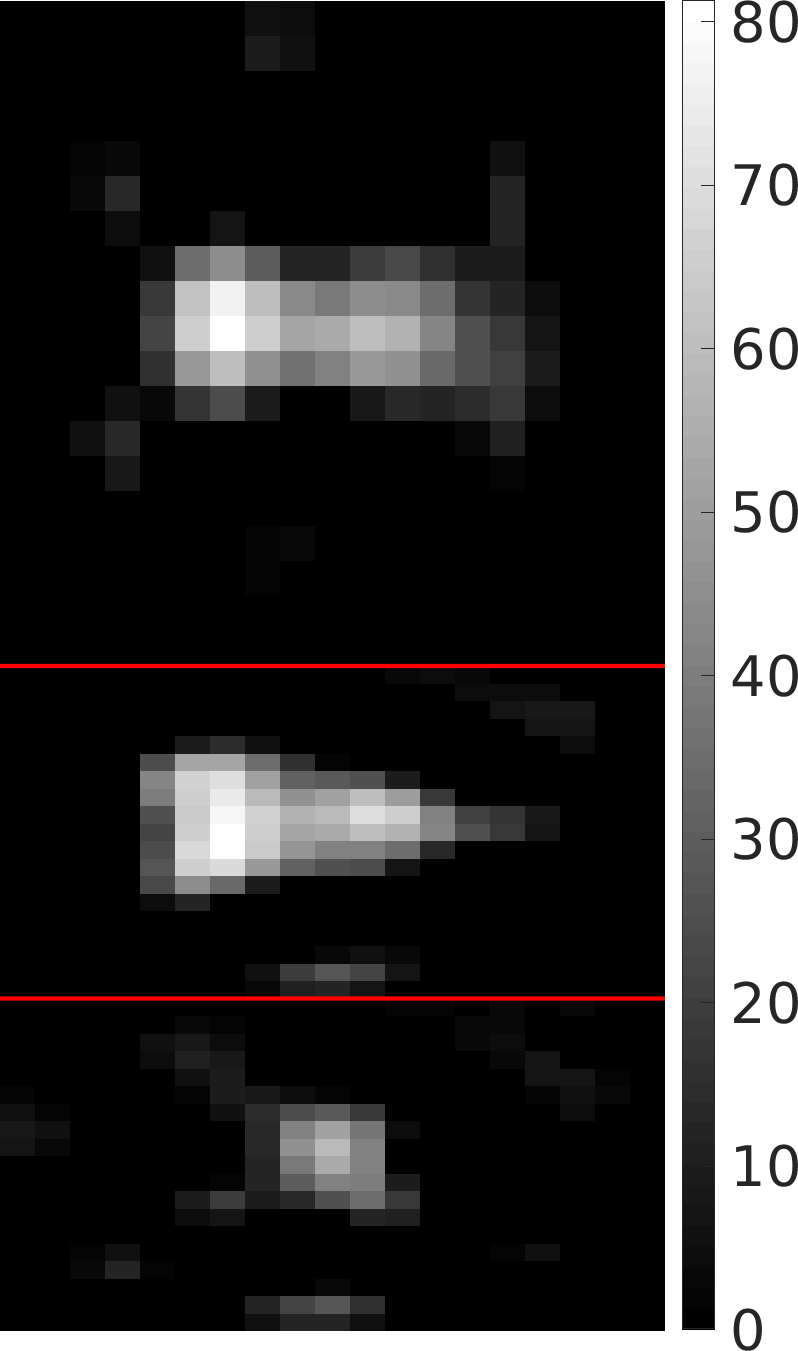}&
 \includegraphics[width=0.14\textwidth]{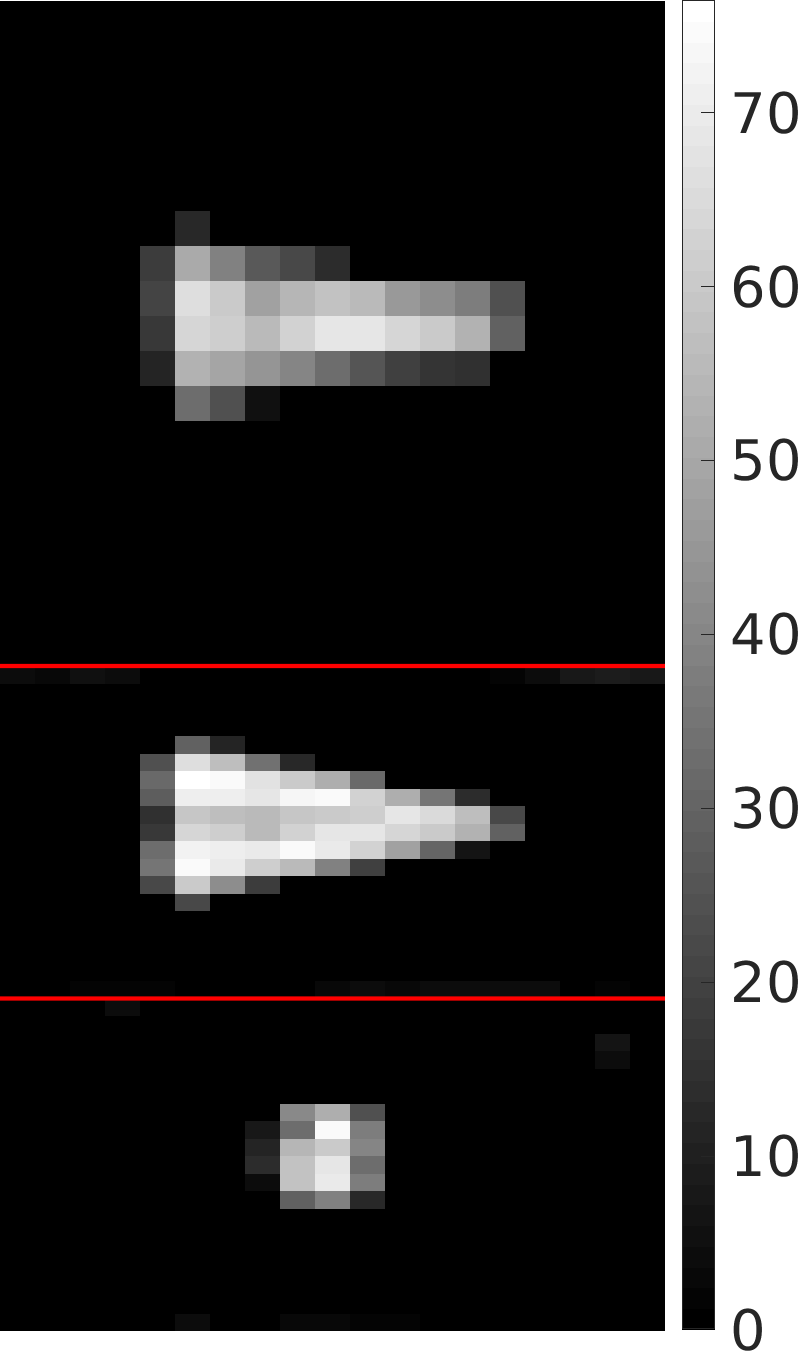}&
 \includegraphics[width=0.14\textwidth]{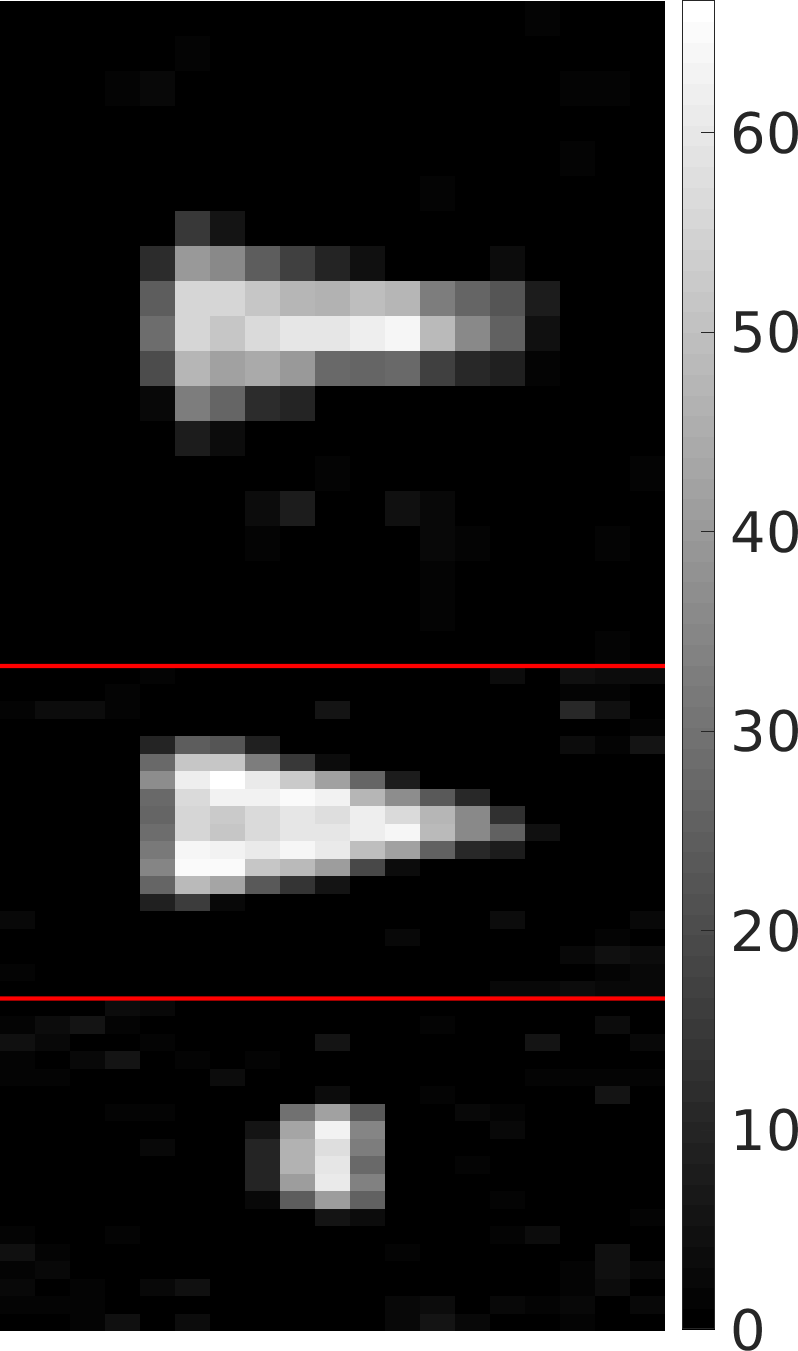}\\
\end{tabular}
\caption{``Shape'' phantom reconstructions (SNR, rSVD1, rSVD2) for $\alpha=3.05 \PLH 10^{-3}$. Concentration in mmol/l.} 
\label{fig:methods_alpha2}
\end{figure}

\begin{figure}
\begin{tabular}{ccc|ccc}
\multicolumn{3}{c|}{Non-whitened} & \multicolumn{3}{c}{Whitened}\\
\hline
SNR & rSVD1 & rSVD2 & SNR & rSVD1 & rSVD2 \\
\hline
\multicolumn{3}{l|}{$k=1000$} & \multicolumn{3}{l}{}\\
 \includegraphics[width=0.14\textwidth]{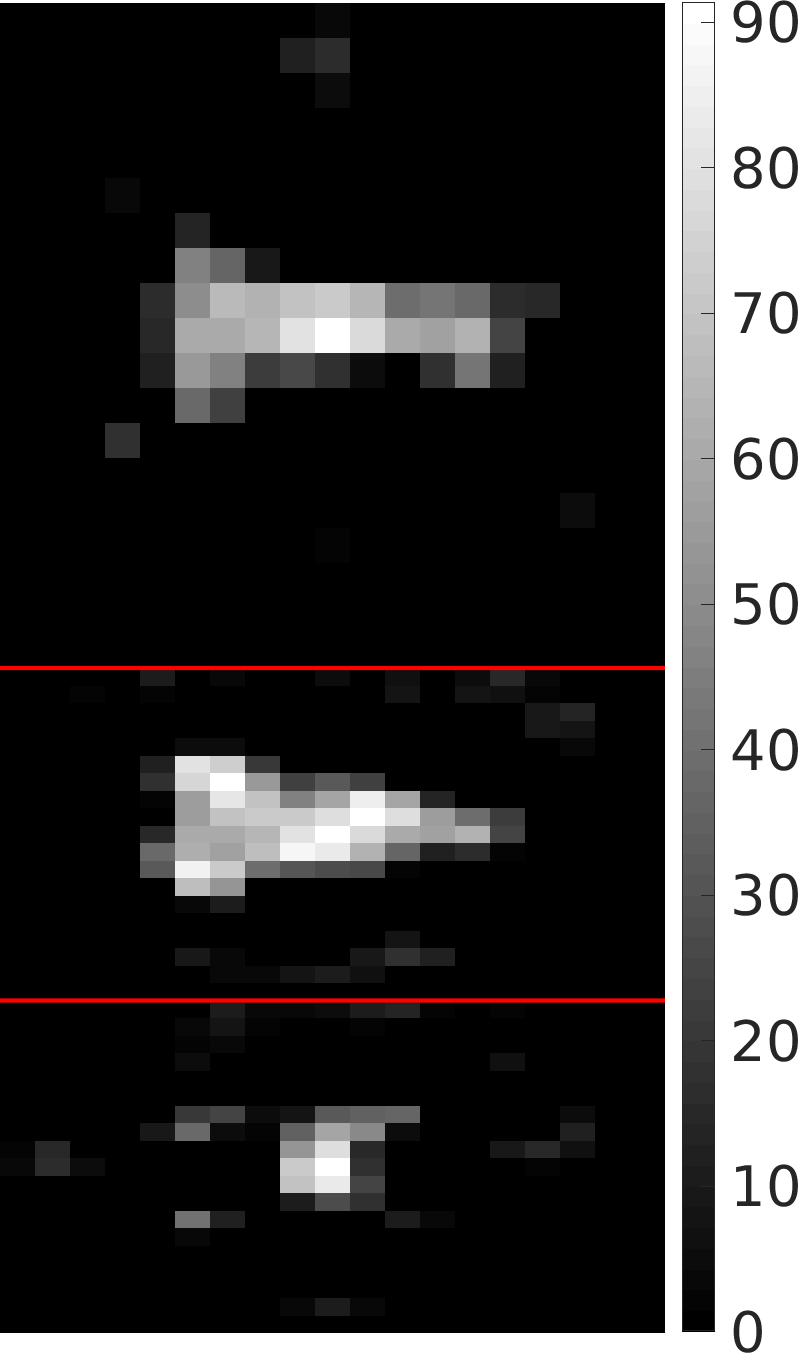}&
 \includegraphics[width=0.14\textwidth]{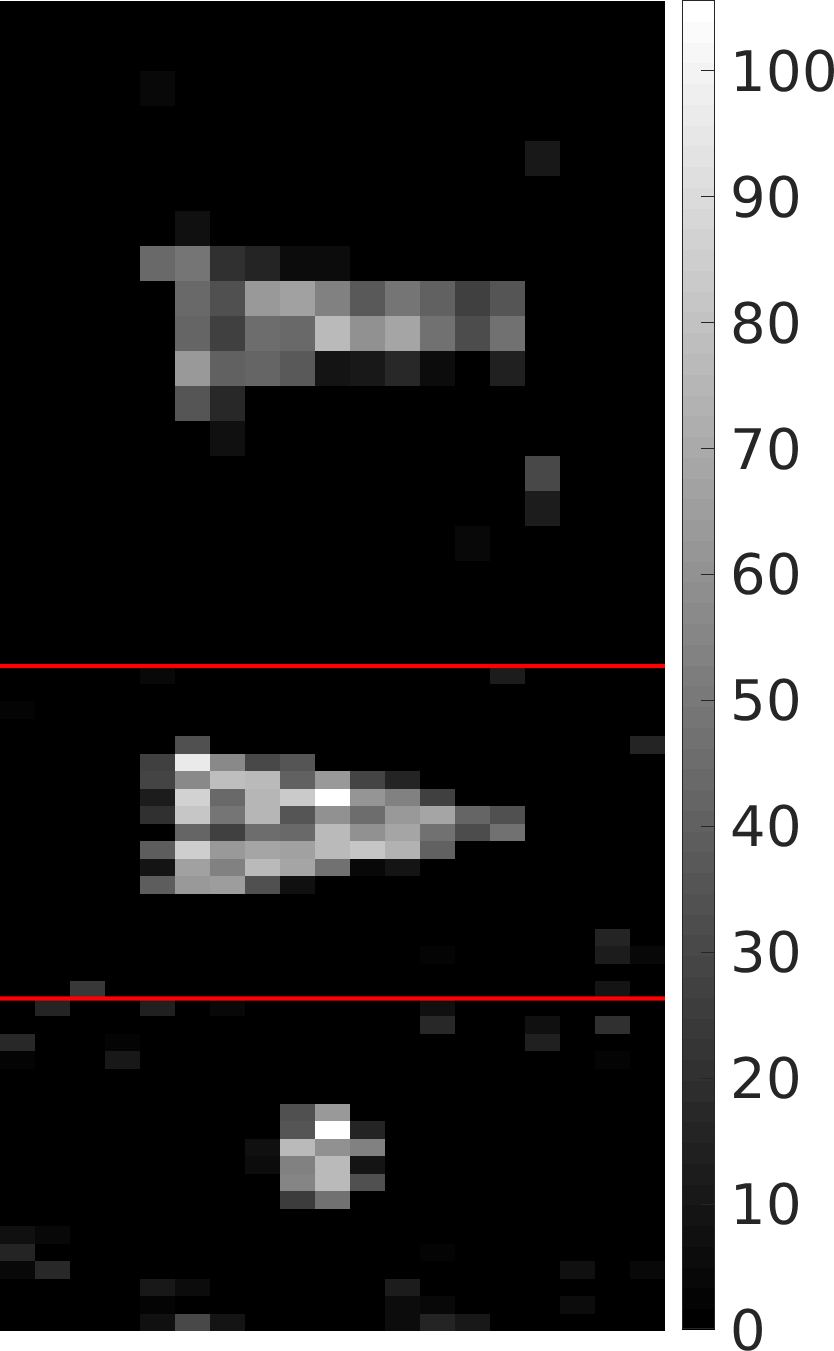}&
 \includegraphics[width=0.14\textwidth]{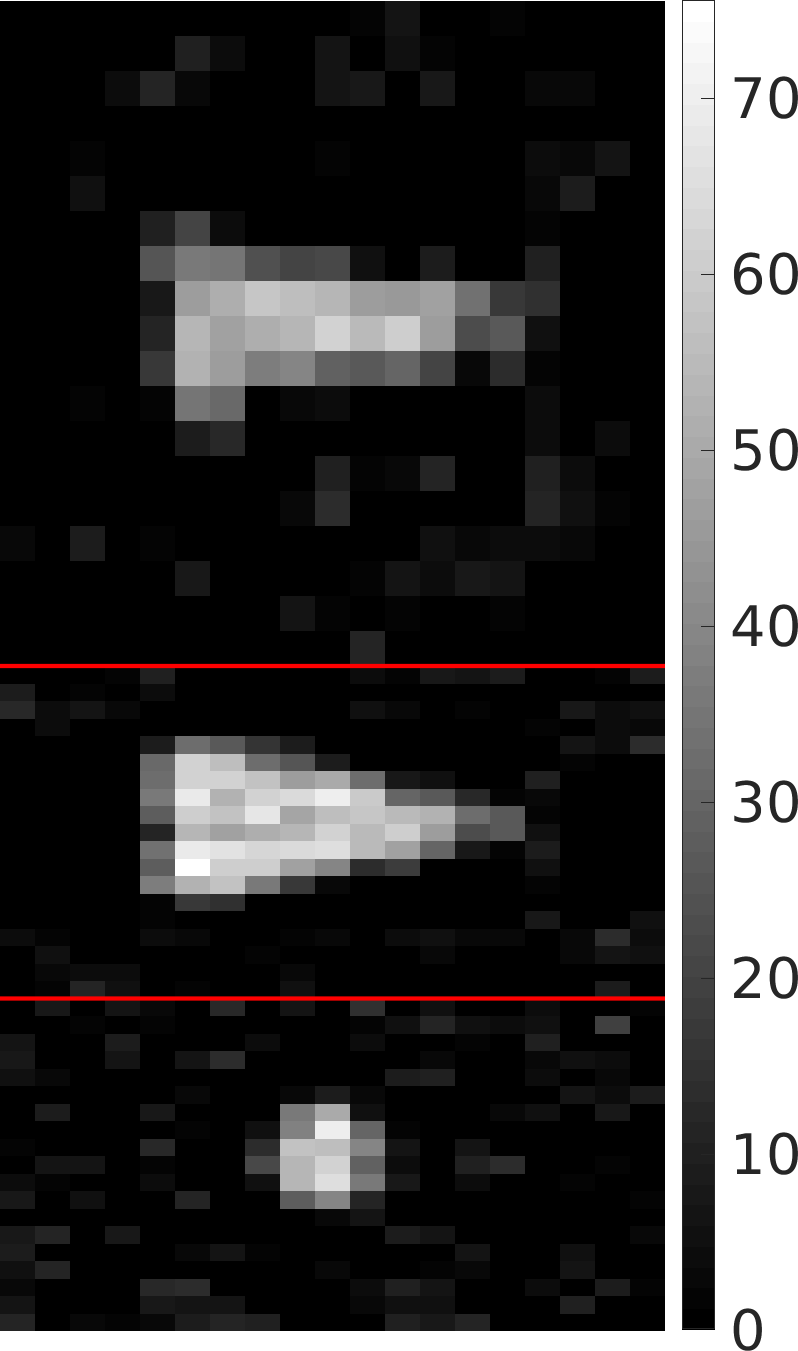}&
 \includegraphics[width=0.14\textwidth]{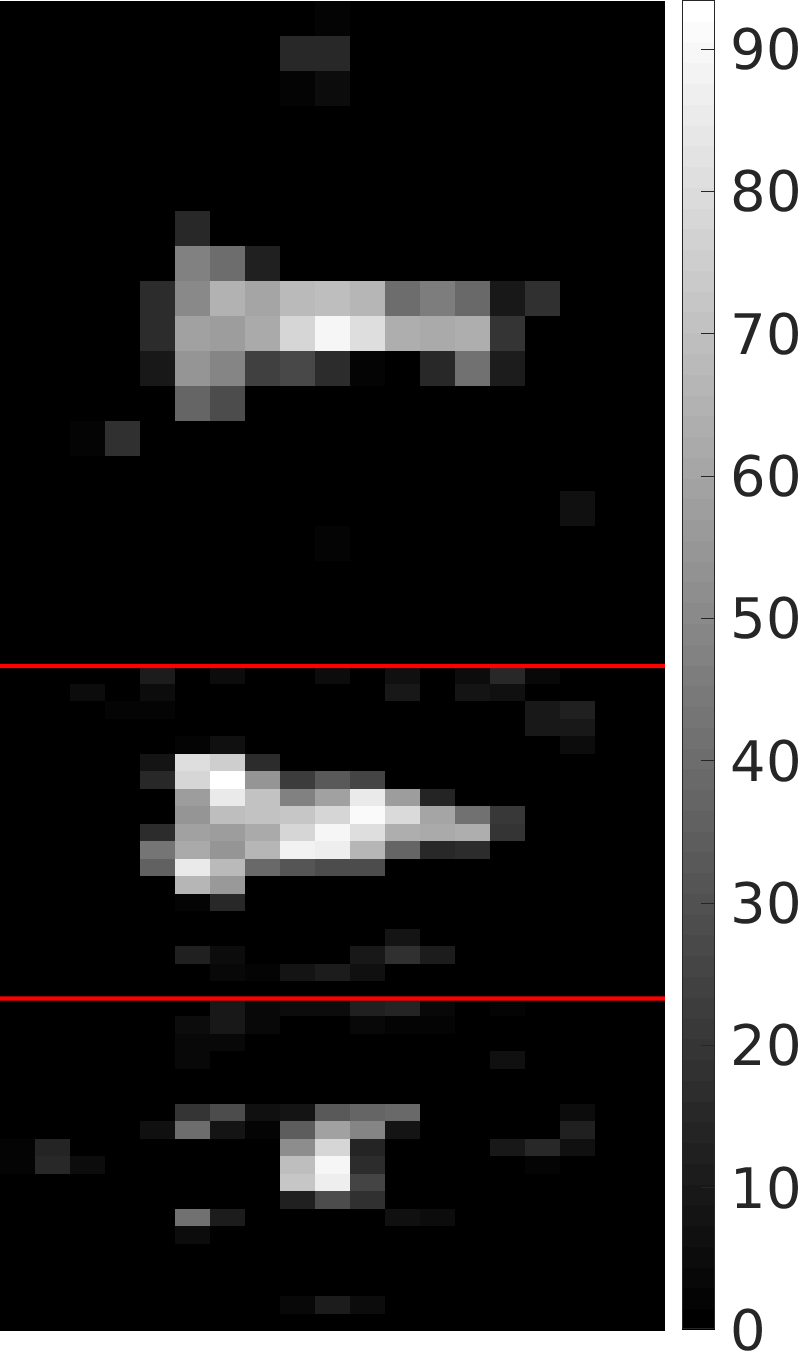}&
 \includegraphics[width=0.14\textwidth]{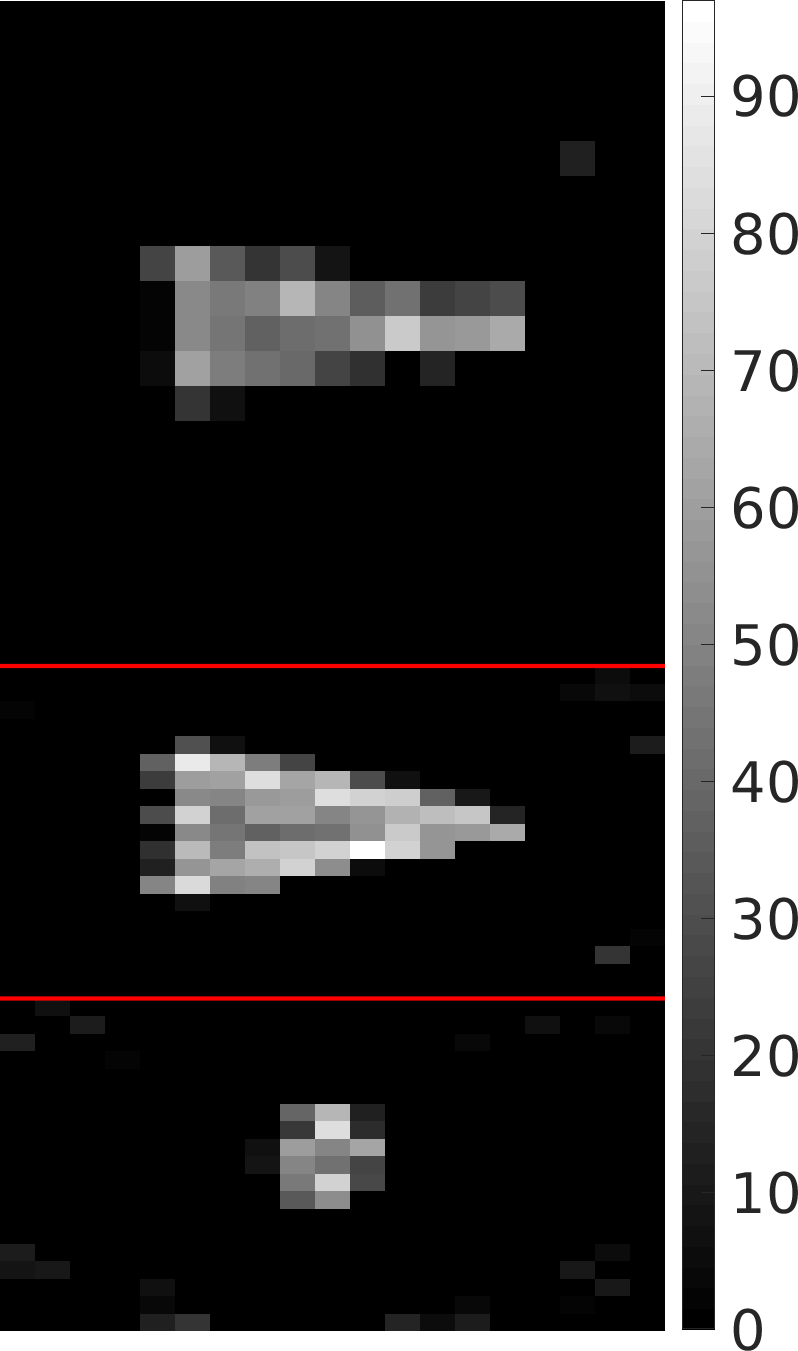}&
 \includegraphics[width=0.14\textwidth]{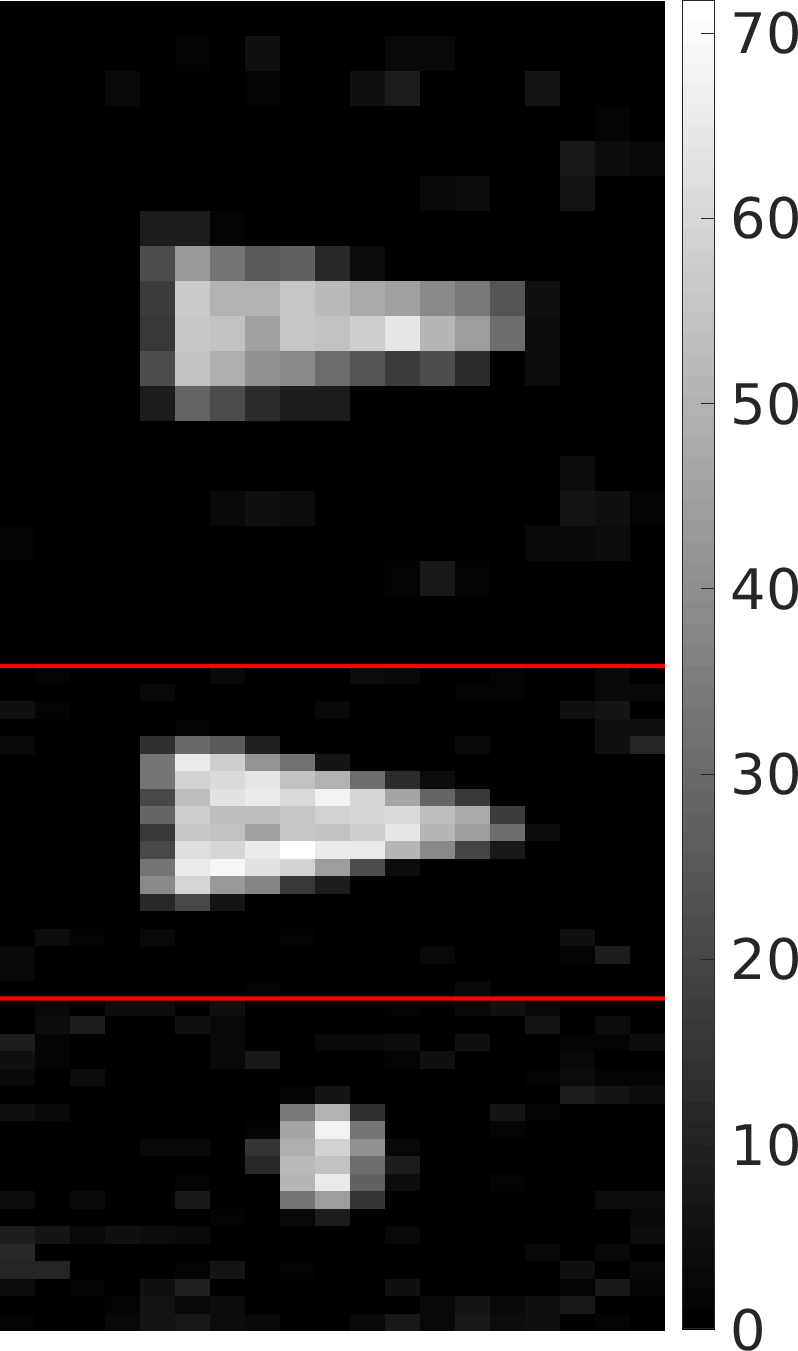}\\
\hline
\multicolumn{3}{l|}{$k=500$} & \multicolumn{3}{l}{}\\
 \includegraphics[width=0.14\textwidth]{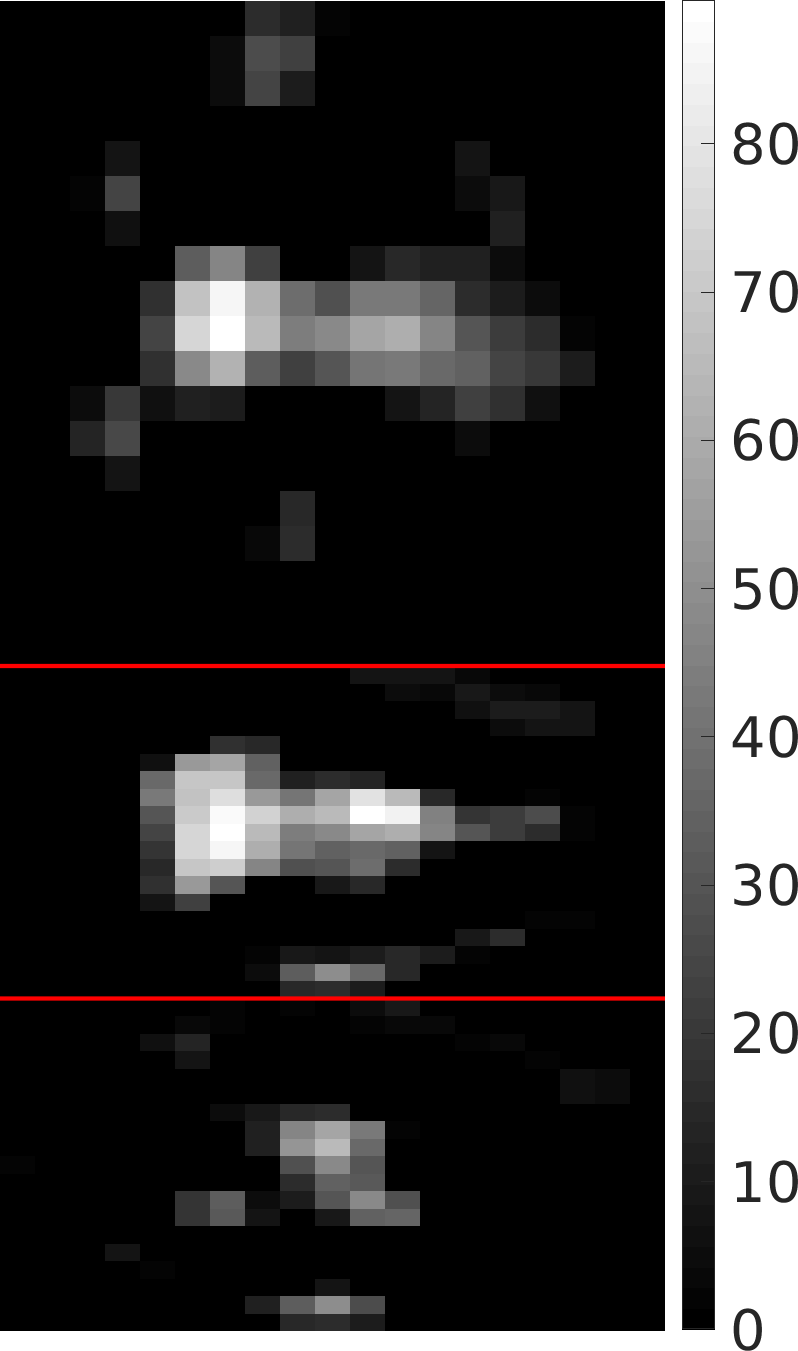}&
 \includegraphics[width=0.14\textwidth]{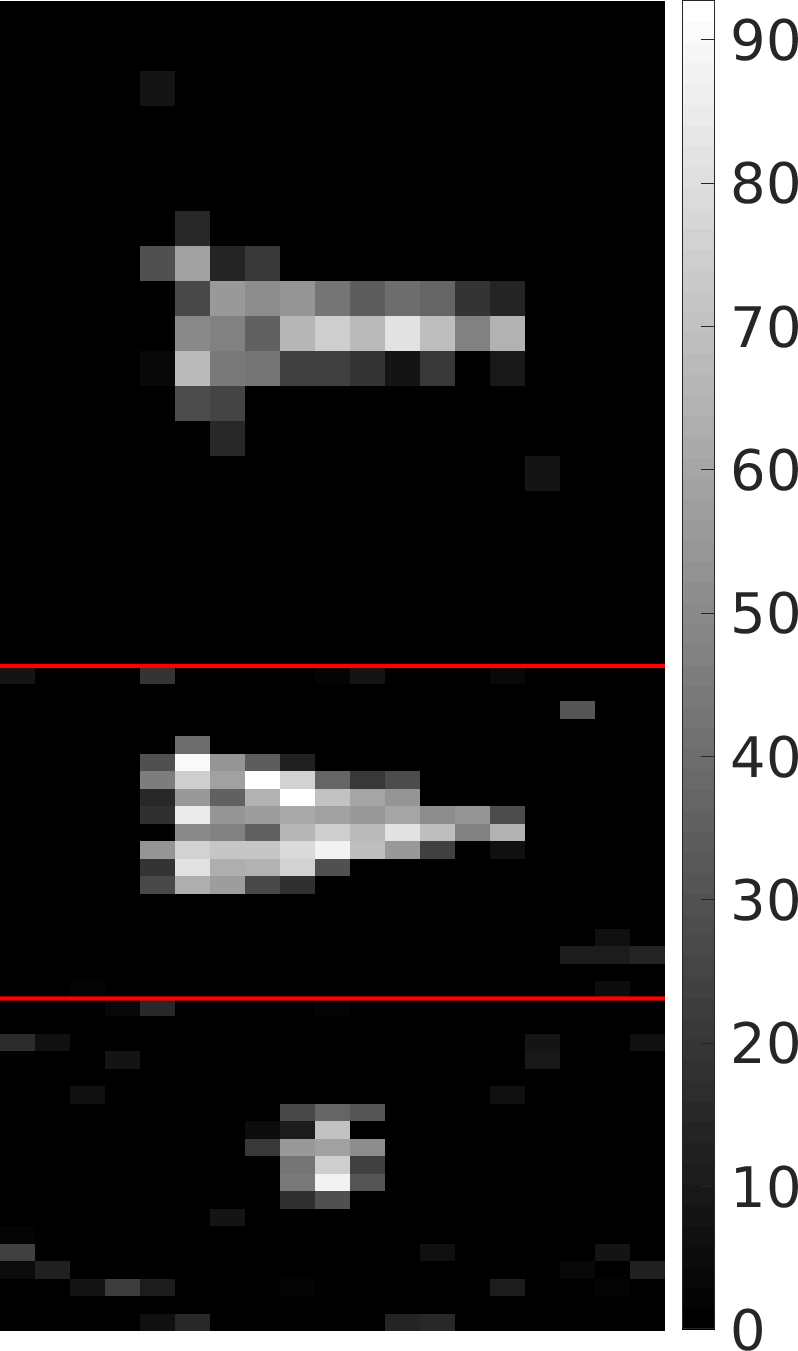}&
 \includegraphics[width=0.14\textwidth]{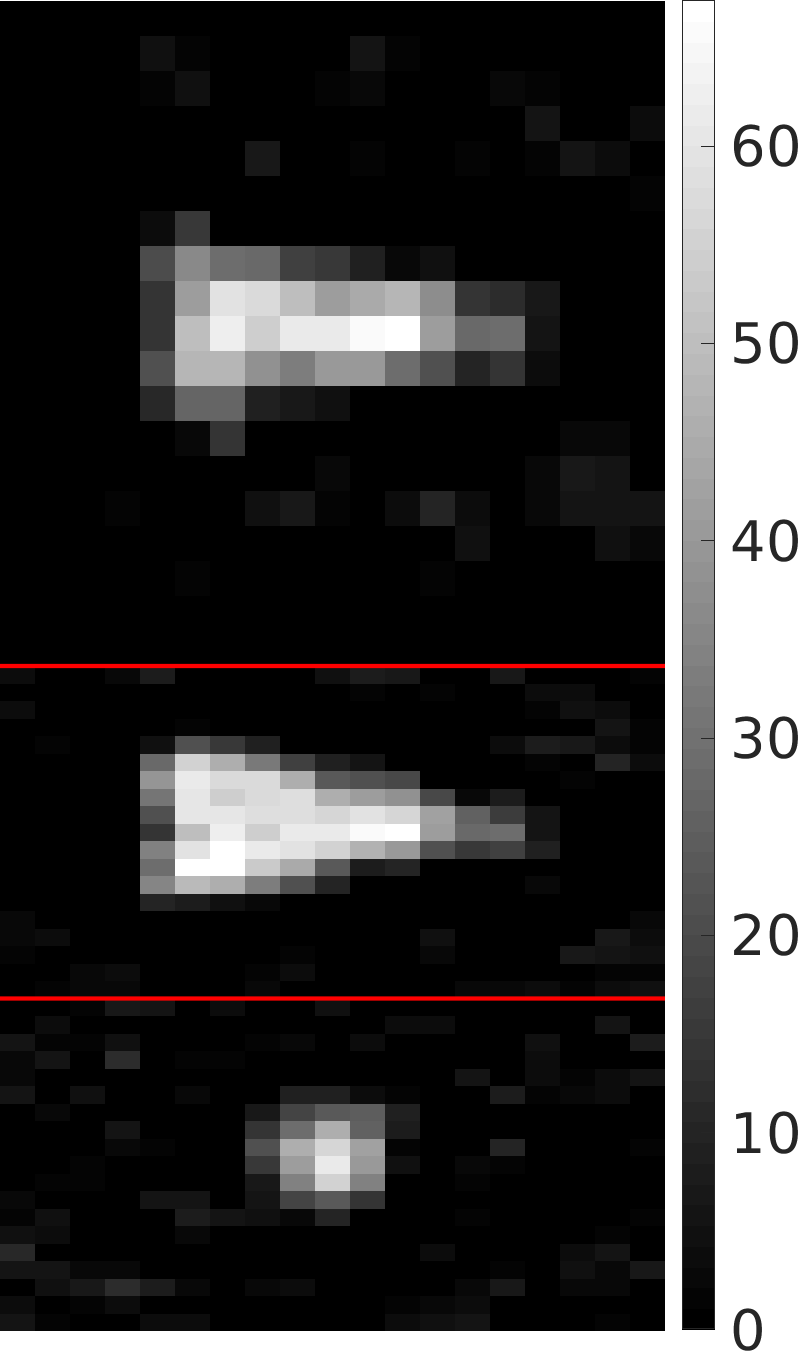}&
 \includegraphics[width=0.14\textwidth]{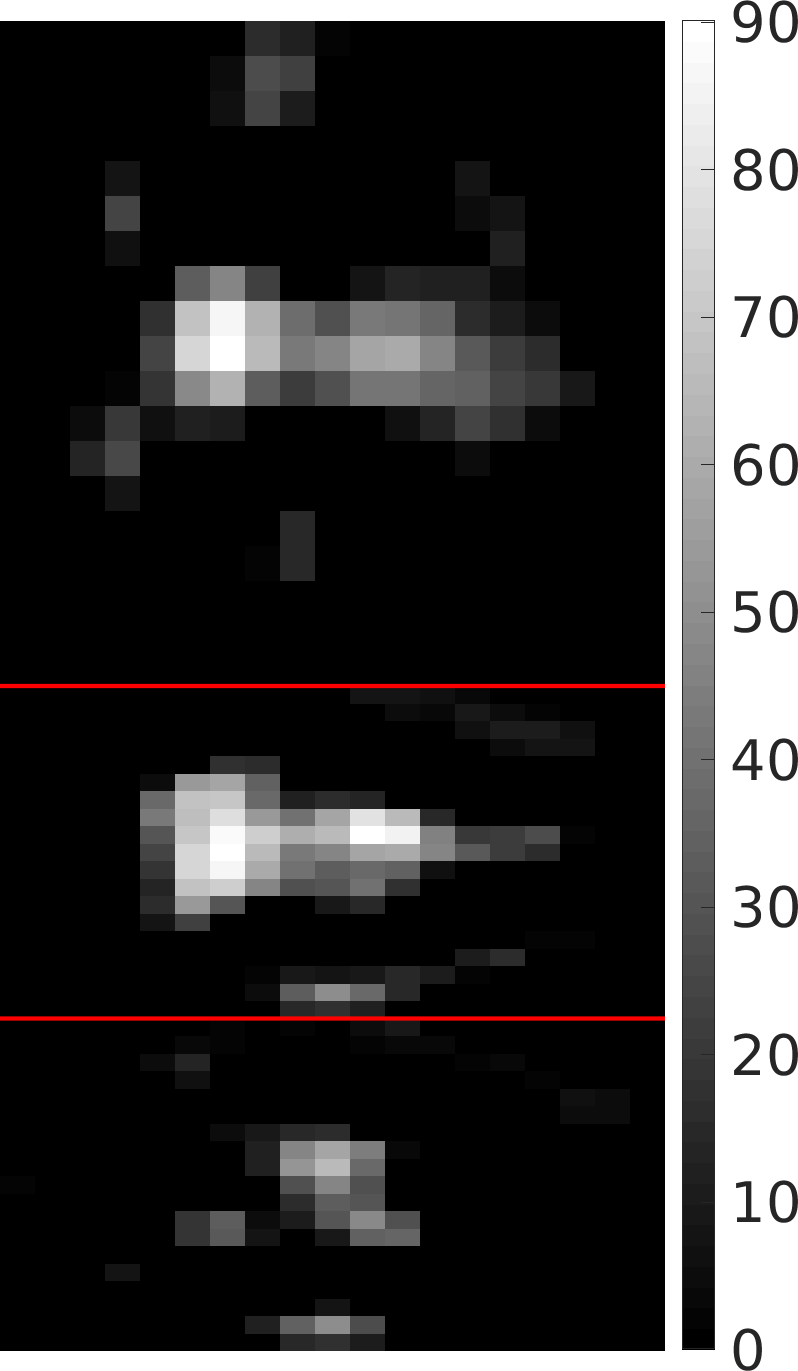}&
 \includegraphics[width=0.14\textwidth]{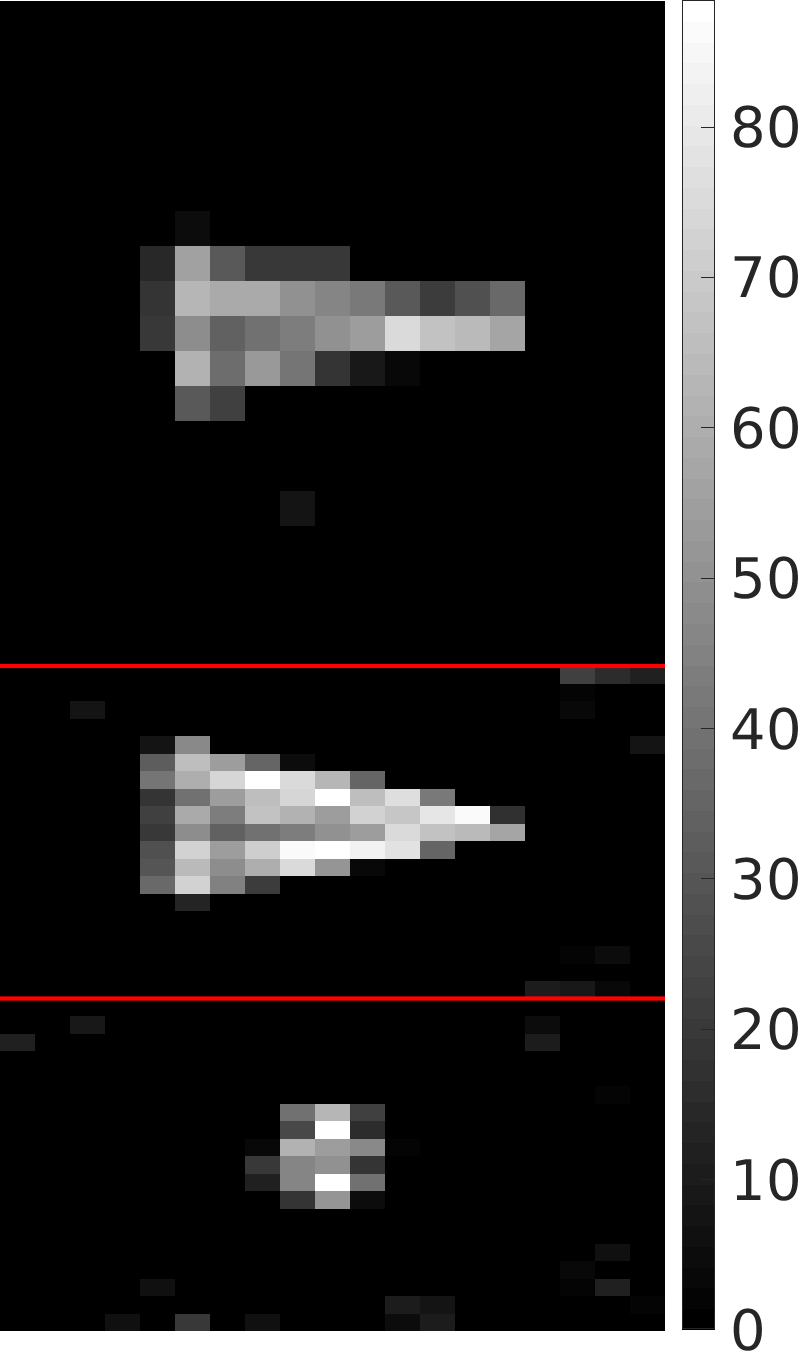}&
 \includegraphics[width=0.14\textwidth]{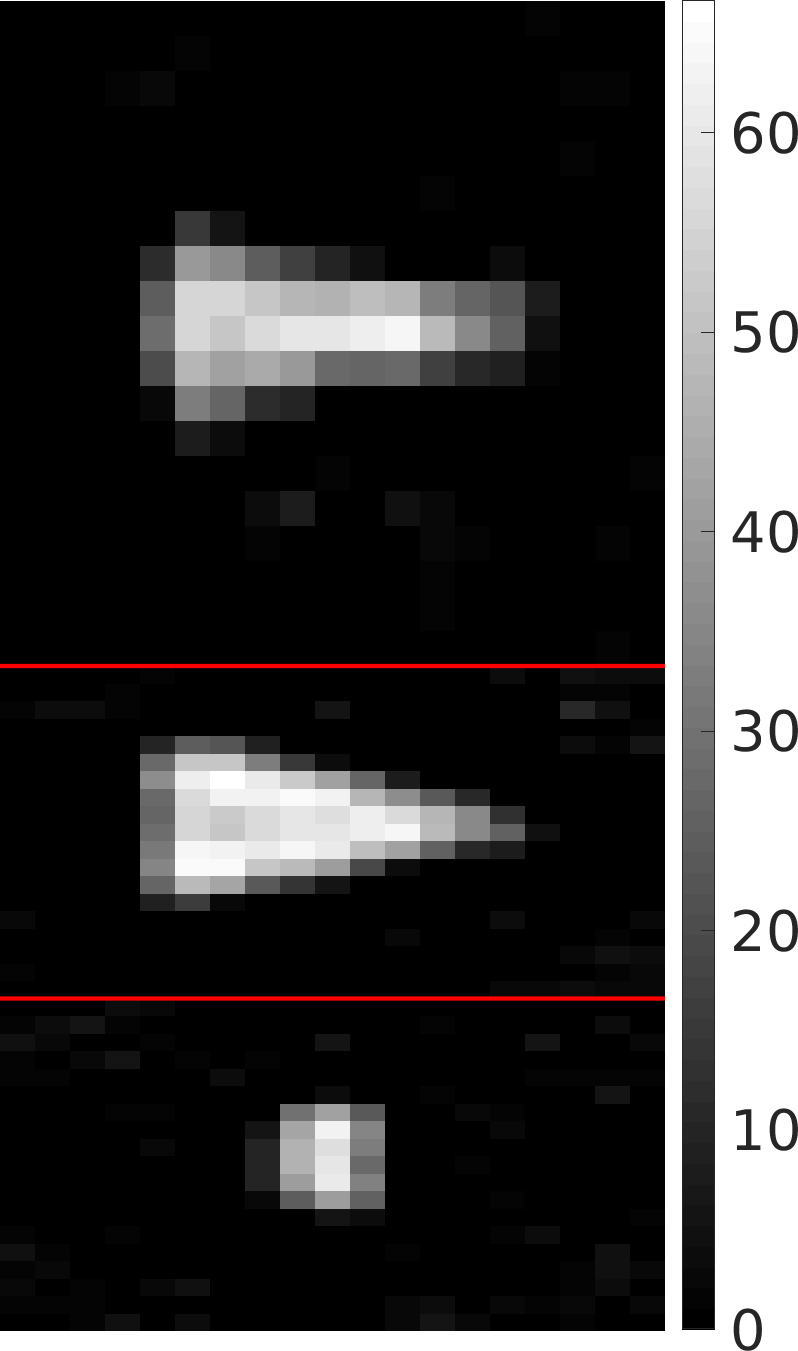}
\end{tabular}
\caption{``Shape'' phantom reconstructions (SNR, rSVD1, rSVD2) for $\alpha=9.54 \PLH 10^{-5}$.
Concentration in mmol/l.}
\label{fig:methods_alpha3}
\end{figure}

Next we focus on the role of whitening in rSVD acceleration. Whitening influences greatly both rSVD1 and 
rSVD2, especially when the $\alpha$ value is small: The whitened reconstructions are of better quality
since the background artifacts are strongly reduced; see Fig. \ref{fig:methods_alpha3}. Note that for small
$\alpha$, a small truncation number $k$ can be very beneficial for improving reconstruction quality, due to
its intrinsic regularizing effect (in a manner similar to the classical truncated SVD \cite{EnglHankeNeubauer:1996}).
Further, comparing $k=500$ for rSVD1 in Figs. \ref{fig:methods_alpha2} and \ref{fig:methods_alpha3}, e.g., in the
$x$-$y$-plane shows that whitening may enables further dimension reduction while maintaining reconstruction
quality (due to the change in the SV decay curve), concurring with the observation from Fig. \ref{tab:sv_energy}. 
Among the three methods under analysis, rSVD2 benefits most
from whitening, since for all three $\alpha$ values, the background artifacts
disappear almost completely. While the precise mechanism remains unclear, it may be attributed to the more
robust SVD without too noisy singular functions corresponding to large singular values. These observations
indicate that whitening is advantageous in the reconstruction: it enables using smaller $\alpha$ values in rSVD1
and rSVD2 to obtain acceptable reconstructions without unnecessary smoothing the actual phantom because
of using for a sufficiently large $\alpha$, cf. Fig. \ref{fig:methods_alpha1}. However, SNR benefits
little from whitening: it relies on the SNR-type quality measure for dimension reduction, already
exploiting the background noise characteristic to a certain degree. Thus, the dimensionality is already
dramatically reduced, and the remaining rows of the reduced system have a large SNR-type quality measure
and are only weakly influenced by the noise.

The computing times of the reconstruction methods are summarized in Table \ref{tab:cmp_times}, where we have 
ignored the cost of preprocessing (e.g., frequency selection or rSVD) since it can be carried out offline. STD is the 
most expensive one among all methods under consideration. The computing time for rSVD1 and SNR are more or less comparable, when using
same $k$ value, due to similar complexity (more precisely, rSVD1 has slightly longer computing times due to the
additional matrix-vector multiplication to project the measurement into the space spanned by the $k$ singular functions
in $\tilde{U}_k$ respectively $\tilde{U}_{W;k}$). rSVD2 is the fastest method due to its non-iterative nature, even if
one takes into account the 20 sweeps over the corresponding reduced systems for SNR and rSVD1. Thus, all the acceleration
approaches can significantly reduce the overall computational cost, with the speedup factor essentially determined by the size of
the reduced system. Note that the speedup is especially important, since in practice one has to choose a proper $\alpha$
value, which inevitably requires solving a fair number of optimization problems. In particular, it holds promise
as a nearly online algorithm.

\begin{table}[hbt!]
\centering
\begin{tabular}{l|cccc}
\hline
 $k$ & SNR & rSVD1 & rSVD2 \\
\hline
 $500$ &    0.3792$\pm$0.0227&    0.3880$\pm$0.0265&    0.0115$\pm$0.0011\\
$1000$ &    0.7572$\pm$0.0391&    0.7775$\pm$0.0403&    0.0210$\pm$0.0022\\
$1500$ &    1.1489$\pm$0.0618&    1.1743$\pm$0.0609&    0.0300$\pm$0.0023\\
$2000$ &    1.5465$\pm$0.0910&    1.5605$\pm$0.0712&    0.0395$\pm$0.0022\\
\hline
\end{tabular}
\caption{Computing times (in seconds) using \texttt{MATLAB} (on a server with 2$\times$Intel
\textsuperscript{\textregistered} Xeon\textsuperscript{\textregistered}  Broadwell-EP Series Processor E5-2687W
v4, 3.00 GHz, 12-Core, and 1.5 TB DDR4 PC2666 main memory), mean and standard deviation over 100 reconstructions.
The computing time for STD is 53.5308$\pm$2.6653. STD for $20n$ iterations, SNR and rSVD1 for $20k$ iterations.
$n=70446$ and $m=6859$ voxels.}\label{tab:cmp_times}
\end{table}

In summary, randomized SVD can significantly accelerate the reconstruction algorithms, within which the whitening
procedure is highly beneficial, while maintaining the overall accuracy.

\subsection{Performance of choice rules}

Last, we illustrate the performance of the two choice rules described in Section \ref{ssec:rule}, i.e., discrepancy
principle (DP) and quasi-optimality  (QO) criterion. The numerical results are given in Figs.
\ref{fig:STD_parameter_choice}-\ref{fig:discrepancy_principle} (with the initial value $\alpha_0=100$ and
decreasing factor $q=0.5$). Due to the challenging nature of choosing appropriate $\tau$, $\delta$, $\sigma$ and
$\epsilon$ (which may also differ for the compared methods), we used the minimum of the first 50 parameters given
by the geometric sequence $\alpha_i$ as the upper bound $\tau\delta + \sigma\epsilon$ for DP in \eqref{eqn:discrepancy}.

\begin{figure}[hbt!]
\centering
 \begin{tabular}{ccc|ccc}
\multicolumn{3}{c|}{Non-whitened} & \multicolumn{3}{c}{Whitened}\\
\hline
 & (discr.) & (quasi.) & & (discr.) & (quasi.)  \\
 & $\alpha_{i^\ast}=1.56 \PLH 10^{0}$  & $\alpha_{i^\ast}=7.81 \PLH 10^{-1}$& & $\alpha_{i^\ast}=1.56 \PLH 10^{0}$ & $\alpha_{i^\ast}=2.44 \PLH 10^{-2}$  \\
\hline
& \includegraphics[width=0.14\textwidth]{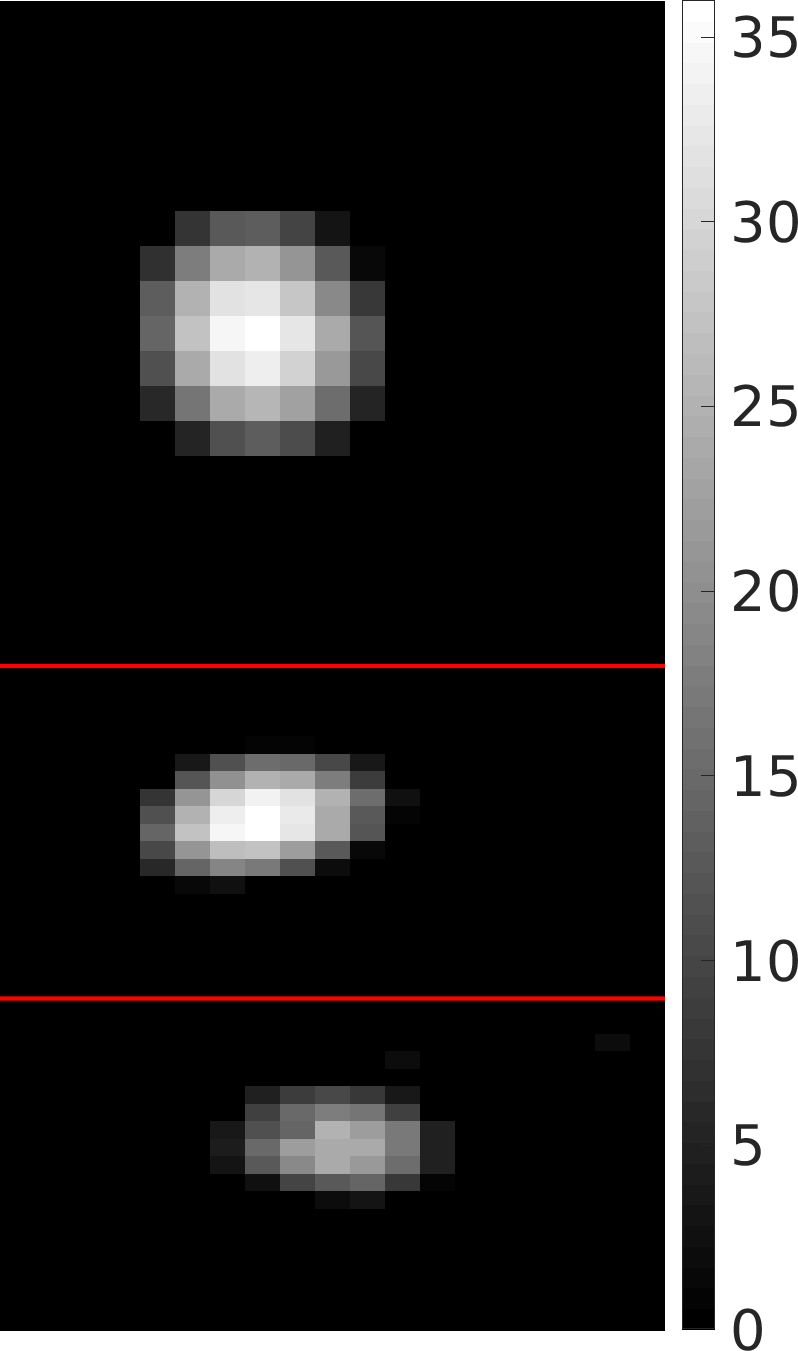}&
\includegraphics[width=0.14\textwidth]{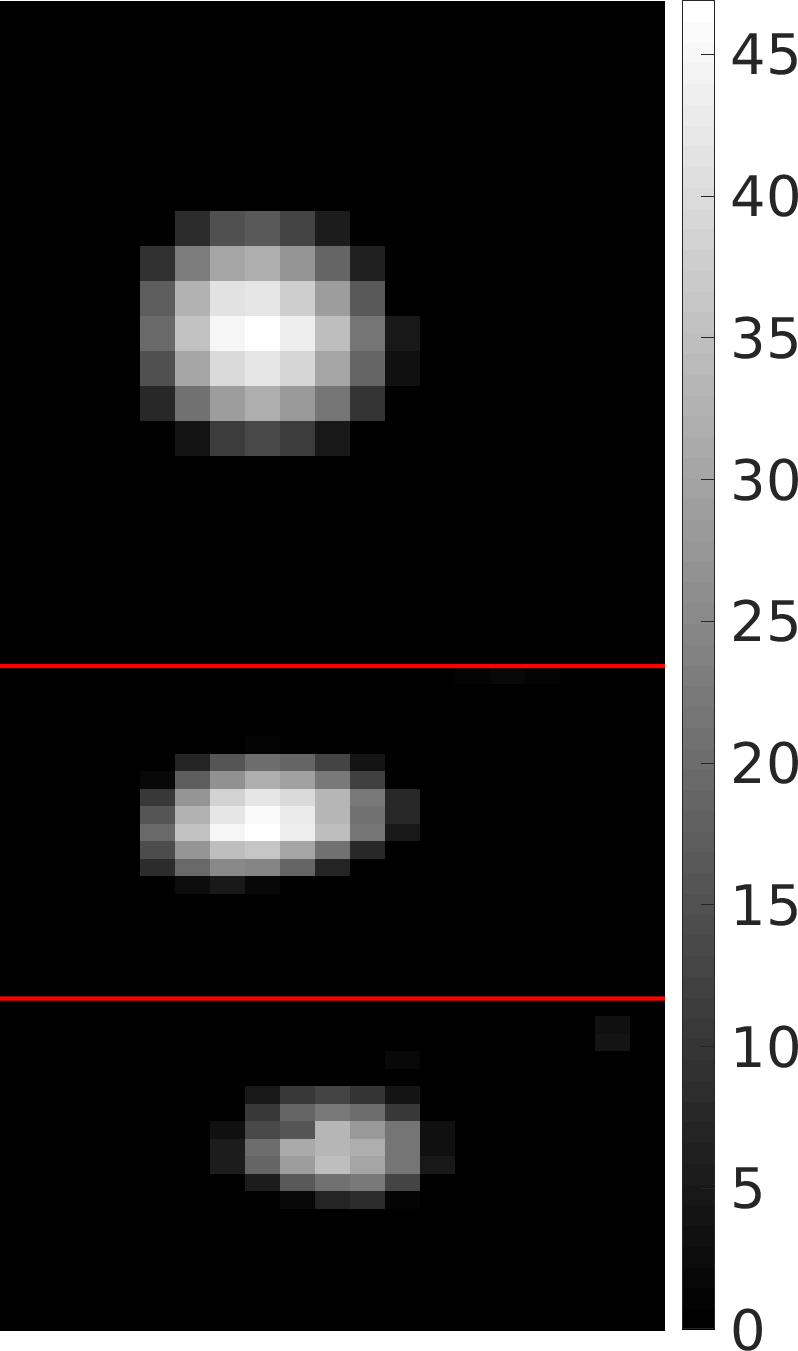}&
& \includegraphics[width=0.14\textwidth]{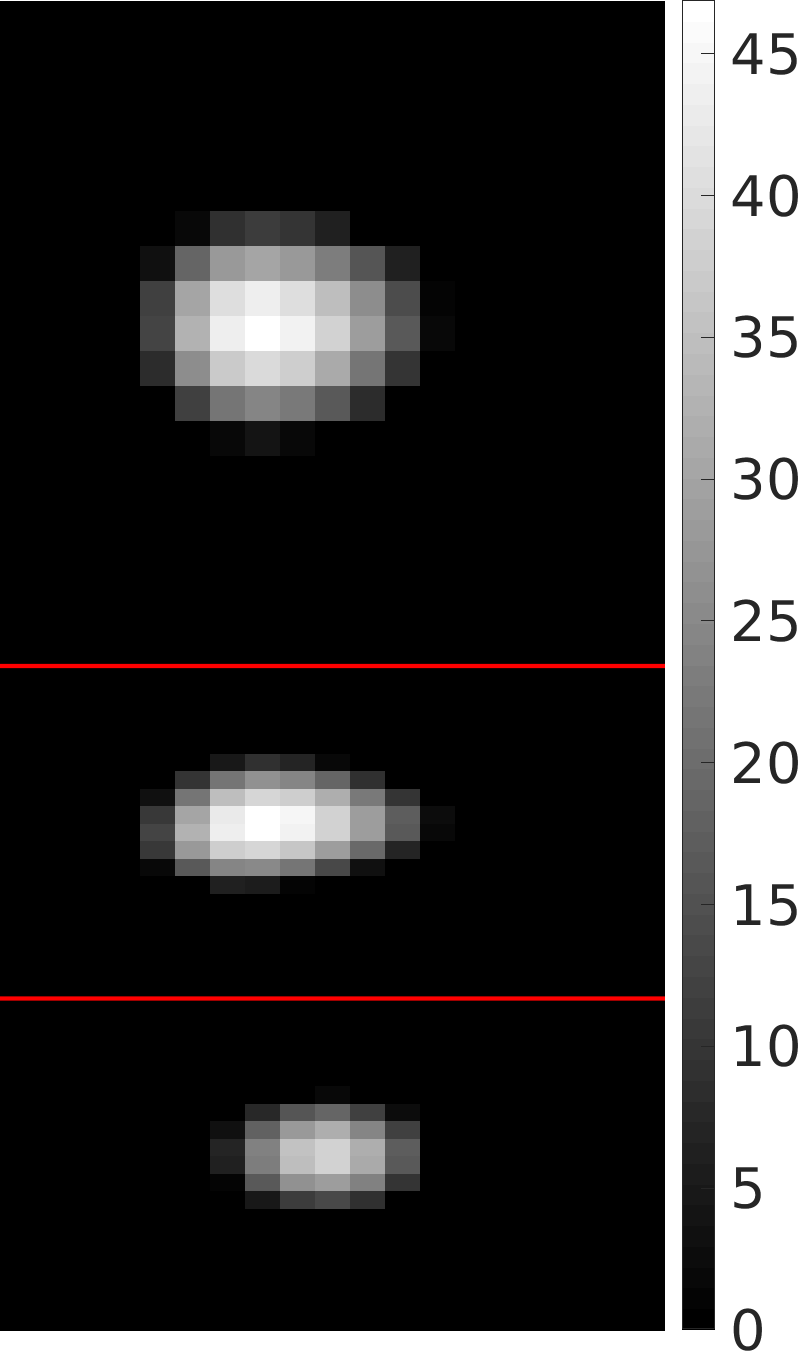}
&  \includegraphics[width=0.14\textwidth]{figures/plots_open_MPI_no_snr_thr/shapePhantom/use_cpp0/dim2000/method1_alpha11.png}
\\
\end{tabular}
\caption{$\alpha_{i^\ast}$ chosen as described in Figs. \ref{fig:quasi_optimality} and \ref{fig:discrepancy_principle} for the
STD reconstruction. Concentration in mmol/l.}
\label{fig:STD_parameter_choice}
\end{figure}

In the non-whitened case, for the given choice of parameters, DP can only give reasonable results for SNR (cf. Fig.
\ref{fig:discrepancy_principle}). This might be related to the SNR-type quality measure: it gives a reduced
system with less noise in each row such that DP can terminate at smaller $\alpha$ values. In all other
non-whitened methods, there might be rows with noise contributions having a larger magnitude, which causes
an early stopping of the choice rule and gives strongly regularized / smoothed reconstructions (also decreased
concentration values). Whitening is also not able to improve the performance of DP. These observations are
in line with the well known fact that DP tends to yield overly smoothing reconstructions, by choosing a too large
regularization parameter \cite{EnglHankeNeubauer:1996}. These empirical observations indicate that the
delicacy of applying DP to MPI imaging, and further research is needed to make it feasible, with one
crucial issue being to obtain a reliable estimate on the noise level and bounds on the modeling errors.

The QO criterion shows excellent performance in the non-whitened case except for  SNR ($k=500$) and
STD. In all methods the performance is further improved when combined with whitening.
Particularly for the proposed methods (rSVD1, rSVD2), it has a far superior performance than DP, and
for these two methods, the best performance is reached for the whitened case. When compared with DP,
STD and SNR  show also improved or comparable performance, except in the SNR non-whitened
case for $k=500$, which, as mentioned earlier, is insufficient to capture the essential information
content in the dataset (and thus does not allow accurate reconstruction due to intrinsic information loss).
For the QO criterion, whitening greatly improves the performance in all cases.

\begin{figure}
\begin{tabular}{ccc|ccc}
\multicolumn{3}{c|}{Non-whitened} & \multicolumn{3}{c}{Whitened}\\
\hline
SNR & rSVD1 & rSVD2 & SNR & rSVD1 & rSVD2 \\
\hline
\multicolumn{3}{l|}{$k=1000$} & \multicolumn{3}{l}{}\\
$\alpha_{i^\ast}=2.44 \PLH 10^{-2}$ & $\alpha_{i^\ast}=2.44 \PLH 10^{-2}$ &$\alpha_{i^\ast}=1.56 \PLH 10^{0}$ & $\alpha_{i^\ast}=4.88 \PLH 10^{-2}$& $\alpha_{i^\ast}=1.22 \PLH 10^{-2}$ & $\alpha_{i^\ast}=1.56 \PLH 10^{0}$ \\
 \includegraphics[width=0.13\textwidth]{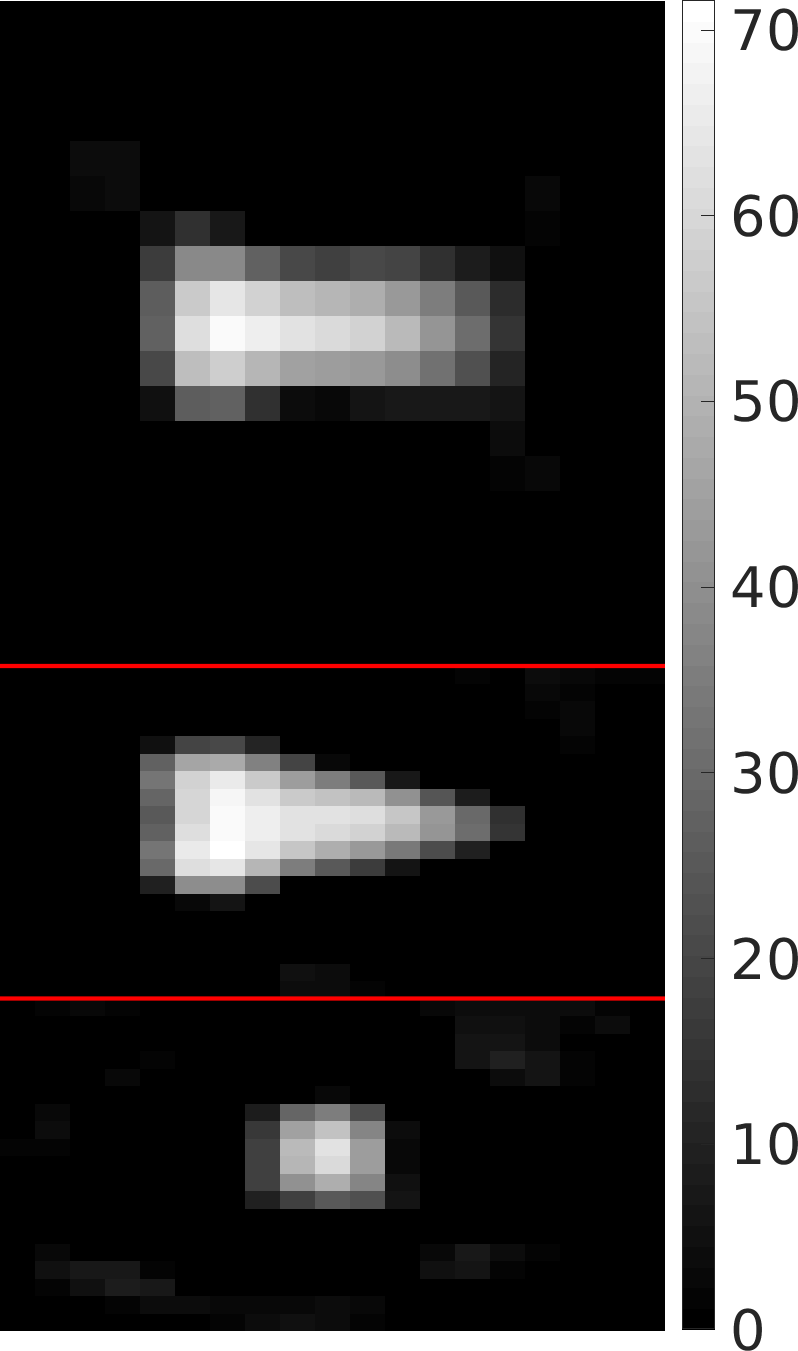}&
 \includegraphics[width=0.13\textwidth]{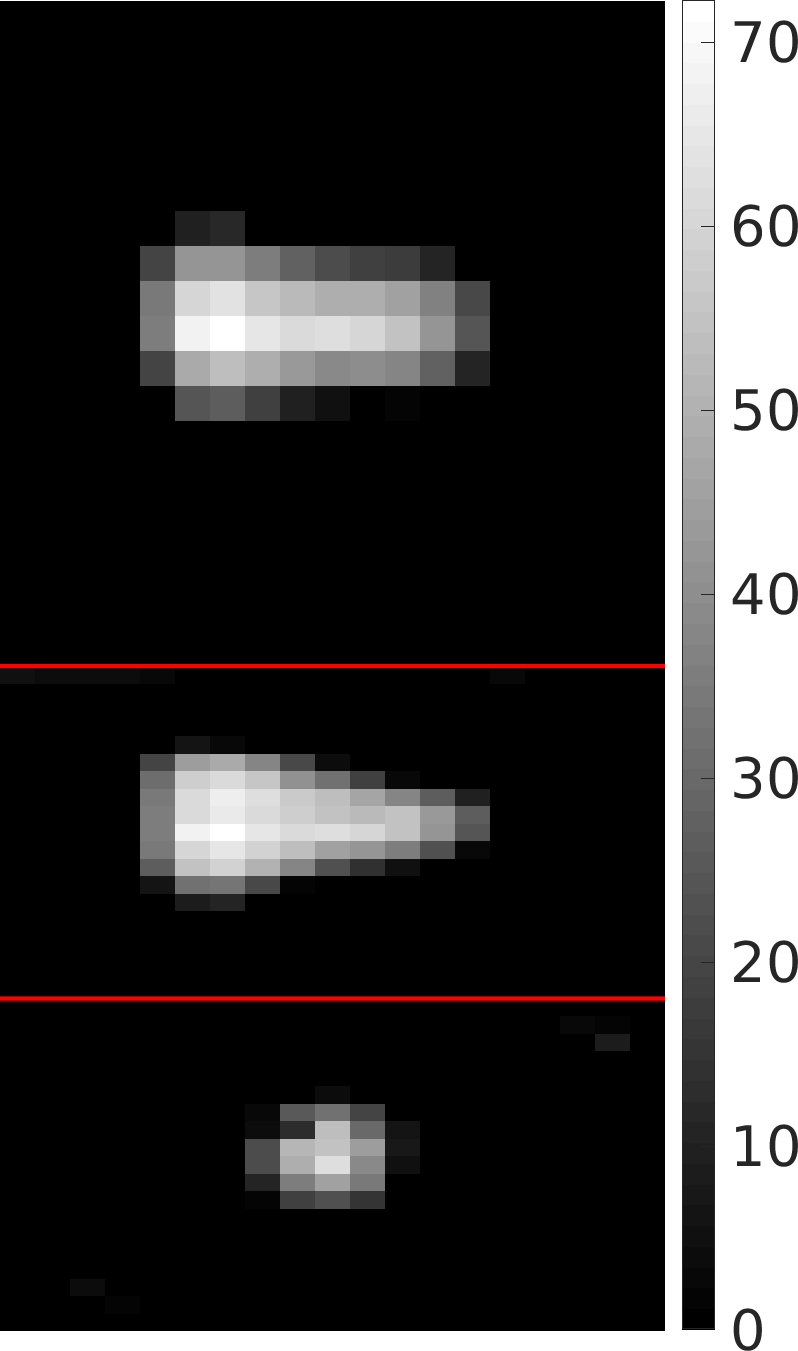}&
 \includegraphics[width=0.13\textwidth]{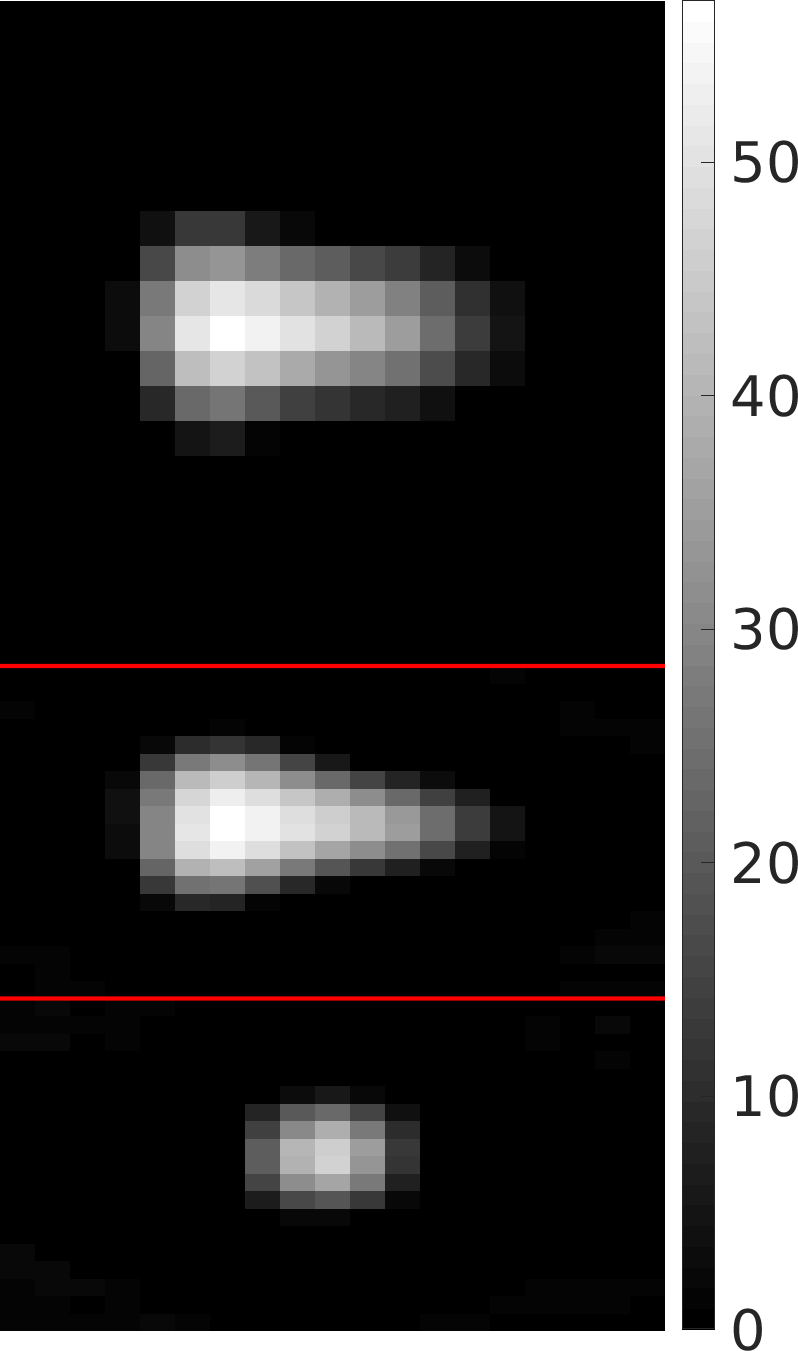}&
 \includegraphics[width=0.13\textwidth]{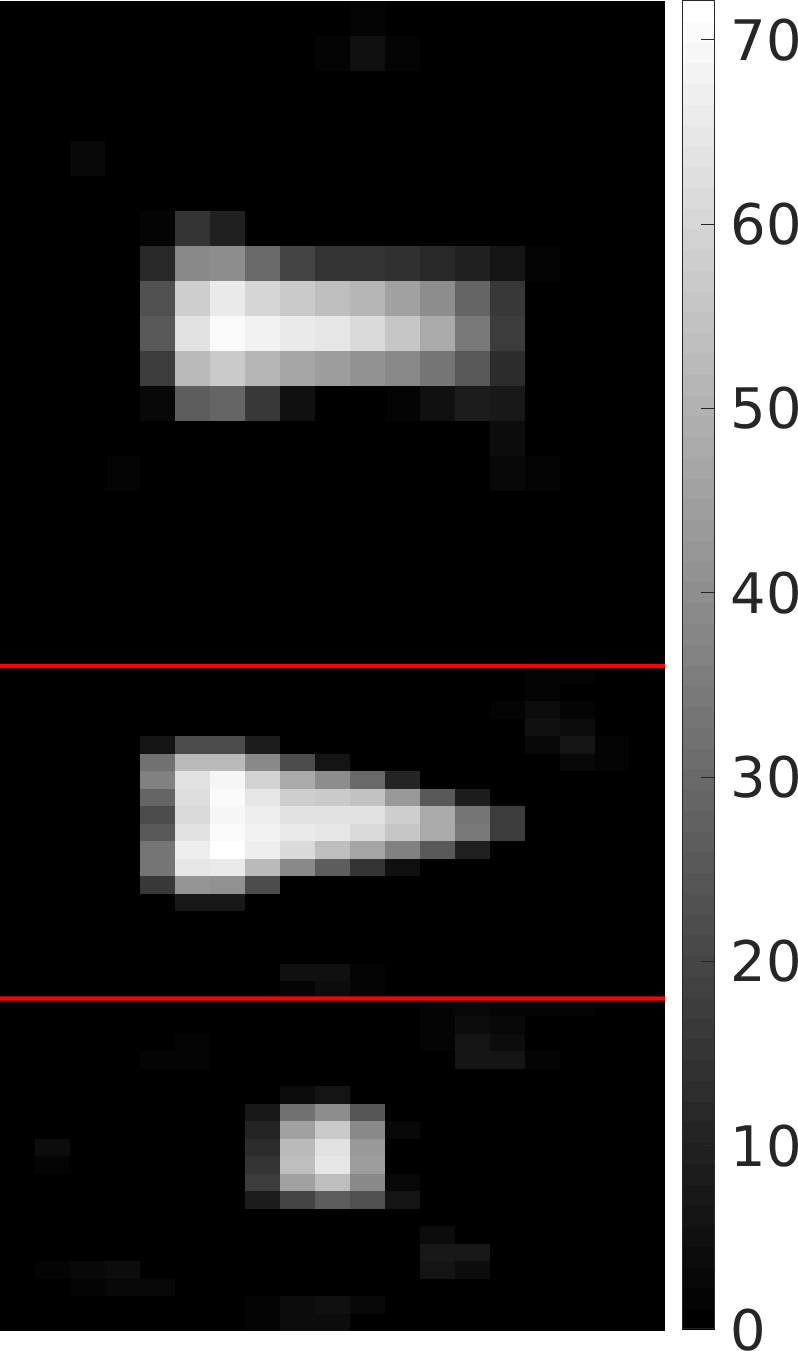}&
 \includegraphics[width=0.13\textwidth]{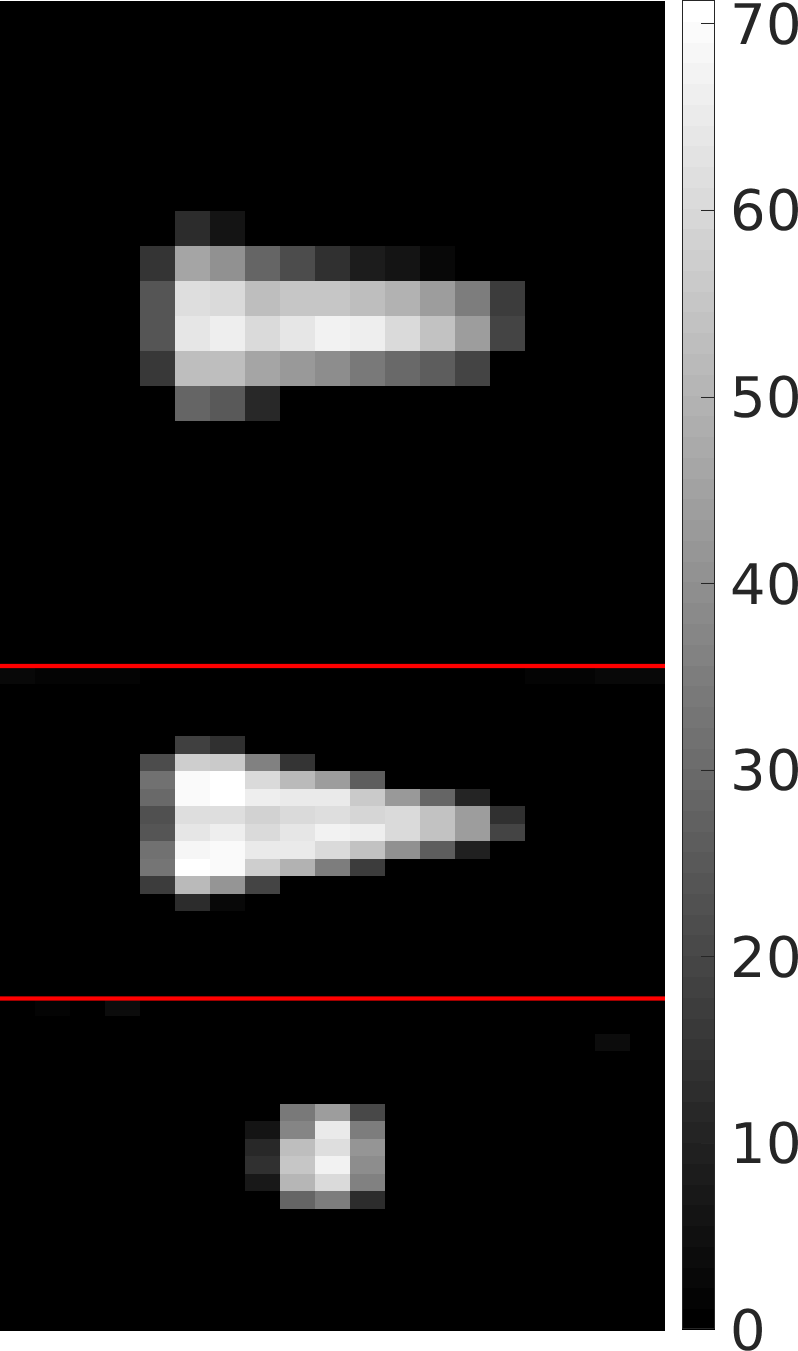}&
 \includegraphics[width=0.13\textwidth]{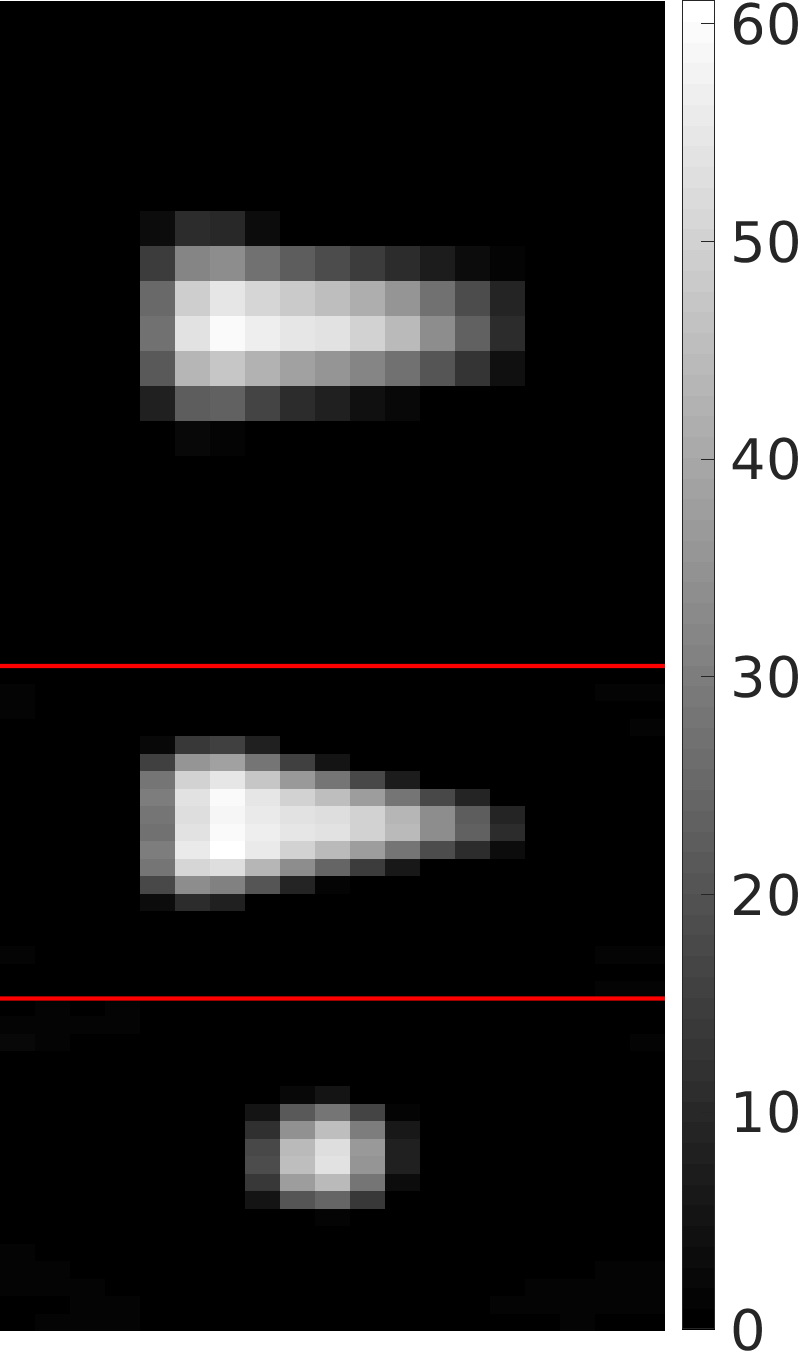}\\
\hline
\multicolumn{3}{l|}{$k=500$} & \multicolumn{3}{l}{}\\
$\alpha_{i^\ast}=12.5 \PLH 10^{0}$ & $\alpha_{i^\ast}=2.44 \PLH 10^{-2}$ &$\alpha_{i^\ast}=1.56 \PLH 10^{0}$ & $\alpha_{i^\ast}=9.77 \PLH 10^{-2}$& $\alpha_{i^\ast}=1.22 \PLH 10^{-2}$ & $\alpha_{i^\ast}=7.27 \PLH 10^{-10}$ \\
 \includegraphics[width=0.13\textwidth]{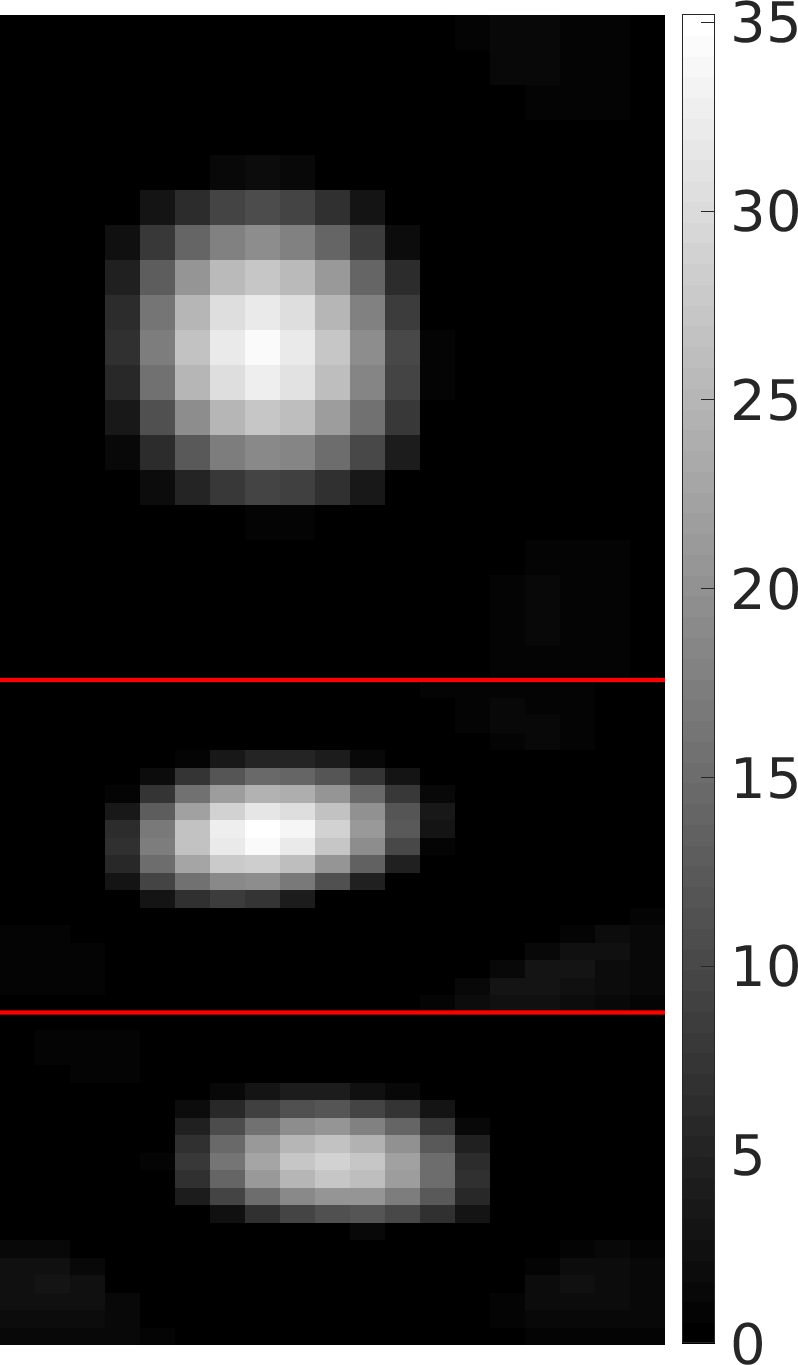}&
 \includegraphics[width=0.13\textwidth]{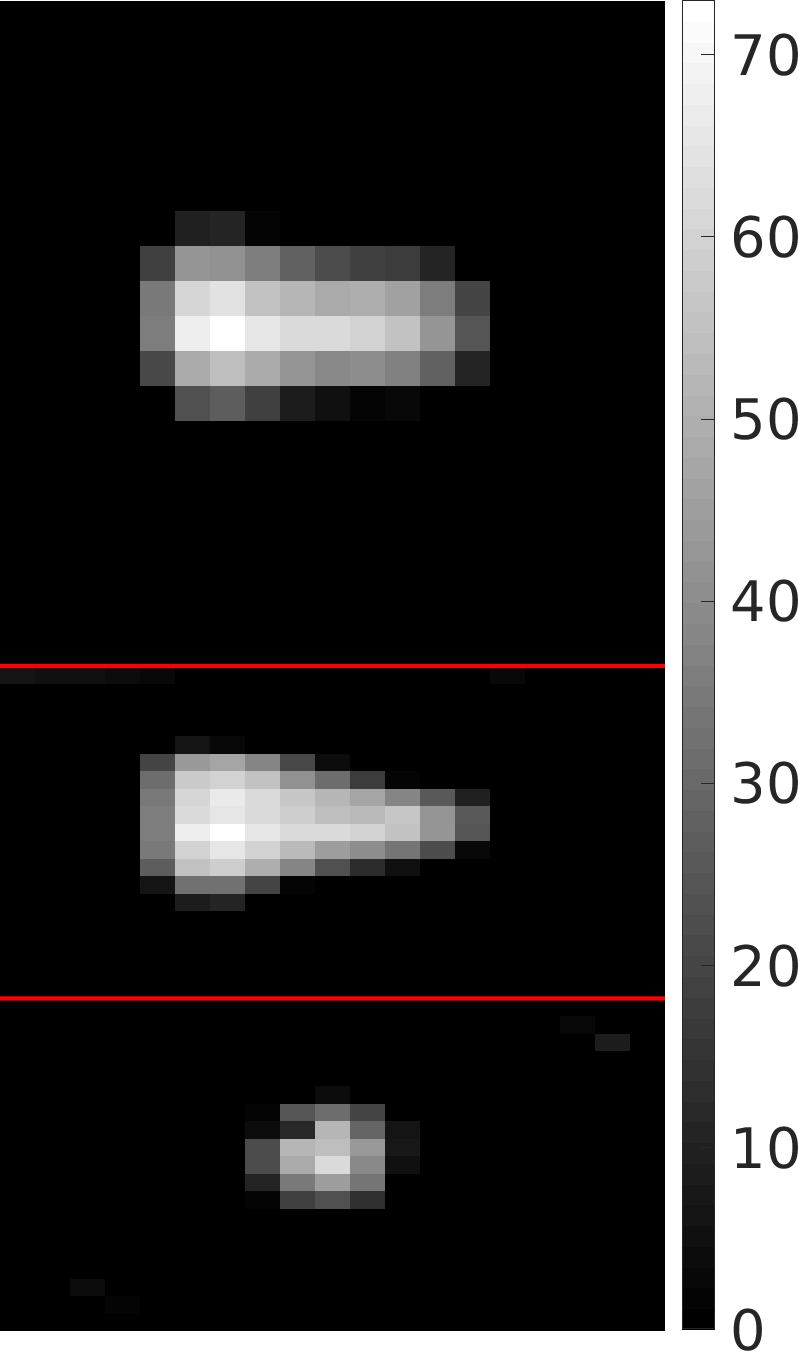}&
 \includegraphics[width=0.13\textwidth]{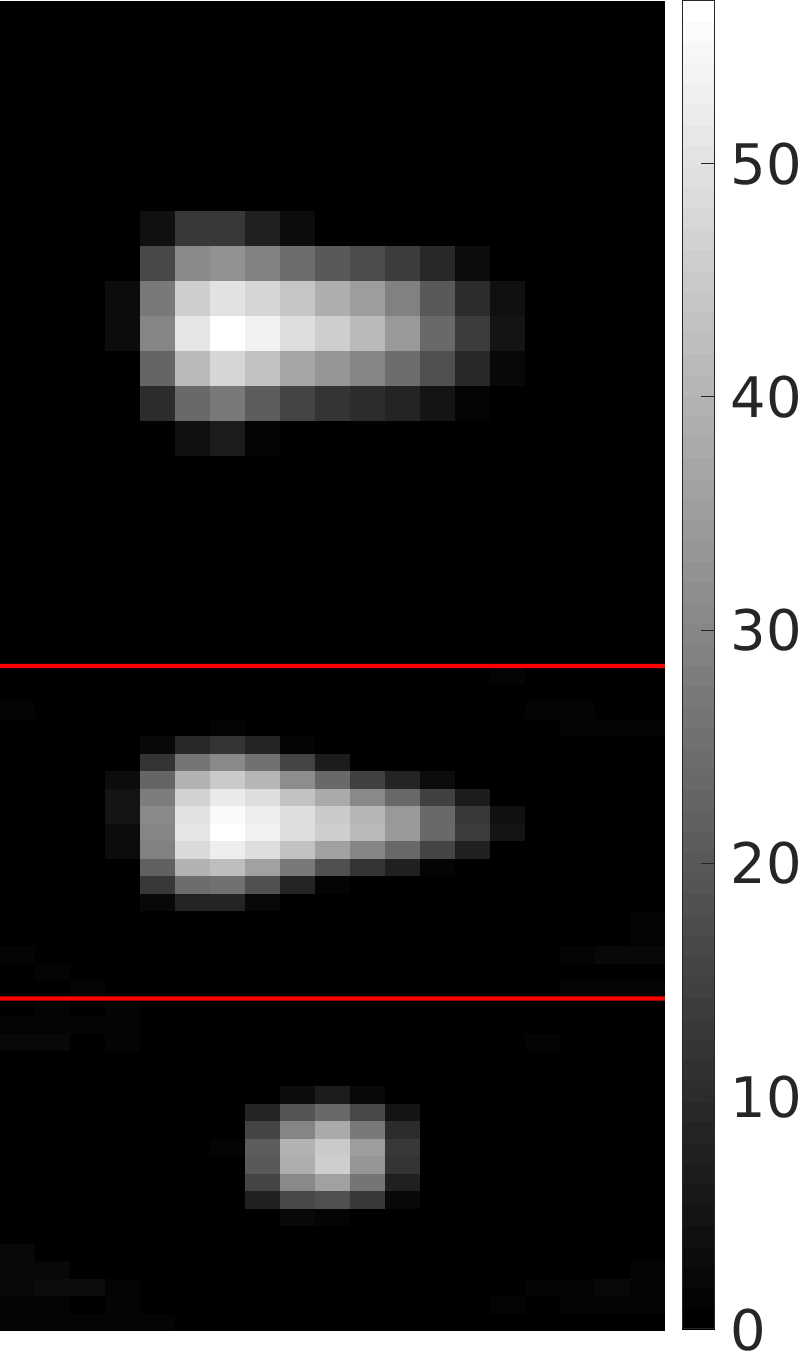}&
  \includegraphics[width=0.13\textwidth]{figures/plots_open_MPI_no_snr_thr/shapePhantom/use_cpp0/dim500/method6_alpha11.png}&
 \includegraphics[width=0.13\textwidth]{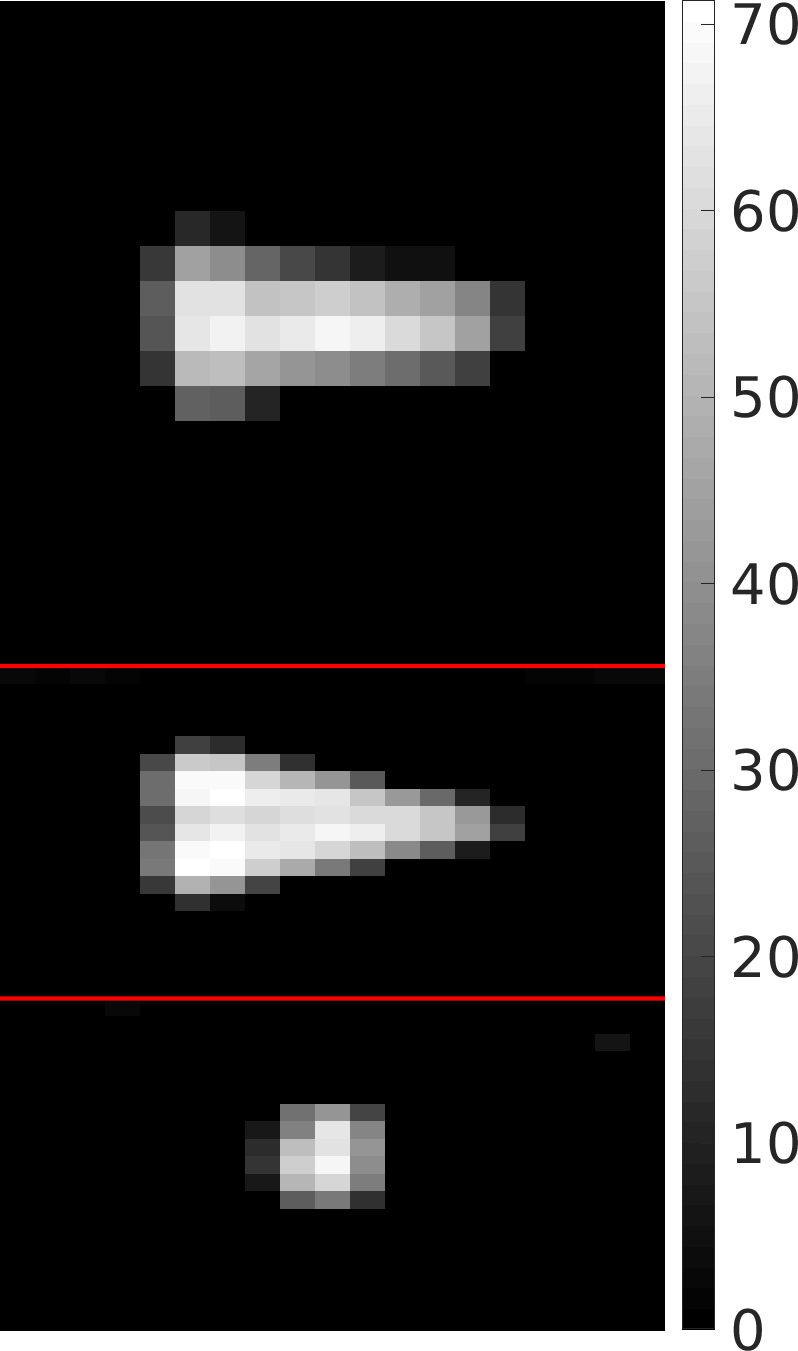}&
\includegraphics[width=0.13\textwidth]{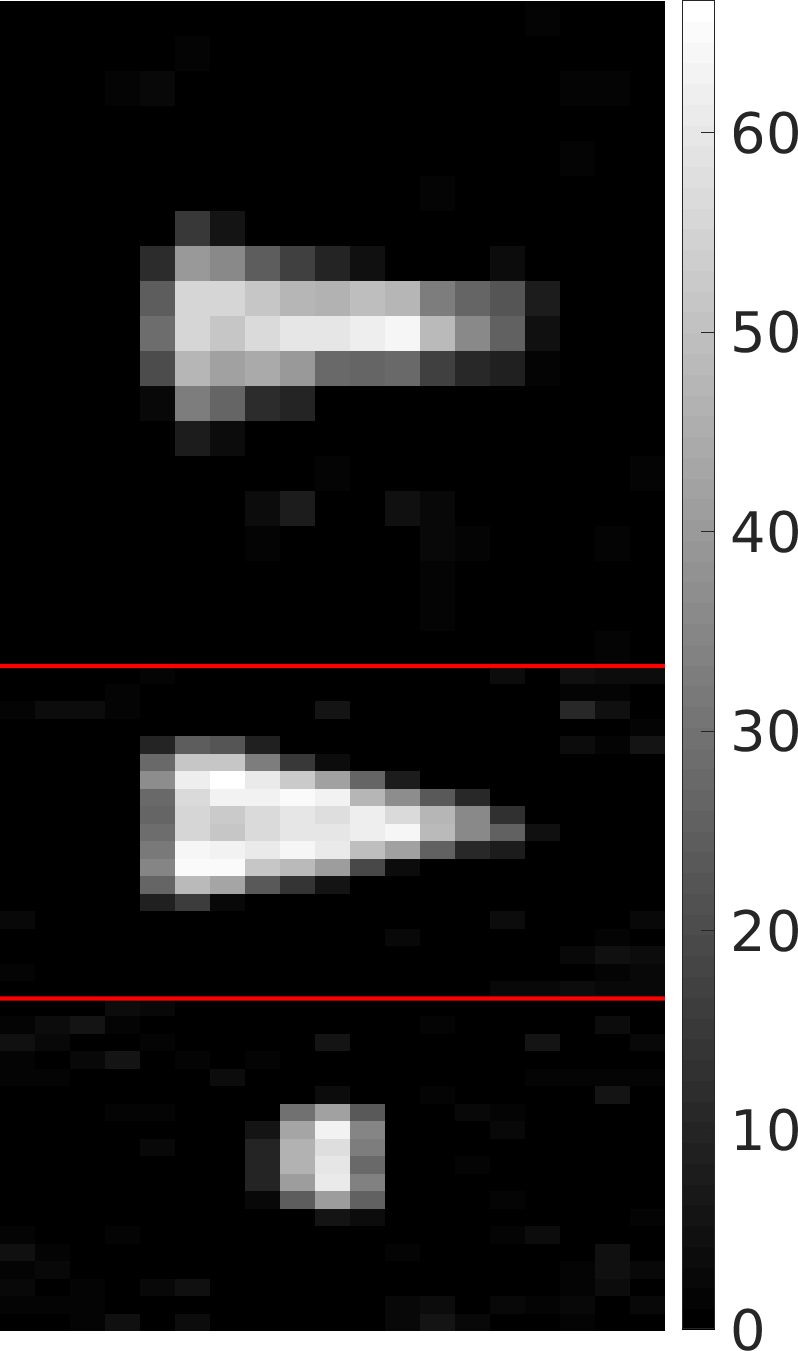}
\end{tabular}
\caption{$\alpha_{i^\ast}$ chosen according to the quasi-optimality principle in \eqref{eqn:quasi_optimality}. Concentration
in mmol/l. }
\label{fig:quasi_optimality}
\end{figure}

\begin{figure}
\begin{tabular}{ccc|ccc}
\multicolumn{3}{c|}{Non-whitened} & \multicolumn{3}{c}{Whitened}\\
\hline
SNR & rSVD1 & rSVD2 & SNR & rSVD1 & rSVD2 \\
\hline
\multicolumn{3}{l|}{$k=1000$} & \multicolumn{3}{l}{}\\
$\alpha_{i^\ast}=2.44 \PLH 10^{-2}$ & $\alpha_{i^\ast}=1.56 \PLH 10^{0}$ &$\alpha_{i^\ast}=25.0 \PLH 10^{0}$ & $\alpha_{i^\ast}=4.88 \PLH 10^{-2}$& $\alpha_{i^\ast}=1.56 \PLH 10^{0}$ & $\alpha_{i^\ast}= 25.0 \PLH 10^{0}$ \\
\includegraphics[width=0.13\textwidth]{figures/plots_open_MPI_no_snr_thr/shapePhantom/use_cpp0/dim1000/method12_alpha13.png}&
 \includegraphics[width=0.13\textwidth]{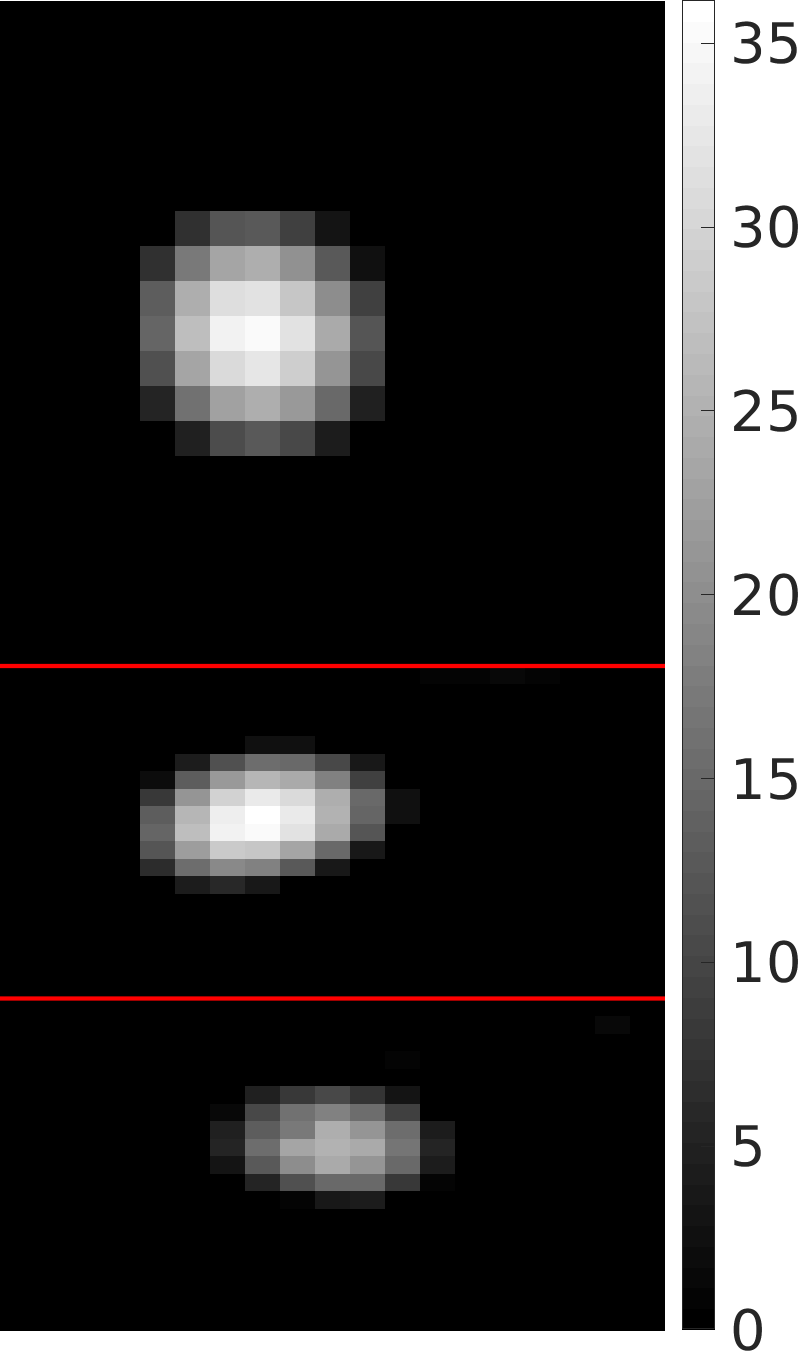}&
  \includegraphics[width=0.13\textwidth]{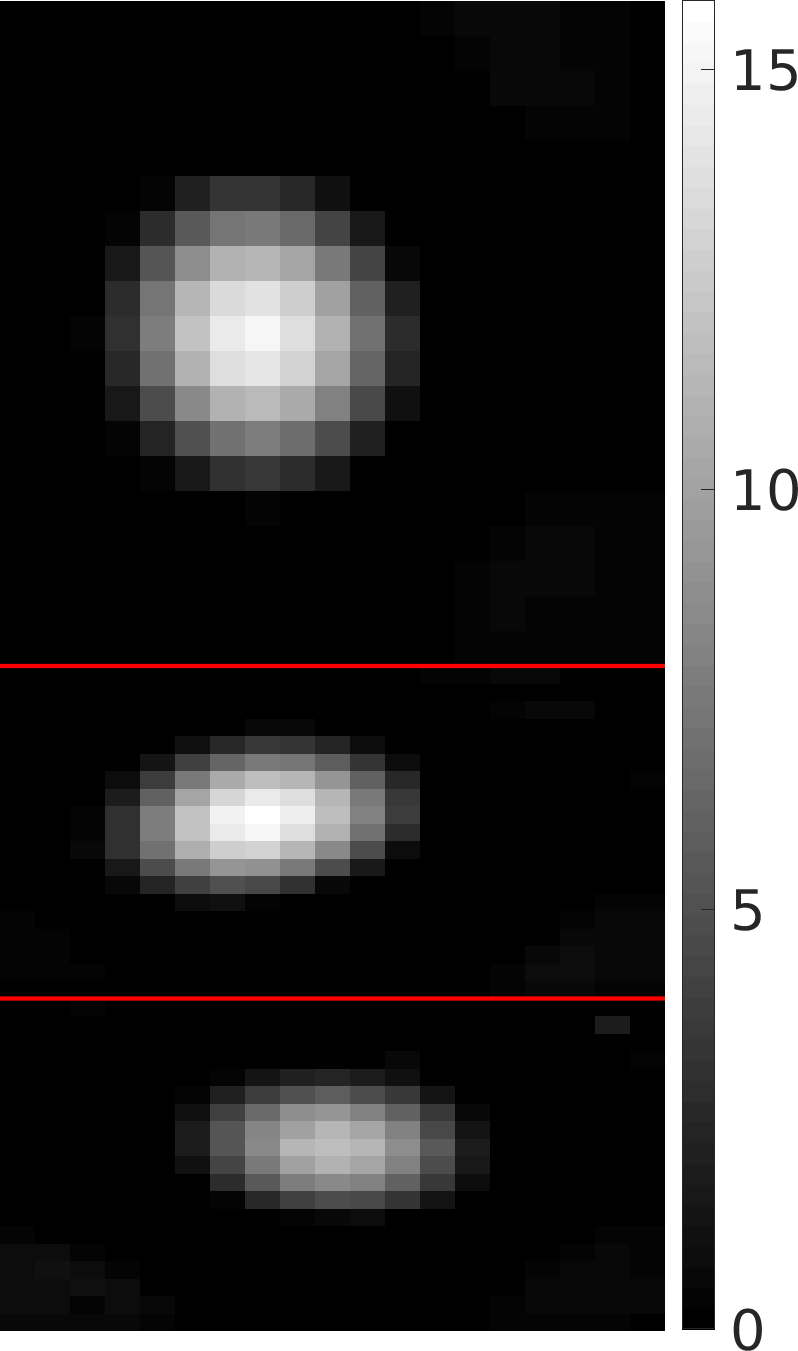}&
 \includegraphics[width=0.13\textwidth]{figures/plots_open_MPI_no_snr_thr/shapePhantom/use_cpp0/dim1000/method6_alpha12.png}&
 \includegraphics[width=0.13\textwidth]{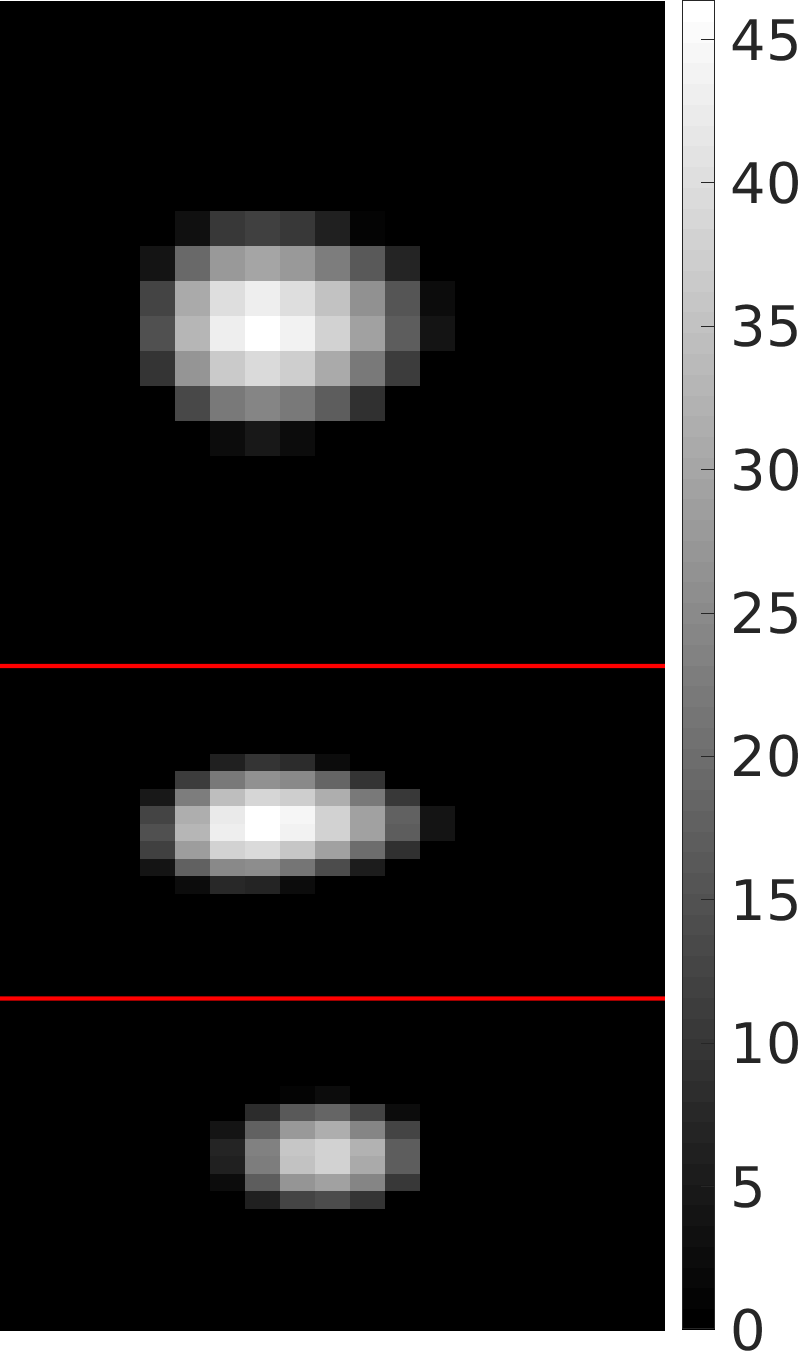}&
 \includegraphics[width=0.13\textwidth]{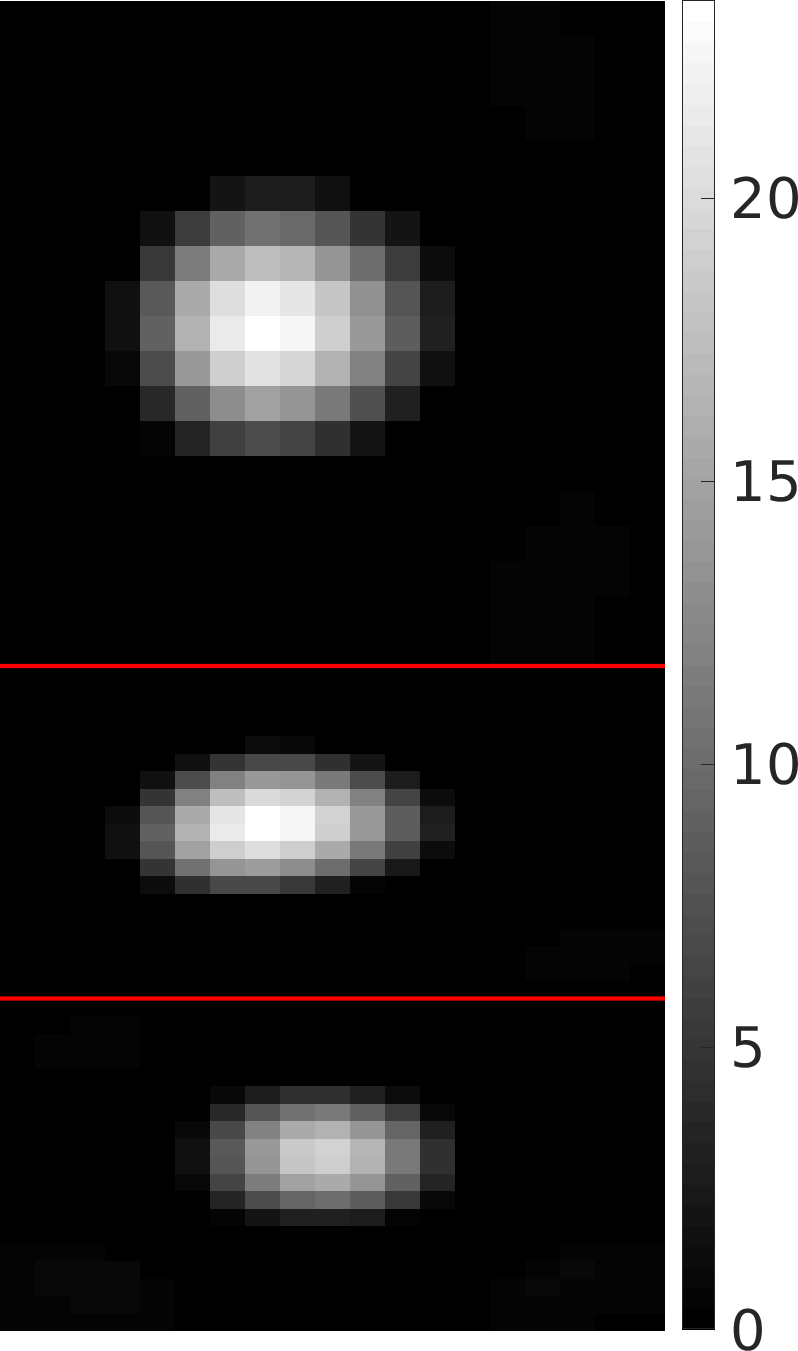}\\
\hline
\multicolumn{3}{l|}{$k=500$} & \multicolumn{3}{l}{}\\
$\alpha_{i^\ast}=4.88 \PLH 10^{-2}$ & $\alpha_{i^\ast}=1.56 \PLH 10^{0}$ &$\alpha_{i^\ast}=25.0 \PLH 10^{0}$ & $\alpha_{i^\ast}=9.77 \PLH 10^{-2}$& $\alpha_{i^\ast}=1.56 \PLH 10^{0}$ & $\alpha_{i^\ast}= 25.0 \PLH 10^{0}$ \\
 \includegraphics[width=0.13\textwidth]{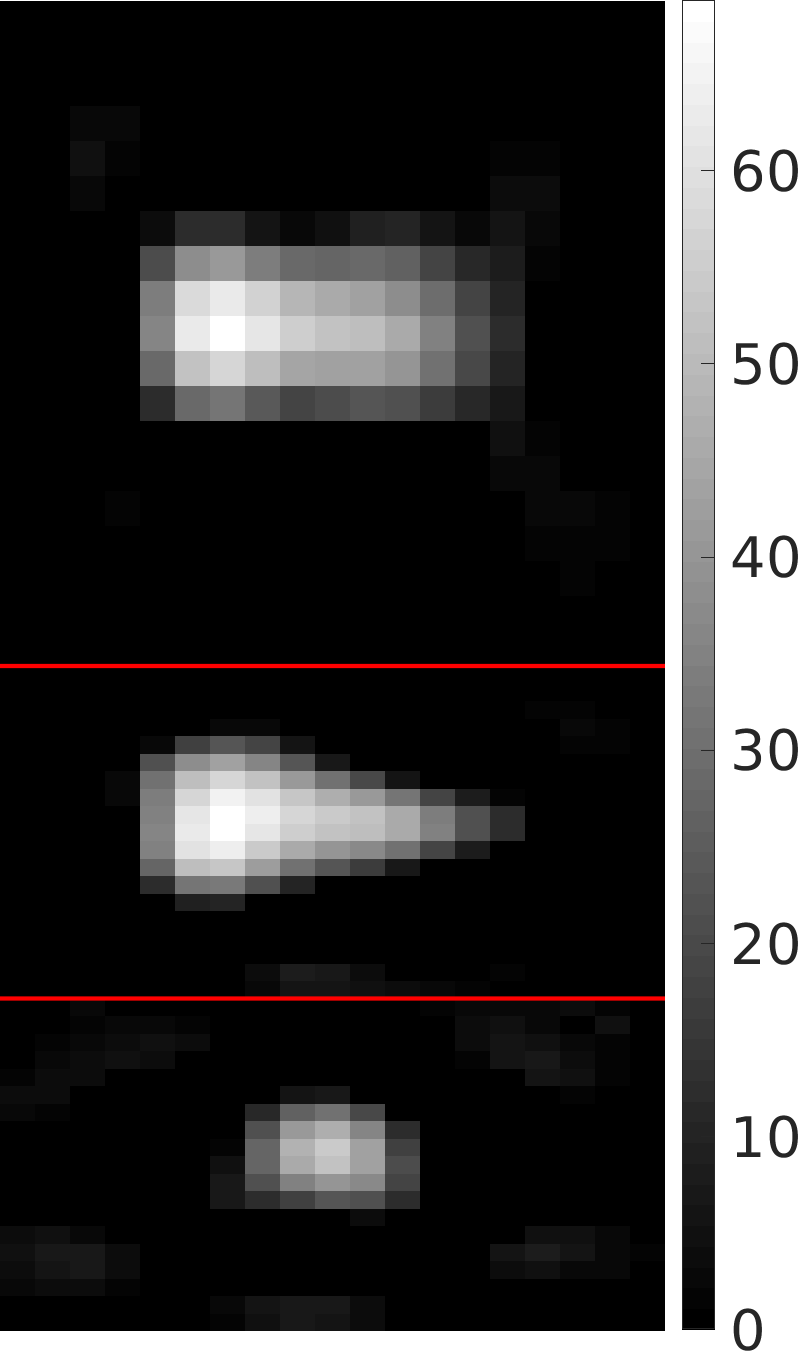}&
 \includegraphics[width=0.13\textwidth]{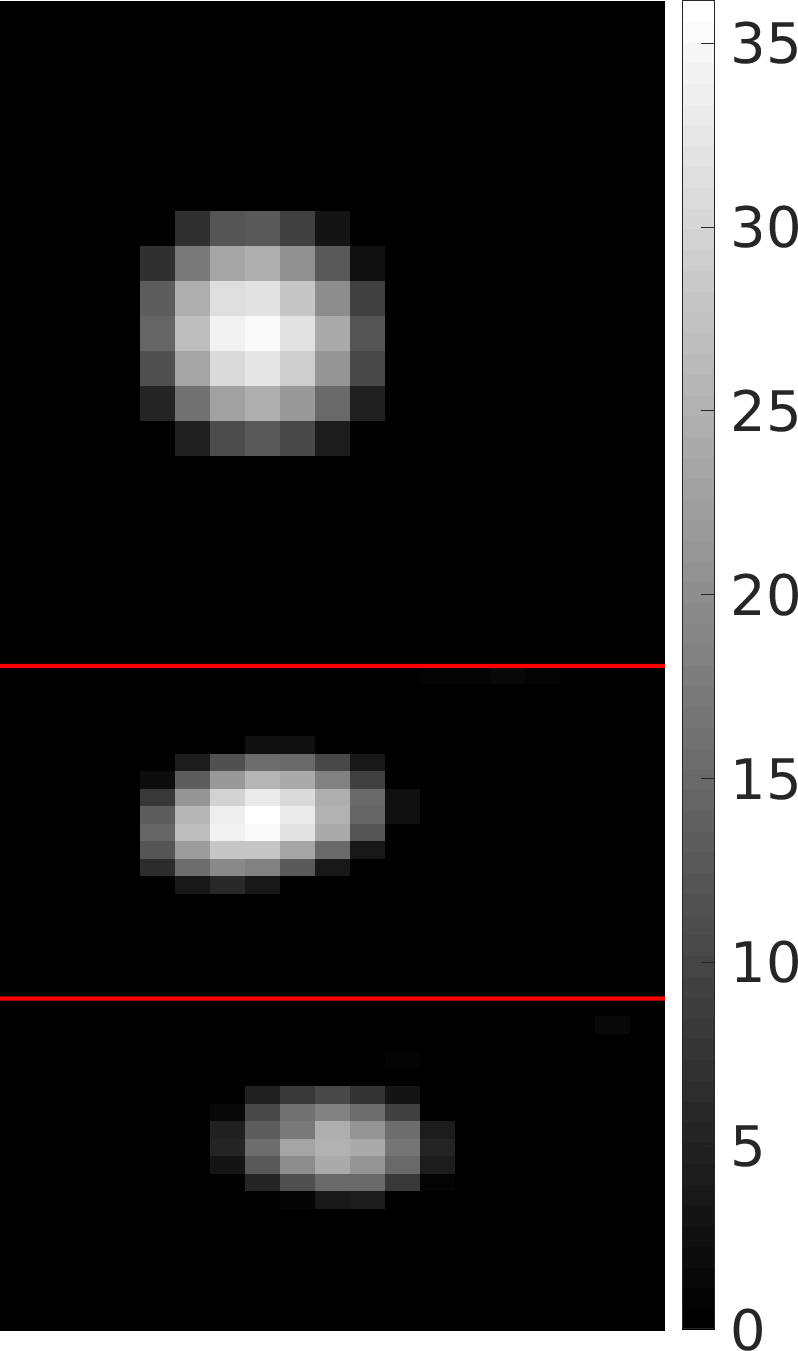}&
 \includegraphics[width=0.13\textwidth]{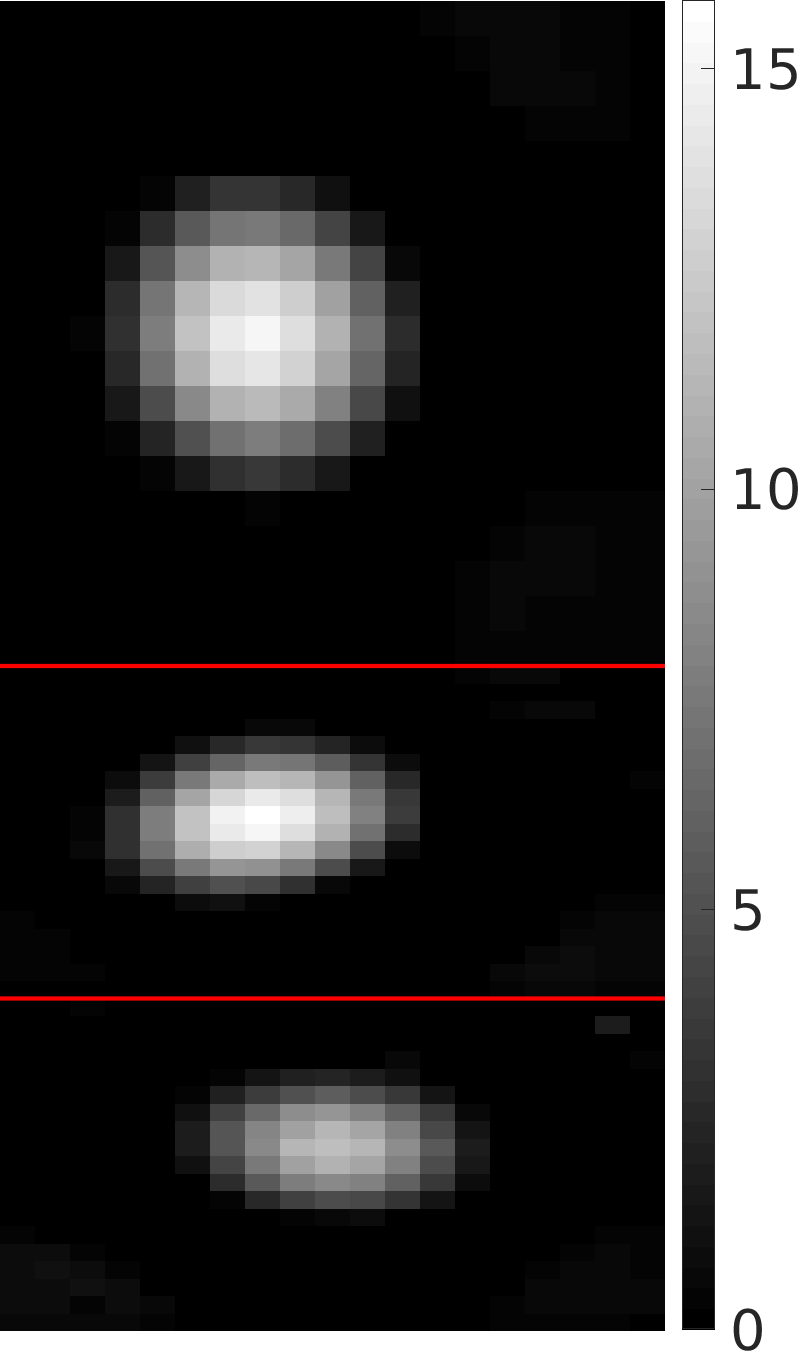}&
 \includegraphics[width=0.13\textwidth]{figures/plots_open_MPI_no_snr_thr/shapePhantom/use_cpp0/dim500/method6_alpha11.png}&
 \includegraphics[width=0.13\textwidth]{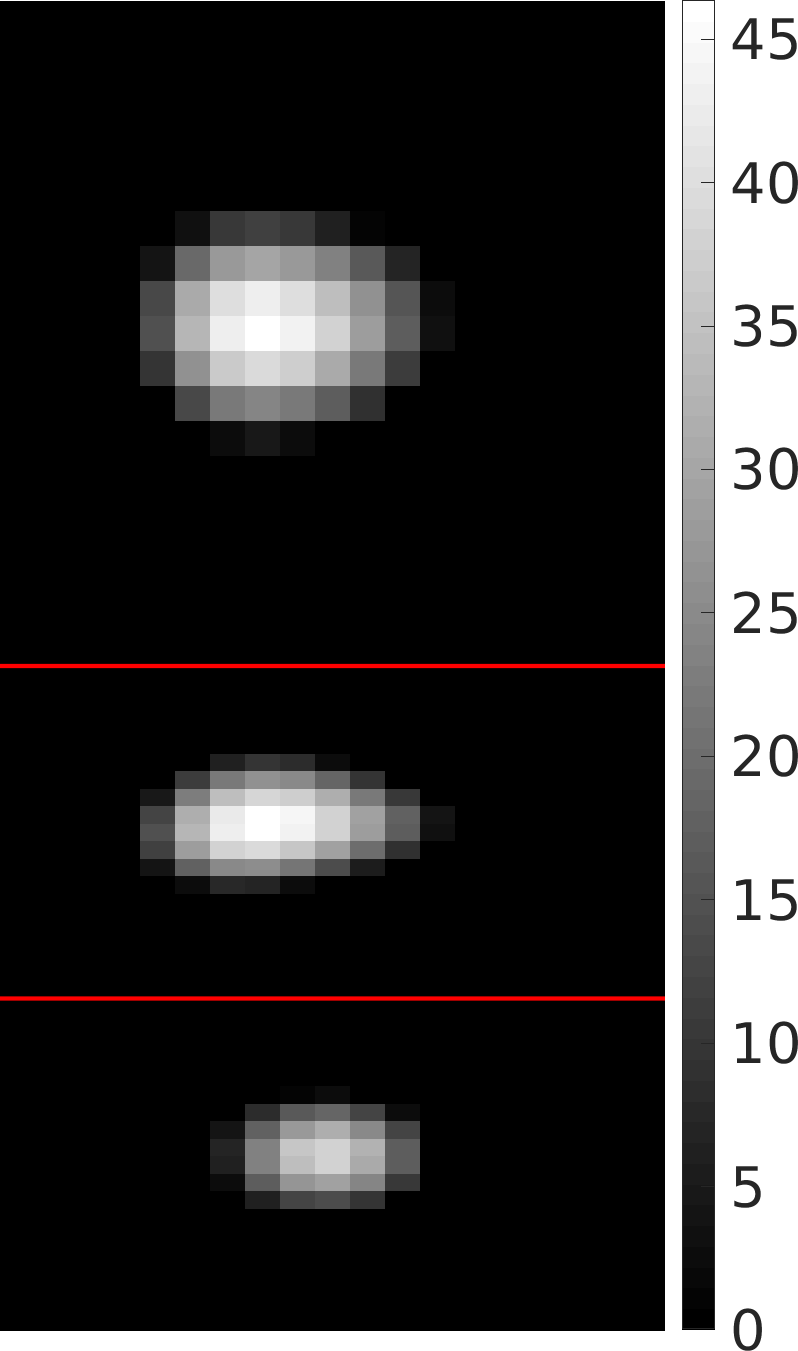}&
 \includegraphics[width=0.13\textwidth]{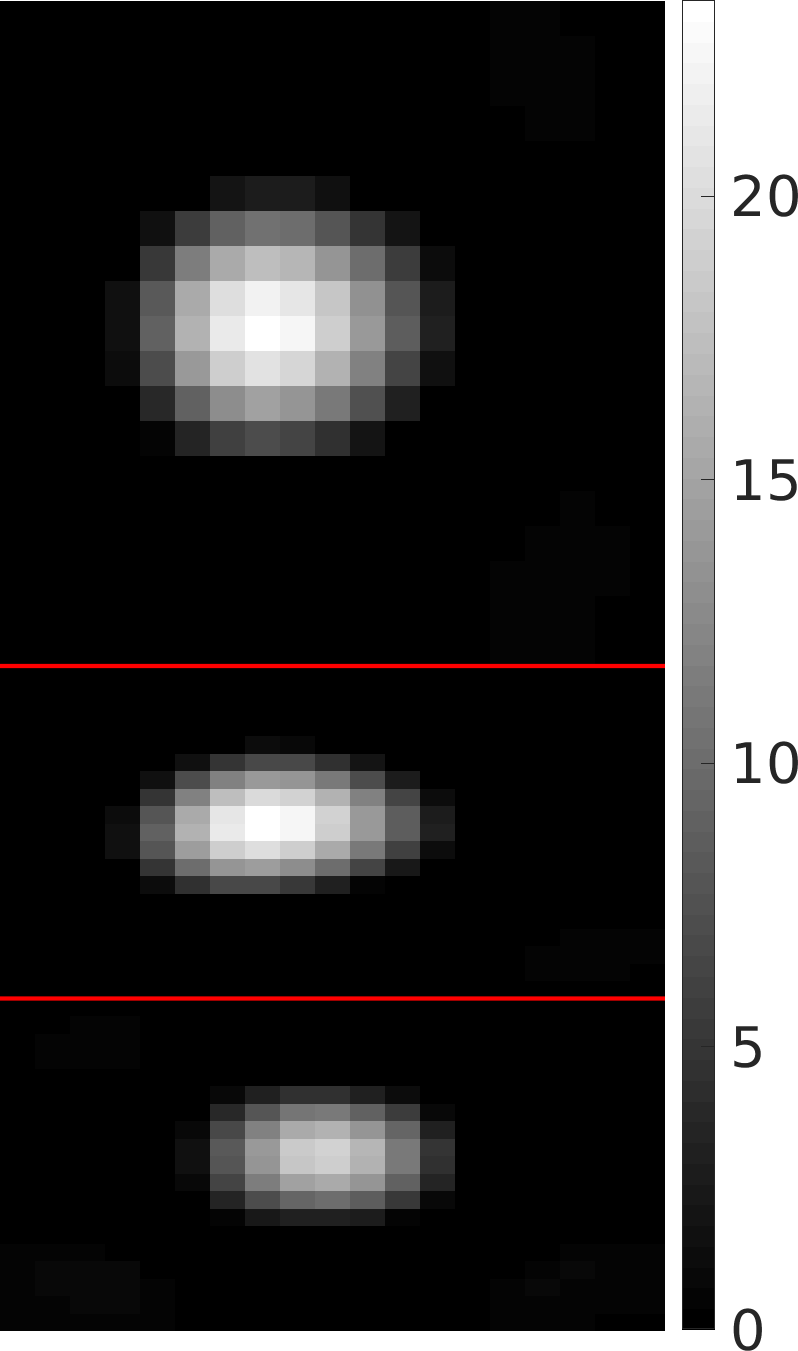}\\
\end{tabular}
\caption{$\alpha_{i^\ast}$ being the regularization parameter with minimal residual which is in line with the discrepancy
principle in \eqref{eqn:discrepancy} for an appropriate choice of $\tau\delta +\sigma \epsilon$. Concentration
in mmol/l.} 
\label{fig:discrepancy_principle}
\end{figure}

\section{Concluding remarks}\label{sec:conc}

In this work we have discussed several numerical issues in the MPI reconstruction from the perspective of modern
inverse theory. First, we propose to include a whitening strategy and to solve a generalized least squares problem
that is adapted to the noise statistics. This step can significantly improve the robustness of the algorithm and its
accuracy. Second, we propose a dimension reduction strategy in the data space via randomized
SVD. The randomized SVD is computationally efficient and scales to very large matrices. This step can greatly
reduce the number of equations, and arrive at an accurate reduced system. The numerical results show that by combining
whitening and low-rank approximation, one can obtain reconstructions of similar quality compared to the benchmark
approach, but at a much lower computational complexity, and meanwhile can improve the image quality when compared
with alternative system reduction approaches like
SNR. Third, we described two parameter choice rules, i.e., discrepancy principle and quasi-optimality criterion,
and numerically studied their performance. It is found that that the choice rules can benefit from whitening, and
the quasi-optimality criterion allows a robust parameter choice with good reconstruction quality. These experimental
findings indicate that the algorithmic techniques can facilitate developing fast robust MPI reconstruction
algorithms.

This study has implications on several issues in the context of MPI reconstruction. The low-rank approximation
provides an alternative (and complementary) to the sparse approximation approaches for the forward map
\cite{Lampe_Fast_2012,knopp2015local,Schmiester2017}. In contrast to these works, the dimension reduction based
on randomized SVD does not rely on an \textit{a priori} choice of a basis for system representation, and
it is optimal with respect to the (weighted) Frobenius / spectral norm. In theory, the ill-posed nature of the
MPI inverse problem \cite{KluthJinLi:2017} allows a memory-efficient representation (i.e., low-rank approximation
with a small $k$) without significant loss of reconstruction quality, which is also fully confirmed by the numerical
results in Section \ref{sec:numer}. The proposed method also does not require an SNR-type quality measure, which
is computed from the noisy measured system matrix data and empty scanner measurements \cite{franke2017}. The
latter is utilized in the proposed method (with a sufficiently large number of repetitions) to obtain a
reasonable approximation of the covariance matrix $C$ for the whitening step. In contrast to the SNR-type
quality measure, the noise characteristic is incorporated via the proposed whitening strategy. The simple
nature of the used background measurement correction does not require additional empty scanner measurements
during the calibration, which can potentially prolong the calibration due to expensive additional robot movements
\cite{weizenecker2009three,Them2016}. Finally, the optimal dimension reduction builds the basis for developing efficient
online reconstructions. In this context the proposed method is advantageous since it allows more robust and
faster (due to possibly much more effective dimension reduction) image reconstruction when compared with the
online reconstruction approach based on the SNR-type quality measure proposed in \cite{Knopp_Online_2016}.


\section*{Acknowledgements}
T. Kluth acknowledges funding by the Deutsche Forschungsgemeinschaft (DFG, German Research Foundation) - project number 281474342/GRK2224/1 ``Pi$^3$ : Parameter Identification - Analysis, Algorithms, Applications''.

\bibliographystyle{abbrv}
\bibliography{literature}

\appendix
\section{Error estimate}\label{app:error}
Now we give an error estimate on the approximation $\tilde x_\alpha^\delta$ in a general setting.
Let $X$ be a Banach space, and $Y$ be a Hilbert space, and $A:X\rightarrow Y$ be a compact
linear operator. Consider the following inverse problem $Ax = y^\dag,$
where $x\in X$. Instead of the exact data $y^\dag=Ax^\dag$, corresponding to the exact solution $x^\dag$, we have
$y^\delta\in Y$ with an accuracy $\delta = \|y^\delta-y^\dag\|$. Let $\tilde A:X\to Y$
be an approximate forward map with $\epsilon = \|\tilde A-A\|$. Then we aim at finding an approximate
solution $\tilde x_\alpha^\delta$ by means of variational regularization \cite{ItoJin:2015}
\begin{equation*}
  \tilde J_\alpha (x) = \tfrac12\|\tilde Ax-y^\delta\|^2 + \alpha\psi(x),
\end{equation*}
where the functional $\psi:X\to \mathbb{R}_+\cap\{0\}$ is a convex, proper and lower-semicontinuous functional. The
common choice includes $\psi(x)=\frac12\|x\|^2$, $\psi(x)=\|x\|_{\ell^1}$ and $\psi(x)=|x|_{\rm TV}$ etc.
We denote by $x_\alpha^\delta$ a minimizer to $J_\alpha$ and by $\tilde x_\alpha^\delta$
a corresponding minimizer to $\tilde J_\alpha$ with a noisy operator $\tilde A$. By $x^\dag$,
we denote a minimum-$\psi$ solution of the equation $Ax=y^\dag$: $x^\dag = \arg\min_{x\in X: Ax=y^\dag}
\psi(x)$. It is easy to see that the derivation below remains valid for in the presence of nonnegativity
constraint, so long as the minimum-$\psi$ solution is feasible.

Let $\partial \psi(x)$ be the subdifferential of $\psi$ at $x$ \cite{EkelandTemam:1999}. For any
$\xi\in \partial \psi(x)$, we define the Bregman distance from $x$ to $x'$ with respect to $\xi$ by
\begin{equation*}
  d_{\xi}(x',x) = \psi(x')-\psi(x)-\langle \xi,x'-x\rangle.
\end{equation*}

Then we have the following error estimate on the approximation $\tilde x_\alpha^\delta$. It may serve as a
guideline for determining the accuracy of the constructed approximation $\tilde A$: the model error
$\epsilon=\|A-\tilde A\|$ should be comparable with data error $\delta$ in order not to compromise
the reconstruction accuracy. The proof is standard \cite{SchusterKlatenbacher:2012,ItoJin:2015} and it is 
given only for completeness.
\begin{theorem}\label{thm:error}
Assume that the exact solution $x^\dag$ fulfills the following source condition: there exists $w\in Y$ such
that $A^*w \in\partial \psi(x^\dag)$. Then for the minimizer $\tilde x_\alpha^\delta$ to the functional $\tilde J_\alpha$, there holds
\begin{equation*}
 d_{\xi}(\tilde x_\alpha^\delta,x^\dag)
  \leq  \alpha^{-1}(\epsilon\|x^\dag\|+\delta)^2  + \alpha\|w\|^2 + \epsilon \|w\|\|x^\dag-\tilde x_\alpha^\delta\|.
\end{equation*}
\end{theorem}
\begin{proof}
By the minimizing property of $\tilde x_\alpha^\delta$, we obtain
\begin{equation*}
  \tfrac{1}{2}\|\tilde A\tilde x_\alpha^\delta-y^\delta\|^2 + \alpha
  d_{\xi}(\tilde x_\alpha^\delta,x^\dag) \leq \tfrac{1}{2}\|\tilde Ax^\dag-y^\delta\|^2-\alpha\langle \xi,\tilde x_\alpha^\delta-x^\dag\rangle.
\end{equation*}
Under the source condition, rearranging the inequality yields
\begin{align*}
\tfrac{1}{2}\|\tilde A\tilde x_\alpha^\delta-y^\delta\|^2 + \alpha d_{\xi}(\tilde x_\alpha^\delta,x^\dag) & \leq
\tfrac{1}{2}\|\tilde Ax^\dag-y^\delta\|^2 - \alpha\langle w,A(\tilde x_\alpha^\delta-x^\dag)\rangle.
\end{align*}
Next we rewrite the terms on the right hand side as
\begin{align*}
  \|\tilde Ax^\dag-y^\delta\|^2 
    & = \|\tilde A\tilde x_\alpha^\delta-y^\delta\|^2 + 2\langle \tilde A\tilde x_\alpha^\delta-y^\delta,\tilde A(x^\dag-\tilde x_\alpha^\delta)\rangle + \|\tilde A(x^
    \dag-\tilde x_\alpha^\delta)\|^2,\\
  \langle w,A(\tilde x_\alpha^\delta-x^\dag)\rangle & = \langle w,A_k(\tilde x_\alpha^\delta-x^\dag)\rangle +\langle w,(A-A_k)(\tilde x_\alpha^\delta-x^\dag)\rangle.
\end{align*}
Combining the last three estimates yields
\begin{align*}
   \tfrac{1}{2}\|\tilde A\tilde x_\alpha^\delta-y^\delta\|^2 + \alpha d_{\xi}(\tilde x_\alpha^\delta,x^\dag)
  \leq & \tfrac12\|\tilde A\tilde x_\alpha^\delta-y^\delta\|^2 + \langle\tilde A\tilde x_\alpha^\delta-y^\delta,\tilde A(x^\dagger-\tilde x_\alpha^\delta)\rangle + \tfrac{1}{2}\|\tilde A(x^\dagger-\tilde x_\alpha^\delta)\|^2\\
   &  +\alpha\langle w,\tilde A(x^\dag-\tilde x_\alpha^\delta)\rangle +\alpha\langle w,(A-\tilde A)(x^\dag-\tilde x_\alpha^\delta)\rangle\\
   =& \tfrac12\|\tilde A\tilde x_\alpha^\delta-y^\delta\|^2 + \langle\tilde Ax^\dag-y^\delta,\tilde A(x^\dagger-\tilde x_\alpha^\delta)\rangle -\tfrac{1}{2}\|\tilde A(x^\dagger-\tilde x_\alpha^\delta)\|^2\\
   &  +\alpha\langle w,\tilde A(x^\dagger-\tilde x_\alpha^\delta)\rangle +\alpha\langle w,(A-\tilde A)(x^\dag-\tilde x_\alpha^\delta)\rangle.
\end{align*}
Collecting the terms, we obtain
\begin{align*}
 \tfrac{1}{2}\|\tilde A(x^\dag-\tilde x_\alpha^\delta)\|^2 + \alpha d_{\xi}(\tilde x_\alpha^\delta,x^\dag)
  \leq  \langle\tilde Ax^\dag-y^\delta,\tilde A(x^\dagger-\tilde x_\alpha^\delta)\rangle +\alpha\langle w,\tilde A(x^\dagger-\tilde x_\alpha^\delta)\rangle +\alpha\langle w,(A-\tilde A)(x^\dag-\tilde x_\alpha^\delta)\rangle
\end{align*}
Thus, by means of Cauchy-Schwarz inequality and Young's inequality, 
\begin{equation*}
 \alpha d_{\xi}(\tilde x_\alpha^\delta,x^\dag)
  \leq  \|\tilde Ax^\dag-y^\delta\|^2 + \alpha^2\|w\|^2 + \alpha\|w\|\|A-\tilde A\| \|x^\dag-\tilde x_\alpha^\delta\|.
\end{equation*}
Meanwhile by the triangle inequality,
\begin{align*}
  \|\tilde Ax^\dag-y^\delta\| & \leq \|(\tilde A-A)x^\dag\| + \|Ax^\dag-y^\delta\| \\
   & \leq \|\tilde A-A\|\|x^\dag\| + \|y^\dagger-y^\delta\| \leq \epsilon \|x^\dag\| + \delta.
\end{align*}
Upon substituting the estimate, we obtain
\begin{equation*}
 d_{\xi}(\tilde x_\alpha^\delta,x^\dag)
  \leq  \alpha^{-1}(\epsilon\|x^\dag\|+\delta)^2  + \alpha\|w\|^2 + \epsilon \|w\|\|x^\dag-\tilde x_\alpha^\delta\|.
\end{equation*}
This completes the proof of the theorem.
\end{proof}
\begin{remark}
For the quadratic penalty $\psi(x)=\frac12\|x\|^2$, the associated Bregman distance $d_\xi(x',x)$ is
given by $d_\xi(x',x)=\frac12\|x'-x\|^2$, and thus the error estimate in Theorem \ref{thm:error} reduces to
\begin{equation*}
 \|\tilde x_\alpha^\delta-x^\dag\|^2
  \leq  2\alpha^{-1}(\epsilon\|x^\dag\|+\delta)^2  + 2\alpha\|w\|^2 + 2\epsilon \|w\|\|x^\dag-\tilde x_\alpha^\delta\|,
\end{equation*}
which together with Young's inequality yields
\begin{equation*}
  \|\tilde x_\alpha^\delta-x^\dag\|^2
  \leq  4\alpha^{-1}(\epsilon\|x^\dag\|+\delta)^2  + 4\alpha\|w\|^2 + 4\epsilon^2 \|w\|^2.
\end{equation*}
The estimate shows that roughly one should choose $\epsilon:=\|A-\tilde A\| $ such that $\epsilon \|x^\dag\|\approx\delta$,
in order to ensure that the overall accuracy is not compromised. This directly gives a guiding principle for
constructing the low-rank approximation $\tilde A$ in Section \ref{ssec:rsvd}. Note that the last term in the estimate
is generally of high order, and the first two terms essentially determines the accuracy.
\end{remark}
\end{document}